\documentclass[reqno, 10pt]{article}
\usepackage{fullpage}

\usepackage{amsthm,amsfonts,amstext,amssymb,mathrsfs,amsmath,latexsym,mathtools}
\allowdisplaybreaks
\usepackage{graphicx}
\usepackage{hyperref}
\usepackage{comment}
\usepackage{enumerate}
\usepackage{caption}
\usepackage{overpic}
\usepackage[super]{nth}
\usepackage{color}
\usepackage[colorinlistoftodos,prependcaption,textsize=small]{todonotes}
\usepackage{setspace}

\usepackage{seqsplit}
\usepackage{tikz}
\usetikzlibrary{patterns}
\usetikzlibrary{decorations.pathmorphing}
\usetikzlibrary{decorations.pathreplacing,decorations.markings}
\usetikzlibrary{shapes}

\makeatletter
\def\grd@save@target#1{%
  \def\grd@target{#1}}
\def\grd@save@start#1{%
  \def\grd@start{#1}}
\tikzset{
  grid with coordinates/.style={
    to path={%
      \pgfextra{%
        \edef\grd@@target{(\tikztotarget)}%
        \tikz@scan@one@point\grd@save@target\grd@@target\relax
        \edef\grd@@start{(\tikztostart)}%
        \tikz@scan@one@point\grd@save@start\grd@@start\relax
        \draw[minor help lines] (\tikztostart) grid (\tikztotarget);
        \draw[major help lines] (\tikztostart) grid (\tikztotarget);
        \grd@start
        \pgfmathsetmacro{\grd@xa}{\the\pgf@x/1cm}
        \pgfmathsetmacro{\grd@ya}{\the\pgf@y/1cm}
        \grd@target
        \pgfmathsetmacro{\grd@xb}{\the\pgf@x/1cm}
        \pgfmathsetmacro{\grd@yb}{\the\pgf@y/1cm}
        \pgfmathsetmacro{\grd@xc}{\grd@xa + \pgfkeysvalueof{/tikz/grid with coordinates/major step}}
        \pgfmathsetmacro{\grd@yc}{\grd@ya + \pgfkeysvalueof{/tikz/grid with coordinates/major step}}
        \foreach \x in {\grd@xa,\grd@xc,...,\grd@xb}
        \node[anchor=north] at (\x,\grd@ya) {\pgfmathprintnumber{\x}};
        \foreach \y in {\grd@ya,\grd@yc,...,\grd@yb}
        \node[anchor=east] at (\grd@xa,\y) {\pgfmathprintnumber{\y}};
      }
    }
  },
  minor help lines/.style={
    help lines,
    step=\pgfkeysvalueof{/tikz/grid with coordinates/minor step}
  },
  major help lines/.style={
    help lines,
    line width=\pgfkeysvalueof{/tikz/grid with coordinates/major line width},
    step=\pgfkeysvalueof{/tikz/grid with coordinates/major step}
  },
  grid with coordinates/.cd,
  minor step/.initial=.2,
  major step/.initial=1,
  major line width/.initial=0.25mm,
}

\pgfdeclarepatternformonly{north east lines wide}%
   {\pgfqpoint{-1pt}{-1pt}}%
   {\pgfqpoint{10pt}{10pt}}%
   {\pgfqpoint{5pt}{5pt}}%
   {
     \pgfsetlinewidth{.8pt}
     \pgfpathmoveto{\pgfqpoint{0pt}{0pt}}
     \pgfpathlineto{\pgfqpoint{9.1pt}{9.1pt}}
     \pgfusepath{stroke}
    }

\tikzset{
  on each segment/.style={
    decorate,
    decoration={
      show path construction,
      moveto code={},
      lineto code={
        \path [#1]
        (\tikzinputsegmentfirst) -- (\tikzinputsegmentlast);
      },
      curveto code={
        \path [#1] (\tikzinputsegmentfirst)
        .. controls
        (\tikzinputsegmentsupporta) and (\tikzinputsegmentsupportb)
        ..
        (\tikzinputsegmentlast);
      },
      closepath code={
        \path [#1]
        (\tikzinputsegmentfirst) -- (\tikzinputsegmentlast);
      },
    },
  },
  mid arrow/.style={postaction={decorate,decoration={
        markings,
        mark=at position .5 with {\arrow[#1]{stealth}}
      }}},
  rmid arrow/.style={postaction={decorate,decoration={
        markings,
        mark=at position .5 with {\arrowreversed[#1]{stealth}}
      }}},
  end arrow/.style={postaction={decorate,decoration={
        markings,
        mark=at position 1 with {\arrow[#1]{stealth}}
      }}},
  start arrow/.style={postaction={decorate,decoration={
        markings,
        mark=at position 0 with {\arrow[#1]{stealth}}
      }}},
   mid3 arrow/.style={postaction={decorate,decoration={
        markings,
        mark=at position .3 with {\arrow[#1]{stealth}}
      }}},
    rmid3 arrow/.style={postaction={decorate,decoration={
        markings,
        mark=at position .7 with {\arrowreversed[#1]{stealth}}
      }}},
      mid4 arrow/.style={postaction={decorate,decoration={
        markings,
        mark=at position .4 with {\arrow[#1]{stealth}}
      }}},
    rmid4 arrow/.style={postaction={decorate,decoration={
        markings,
        mark=at position .4 with {\arrowreversed[#1]{stealth}}
      }}},
    mid6 arrow/.style={postaction={decorate,decoration={
        markings,
        mark=at position .65 with {\arrow[#1]{stealth}}
      }}},
}

\makeatother
\tikzset{every state/.style={minimum size=0pt}}

\tikzset{
    mark position/.style args={#1(#2)}{
        postaction={
            decorate,
            decoration={
                markings,
                mark=at position #1 with \coordinate (#2);
            }
        }
    }
}

\usetikzlibrary{3d}
\usepackage{xcolor}
\usepackage{xifthen}

\usetikzlibrary{decorations}

\pgfdeclaredecoration{ignore}{final}
{
\state{final}{}
}

\pgfdeclaremetadecoration{middle}{initial}{
    \state{initial}[
        width={0pt},
        next state=middle
    ]
    {\decoration{moveto}}

    \state{middle}[
        width={\pgfdecorationsegmentlength*\pgfmetadecoratedpathlength},
        next state=final
    ]
    {\decoration{curveto}}

    \state{final}
    {\decoration{ignore}}
}

\tikzset{middle segment/.style={decoration={middle},decorate, segment length=#1}}

\newtheorem{thm}{Theorem}[section]
\newtheorem{prop}[thm]{Proposition}
\newtheorem{lem}[thm]{Lemma}
\newtheorem{cor}[thm]{Corollary}

\theoremstyle{definition}
\newtheorem{definition}[thm]{Definition}

\numberwithin{equation}{section}

\theoremstyle{remark}

\DeclareMathOperator{\ai}{Ai}
\DeclareMathOperator{\re}{Re}
\DeclareMathOperator{\im}{Im}

\DeclareMathOperator{\Tr}{Tr}

\DeclareMathOperator*{\res}{Res}

\DeclareMathOperator{\Li}{Li}

\newcommand{\Boh}{\mathcal{O}}

\newcommand{\mb}{\mathbb}
\newcommand{\bm}{\mathbf}

\newcommand{\R}{{\mathbb R}}
\newcommand{\C}{{\mathbb C}}

\newcommand{\Z}{{\mathbb Z}}

\newcommand{\K}{{\mathbb K}}
\newcommand{\I}{\mathbb I}
\newcommand{\D}{\mathbb D}

\renewcommand{\P}{\mathbb P}

\DeclarePairedDelimiter{\norm}{\lVert}{\rVert}

\newcommand{\URS}{\mathsf{U}}
\newcommand{\LRS}{\mathsf{L}}
\newcommand{\ff}{f}
\newcommand{\EE}{E}
\newcommand{\ffRS}{\mathsf{\ff}}
\newcommand{\EERS}{\mathsf{\EE}}
\newcommand{\ggRS}{\mathsf{g}}
\newcommand{\Fp}{F}
\newcommand{\Fk}{F_{\mathrm{KPZ}}}

\newcommand{\Kk}{\mb K_{\mathrm{KPZ}}}
\newcommand{\Qp}{Q_{\mathrm{per}}}
\newcommand{\Up}{U_{\mathrm{per}}}
\newcommand{\Pp}{P_{\mathrm{per}}}
\newcommand{\Rp}{R_{\mathrm{per}}}
\newcommand{\Qk}{Q_{\mathrm{KPZ}}}
\newcommand{\Uk}{U_{\mathrm{KPZ}}}
\newcommand{\Pk}{P_{\mathrm{KPZ}}}
\newcommand{\Rk}{R_{\mathrm{KPZ}}}

\newcommand{\Fic}{F_{\mathrm{ic}}}
\newcommand{\Kzic}{\mathbb K_z^{\mathrm{ic}}}
\newcommand{\bKzic}{\bm K_z^{\mathrm{ic}}}

\newcommand{\FF}{F_{\mathrm{1}}}
\newcommand{\FGOE}{F_{\mathrm{GOE}}}
\newcommand{\RS}{\mathcal{R}}
\newcommand{\rx}{\mathrm{x}}
\newcommand{\CB}{\mathcal{B}}

\newcommand{\CT}{\mathcal{T}}
\newcommand{\Tgamma}{T_\gamma}
\newcommand{\Tngamma}{T_{-\gamma}}
\newcommand{\Thalf}{T_{\mathrm{1/2}}}
\newcommand{\Tnhalf}{T_{\mathrm{-1/2}}}
\newcommand{\out}{\mathrm{out}}
\newcommand{\inn}{\mathrm{in}}
\newcommand{\RA}{\mathrm{A}}
\newcommand{\fe}{f}
\newcommand{\Clambda}{\Lambda}

\newcommand{\beq}{ \begin{equation} }
\newcommand{\eeq}{ \end{equation} }
\newcommand{\beqq}{ \begin{equation*} }
\newcommand{\eeqq}{ \end{equation*} }

\newcommand{\mJ}{\mb J}
\newcommand{\bJ}{\bm J}

\newcommand{\cb}{b}

\newcommand\restr[2]{{
  \left.\kern-\nulldelimiterspace 
  #1 
  \vphantom{\big|} 
  \right|_{#2} 
  }}

\newcommand{\cm}{C}

\title{Limiting one-point distribution of periodic TASEP}

\author{Jinho Baik\footnote{Department of Mathematics, University of Michigan,
Ann Arbor, MI 48109, USA,  
\texttt{baik@umich.edu}} \and Zhipeng Liu\footnote{Department of Mathematics, University of Kansas, Lawrence, KS 66045, USA, 
\texttt{zhipeng@ku.edu}} \and Guilherme L.~F.~Silva\footnote{Instituto de Ci\^encias Matem\'aticas e de Computa\c{c}\~ao, Universidade de S\~ao Paulo, S\~ao Carlos, SP, 13566-590, Brazil, \texttt{silvag@usp.br}}}

\date{\today}

\begin{document}

\maketitle

\begin{abstract}
The relaxation time limit of the one-point distribution of the spatially periodic totally asymmetric simple exclusion process is expected to be the universal one point distribution for the models in the KPZ universality class 
in a periodic domain. 
Unlike the infinite line case,  the limiting one point distribution depends non-trivially on the scaled time parameter.  
We study several properties of this distribution for the case of the periodic step and flat initial conditions.
We show that the distribution changes from a Tracy-Widom distribution in the small time limit to the Gaussian distribution in the large time limit, and also obtain right tail estimate for all time. 
Furthermore, we establish a connection to integrable differential equations such as the KP equation, coupled systems of mKdV  and nonlinear heat equations, and the KdV equation. 
\end{abstract}

\setcounter{tocdepth}{1}
\tableofcontents

\section{Introduction}

The scaled height field of the models in the KPZ universality class is expected to converge to a universal field in the large time limit. The limit, the KPZ fixed point, was constructed recently by Matetski, Quastel and Remenik, \cite{Matetski-Quastel-Remenik17} and also by Dauvergne, Ortmann and Virag \cite{Dauvergne-Ortmann-Virag18}. For the so-called step initial condition, the one-point marginal $F_{\mathrm{KPZ}}$ of the KPZ fixed point is given by (see, for example, \cite{Baik-Deift-Johansson99, Johansson00, Tracy-Widom09, Amir-Corwin-Quastel11, Borodin-Corwin-Ferrari14}) the scaled GUE Tracy-Widom distribution
\begin{equation}\label{eq:KPZ_GUE}
F_{\mathrm{KPZ}}(x;\tau,\gamma)=F_{\mathrm{GUE}}\left(\frac{x}{\tau^{1/3}}+ \frac{\gamma^2}{4\tau^{4/3}} \right)
\end{equation}
where $\tau, \gamma, x$ denote the time, location, and height parameters, respectively. 
From its expression it is apparent that $F_{\mathrm{KPZ}}$ is invariant under the KPZ re-scaling 
$(\tau, \gamma, x) \mapsto (\alpha^3\tau, \alpha^2\gamma, \alpha x)$ for all $\alpha>0$. 

In this paper, we study the analogue of the GUE Tracy-Widom distribution which arises when the infinite line is changed to a periodic domain.
The spatial periodicity introduces the following new feature. 
Let $L$ denote the period and $t$ the time. 
For the models in the KPZ universality class, the height functions at two locations at the same time are critically correlated if the distance is $O(t^{2/3})$. 
Hence, in the periodic case, all points are critically correlated if $t=O(L^{3/2})$, which is called the relaxation time scale. Since the periodicity effect should diminish when $t\ll L^{3/2}$ and amplify when $t\gg L^{3/2}$, the scaled time parameter $\tau=t L^{-3/2}$ measures the effect of the periodicity. 
As a result, the relaxation time limit of the one-point distribution depends on $\tau$ non-trivially and it is not invariant under the KPZ re-scaling.

The relaxation time limit (for both one-point and multi-point distributions) was evaluated\footnote{See \cite{Gwa-SpohnBethe, Derrida-Lebowitz98, Priezzhev2003, Golinelli-Mallick04, Golinelli-Mallick05, Brankov-Papoyan-Poghosyan-Priezzhev06, Prolhac14} for other properties of periodic models in the KPZ universality class.} for the periodic totally asymmetric simple exclusion process (TASEP) 
in \cite{Prolhac16, Baik-Liu16, Liu16, Baik-Liu19, Baik-Liu19b}. 
Assume the periodic step initial condition; this means that for the fixed density of particles $\rho\in (0,1)$, we impose at time $0$ that consecutive $\rho L$ are sites occupied, the next $(1-\rho)L$ unoccupied sites, and this pattern repeats, see Figure \ref{fig:density}. 
The one-point distribution converges, as $t=\Boh(L^{3/2})\to \infty$, to a distribution $\Fp$ independent of $\rho$. 
This limit was computed explicitly in a physics work \cite{Prolhac16} and a rigorous work \cite{Baik-Liu16}, and it takes the form\footnote{We follow the formula of \cite{Baik-Liu16}. The formula obtained in \cite{Prolhac16} has a different form but recently \cite{Prolhac19} verified that the two formulas are equivalent.}
\beq \label{eq:Fdeffirst}
	\Fp(x;\tau,\gamma)=\oint e^{xA_1(z)+\tau A_2(z)+2B(z)} \det(\mb I-\mb K_z)\frac{d z}{2\pi i z}.
\eeq 
The contour is a small circle enclosing the origin in the complex plane. 
The functions $A_1, A_2, B$ are given in terms of polylog functions, and they are independent of the physical parameters $\tau,\gamma$ and $x$. 
It was shown that $\Fp$ is a distribution function. 
Since it is a limit for a model in a periodic domain, the distribution function $\Fp(x;\tau,\gamma)$ is periodic in $\gamma$, and its period is normalized so that $F(x; \tau, \gamma+1)=F(x; \tau, \gamma)$. 
The function $\Fp$ is expected to be the universal relaxation time limit of the one-point distribution of the KPZ universality in a periodic domain for the periodic step initial condition. 
In other words, it is expected to be the marginal of the periodic KPZ fixed point. 
The goal of this paper is to study several properties of $F(x;\tau,\gamma)$. 

\begin{figure}
    \centering
    \begin{minipage}{0.4\textwidth}\centering
    \includegraphics[scale=.6]{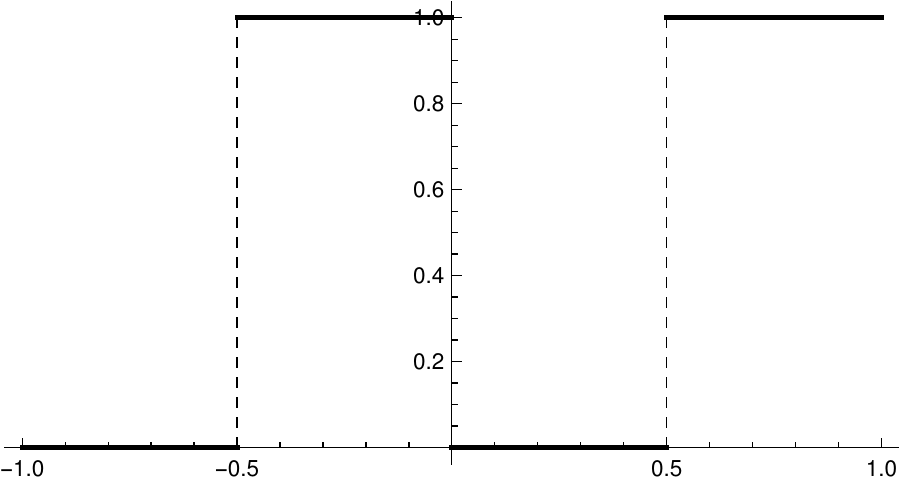}
    \end{minipage}
    \begin{minipage}{0.4\textwidth}\centering
    \includegraphics[scale=.6]{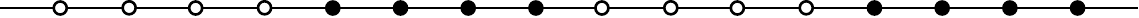}
    \end{minipage}
    \caption{The density profile and particle configuration for periodic step initial condition when $\rho=1/2$.}
    \label{fig:density}
\end{figure}

\subsection{Analytic structure of $\Fp$ and $\Fk$} 
 
The invariance of $\Fk$ through the KPZ re-scaling tells us that $\Fk$ can be recast as a function of one single variable, namely \eqref{eq:KPZ_GUE}, and all one needs to know about $\Fk$ is encoded in $F_{\mathrm{GUE}}$. In contrast, its periodic counterpart $\Fp$ does not enjoy the same invariance and truly depends on the three parameters $(\tau,\gamma,x)$.

For a better comparison between $\Fp$ and $\Fk$, we unwrap the dependence of $\Fk$ on the parameters $(\tau,\gamma,x)$ from $F_{\mathrm{GUE}}$. Using the representation \eqref{eq:KPZ_GUE} the distribution $\Fk$ takes the form 
\begin{equation}\label{eq:kpAgamma}
	\Fk(x;\tau,\gamma)=\det(\mb I-\Kk)_{L^2(0,\infty)} =\det(\mb I-\mb A_{-\gamma}\mb A_\gamma)_{L^2(0,\infty)}
\end{equation}
where $\Kk=\mb A_{-\gamma}\mb A_\gamma$, with $\mb A_\gamma$ being the operator acting on $L^2(0,\infty)$ with 
the kernel $\bm A_{\gamma}$,
\begin{equation*}
	\bm A_{\gamma}(s,t)=\mathcal A_\gamma(s+x+t)
\end{equation*}
and
\begin{equation}\label{eq:KPZ_shiftedairykernel}
	\mathcal A_\gamma(s)=\mathcal A_\gamma(s;\tau):=\frac{e^{\frac{\gamma^3}{12\tau}+\frac{\gamma s}{2\tau}}}{\tau^{1/3}}\ai\left(\frac{s}{\tau^{1/3}}+\frac{\gamma^2}{4\tau^{4/3}}\right)
	=\frac{1}{2\pi i}\int_{\Lambda_-}e^{-\frac{\tau}{3}u^3+\frac{\gamma}{2}u^2+su}du.
\end{equation}
Here $\ai$ denotes the Airy function.
Note that the function $\mathcal{A}_\gamma$ satisfies the KPZ scaling invariance: $\mathcal A_\gamma(s;\tau) =\alpha^{-1}\mathcal A_{\alpha^2\gamma}(\alpha s;\alpha^3\tau)$. The kernel for the product operator $\Kk$ is
\begin{equation*}
\bm K_{\mathrm{KPZ}}(s,t)= \int_0^\infty \mathcal A_{-\gamma}(s+x+u)\mathcal A_{\gamma}(u+x+t)du,\quad s,t\geq 0.
\end{equation*}
In \eqref{eq:KPZ_shiftedairykernel}, $\Lambda_-$ is any unbounded oriented contour from $e^{-i \theta_1}\infty$ to $e^{i \theta_2}\infty$ for some $\theta_1, \theta_2\in (\pi/2, 5\pi/6)$; see Figure \ref{fig:LambdaandS}.

Let us now turn to a representation of $\Fp$. The definition \eqref{eq:Fdeffirst} obtained in \cite{Baik-Liu16} involves the operator $\mb K_z$ which acts on the discrete space $\ell^2(\mathcal S_-)$, with (see Figure \ref{fig:LambdaandS})\footnote{The elements of the larger set $\mathcal S(z):=\{w\in \C\mid e^{-w^2/2}=z\}$ that contains $\mathcal S_-$ are called Bethe roots since they arise in the Bethe ansatz analysis for the periodic TASEP.}

\begin{equation}\label{def:bethe_root_set}
	\mathcal S_-=\mathcal S_-(z):=\{w\in \C\mid e^{-w^2/2}=z, \; \re w<0\}=\{w= -\sqrt{-2\log z +4\pi i k}, \; k\in \Z\}.
\end{equation}
This set is a discrete subset of the hyperbola 
\beqq
	\Lambda_-=\Lambda_-(z) =\{u\in \C \mid |e^{-w^2/2}|=|z|, \; \re w<0 \}
	=\{u\in \C\mid \re(w^2)=-2\log |z|, \; \re w<0 \}
\eeqq
which we could use for \eqref{eq:KPZ_shiftedairykernel}. Our first result is the following. 

\begin{figure}\centering
\begin{minipage}{.4\textwidth}
\centering
\begin{tikzpicture}[scale=0.5]
\begin{scope}
	\clip (0, 0) circle (4.1);
	\path[fill=gray!20] (-0.1,0) -- (-0.1,5) -- (-6.92,4) -- cycle;
	\path[fill=gray!20] (-0.1,0) -- (-0.1,-5) -- (-6.92,-4) -- cycle;
\end{scope}

\draw [line width=0.4mm,gray!50] (-4.3,0)--(4.3,0);
\draw [line width=0.4mm,gray!50] (0,-4.3)--(0,4.3);

\draw[domain=-2.8:2.8,smooth,variable=\y,black,thick]  plot ({-(\y*\y+1)^(1/2)},{\y});
\end{tikzpicture}
\end{minipage}%
\begin{minipage}{.4\textwidth}
\centering
\begin{tikzpicture}[scale=0.5]

\draw [line width=0.3mm,lightgray] (-4.3,0)--(4.3,0);
\draw [line width=0.3mm,lightgray] (0,-4.3)--(0,4.3);
\draw[thick,domain=-3:3,smooth,variable=\y,lightgray,thick]  plot ({-(\y*\y+1)^(1/2)},{\y});
\fill (-1.06573,-0.368479) circle[radius=2.5pt] node [above,shift={(0pt,0pt)}] {};
\fill (-2.0173,-1.75199) circle[radius=2.5pt] node [above,shift={(0pt,0pt)}] {};
\fill (-2.68227,-2.48889) circle[radius=2.5pt] node [above,shift={(0pt,0pt)}] {};
\fill (-2.53213,2.3263) circle[radius=2.5pt] node [above,shift={(0pt,0pt)}] {};
\fill (-1.81494,1.51459) circle[radius=2.5pt] node [above,shift={(0pt,0pt)}] {};
\end{tikzpicture}
\end{minipage}
\caption{The left picture is an example of the contour $\Lambda_-$. The right picture is the discrete set $\mathcal S_-$.}
\label{fig:LambdaandS}
\end{figure}
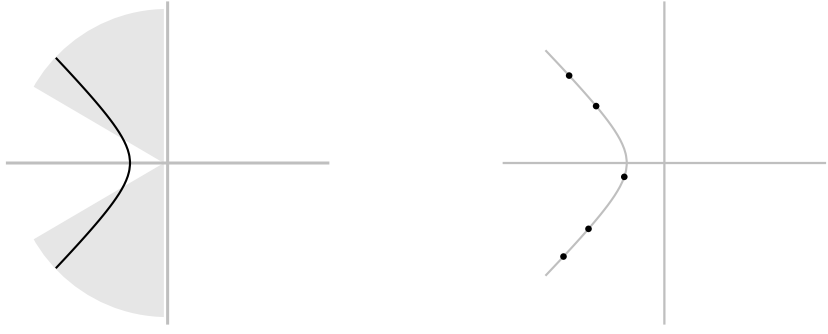	
	
\begin{lem}\label{lem:kernelK_kernelT}
The identity
$$
\det(\mb I-\mb K_z)_{\ell^2(\mathcal S_-)}=\det(\mb I-\mb T_{-\gamma}\mb T_{\gamma})_{L^2(0,\infty)}
$$
holds true, where $\mb T_{\gamma}:L^2(0,\infty)\to L^2(0,\infty)$ is the trace class operator defined by its kernel
\beq\label{eq:defTgh}
	\bm T_\gamma(s,t):=\mathcal T_\gamma(s+x+t),\quad \mathcal T_\gamma(s)
	=\mathcal T_{\gamma}(s;\tau,z):=\sum_{\xi\in \mathcal S_-}\frac{e^{-\frac{\tau}{3}\xi^3+\frac{\gamma}{2}\xi^2+s\xi-Q(\xi)}}{-\xi},
\eeq
where $Q(\xi)$ is a polylog integral whose formula is given in \eqref{def:Q_function}.
\end{lem}

We emphasize that $Q(\xi)$ does not depend on any of the parameters $\tau,\gamma,x$ and $z$. Its exact expression is not relevant at the moment. We also note that $\mathcal T_\gamma$ depends on $z$ through the set $\mathcal S_-=\mathcal S_-(z)$. Using the above lemma, \eqref{eq:Fdeffirst} updates to
\begin{equation}\label{eq:FdefTgamma}
\Fp(x;\tau,\gamma)=\oint e^{xA_1(z)+\tau A_2(z)+2B(z)} \det(\mb I-\mb T_{-\gamma}\mb T_{\gamma})_{L^2(0,\infty)} \frac{d z}{2\pi i z}
\end{equation}
where the kernel for the product $\mb T_{-\gamma}\mb T_{\gamma}$ is
\begin{equation}\label{eq:pKPZ_prodkernel}
(s,t)\mapsto \int_0^\infty \mathcal T_{-\gamma}(s+x+u)\mathcal T_{\gamma}(u+x+t)du,\quad s,t\geq 0.
\end{equation}

If we were to neglect the term $Q$ in \eqref{eq:defTgh}, the formula of $\mathcal T_{\gamma}(s)$ would be
\begin{equation}\label{eq:kernelT_noQ}
	\sum_{\xi\in \mathcal S_-} \frac{e^{-\frac{\tau}{3}\xi^3+\frac{\gamma}{2}\xi^2+s\xi}}{-\xi}.
\end{equation}
The set $\mathcal S_-$ is a discrete subset of $\Lambda_-$. The spacing between two neighboring points of $\mathcal S_-$ is approximately $\frac{2\pi i}{-\xi}$. Hence, \eqref{eq:kernelT_noQ} is a discretization of the rescaled Airy function $\mathcal A_\gamma$ in \eqref{eq:KPZ_shiftedairykernel}, and this discretization depends on the variable $z$ because $\mathcal S_-= \mathcal S_-(z)$. Consequently, the formula \eqref{eq:FdefTgamma} of $\Fp$ is a linear superposition of modified (due to $Q$) discrete versions of $\Fk(x;\tau,\gamma)$. 

The formula \eqref{eq:kernelT_noQ} is a sum over a discrete set, so it is not invariant under the KPZ rescaling. However, the discrete set explains the spatial periodicity: observing that $e^{-\xi^2/2}=z$ for $\xi \in \mathcal S_-$ the identities 
\begin{equation*}
\mathcal T_{\gamma+1}(s)=z^{-1}\mathcal T_{\gamma}(s) \quad \text{and}\quad \mathcal T_{-\gamma-1}(s)=z\mathcal T_{\gamma}(s),
\end{equation*}
follow, and hence the product kernel \eqref{eq:pKPZ_prodkernel} is invariant under $\gamma\mapsto \gamma+1$. 

\subsection{Integrable differential equations} \label{sec:intropde}

In the previous subsection, we discussed Airy-like operator formulas for $\det(\mb I-\mb K_z)$ and $\det(\mb I-\Kk)$.
In Lemmas~\ref{lem:integrability_operator_H} and \ref{lem:integrability_operator_F} we will show that both determinants also admit a representation in terms of so-called IIKS-integrable operators, and these in turn can be canonically associated to Riemann-Hilbert problems. Exploring the connection of Riemann-Hilbert problems with integrable differential equations we obtain the next result.

\begin{thm}\label{thm:integrable_systems}
The second log derivatives of the Fredholm determinants of $\mb K_z$ and $\Kk$ decompose as
$$
\partial_{xx} \log \det(\mb I-\mb \Kk)=\Pk(\tau,\gamma,x)\Rk(\tau,\gamma,x) \quad \text{and}\quad \partial_{xx} \log \det(\mb I-\mb K_z)=\Pp(\tau,\gamma,x)\Rp(\tau,\gamma,x),
$$
where $(p,r)=(\Pk,\Rk)$ or $(\Pp,\Rp)$ is a solution to the following two systems of differential equations.
\begin{enumerate}[(a)]
\item As a function of $\tau$ and $x$, they satisfy a system of coupled modified Korteweg--de Vries (mKdV) equations\footnote{The scale $\tau\mapsto 3\tau$ changes the equations to more standard form $p_{\tau}+ p_{xxx} +  6 prp_x=0$ and $r_{\tau} + r_{xxx} + 6 pr r_x =0$.}
\begin{equation} \label{eq:coupledMKdV}
\begin{cases}
	&3p_{\tau}+ p_{xxx} +  6 prp_x=0 , \\
	&3 r_{\tau} + r_{xxx} + 6 pr r_x =0 . 
\end{cases} 
\end{equation}
\item As a function of $\gamma$ and $x$, they satisfy a  system of coupled nonlinear heat equations
\begin{equation}  \label{eq:coupledNLS}
\begin{cases}
	&p_\gamma = - p_{xx} - 2p^2r  , \\
	&r_\gamma =  r_{xx} + 2pr^2 . 
\end{cases} 
\end{equation}
\end{enumerate}
Finally, the symmetry property
\begin{equation}\label{eq:prgammaminus}
    p(\gamma)=-r(-\gamma)
\end{equation}
also holds.
\end{thm}

The coupled systems and the symmetry \eqref{eq:prgammaminus} imply that $p$
itself solves non-local differential equations
\begin{equation}\label{eq:nonlocal_1}
	3p_\tau(\gamma)+p_{xxx}(\gamma)-6p(\gamma)p(-\gamma)p_x(\gamma)=0
\end{equation}
and
\begin{equation}\label{eq:nonlocal_2}
	p_\gamma(\gamma)+p_{xx}(\gamma)-2p(\gamma)^2p(-\gamma)=0.
\end{equation}
For the particular case when $\gamma=0$, which will be discussed in Theorem~\ref{thm:intflat} for the periodic TASEP with flat initial condition, Equation \eqref{eq:nonlocal_1} becomes the defocusing mKdV equation\footnote{The scale $\tau\mapsto 3\tau$ transform the equation to the standard form $p_\tau +p_{xxx}-6p^2p_x=0$. If the sign of the nonlinear term changes, $p_\tau +p_{xxx} + 6p^2p_x=0$, the equation becomes the focusing mKdV equation instead of the defocusing mKdV equation.} $3p_\tau +p_{xxx}-6p^2p_x=0$.

It is simple to show (see Section \ref{sec:kp_equation}) that the compatibility of the two systems \eqref{eq:coupledMKdV}--\eqref{eq:coupledNLS} yield
\begin{cor}[KP equation]\label{thm:integrable_structure}
The second log derivative of the Fredholm determinants 
$$
	\Up(\tau, \gamma, x)=\partial_{xx}\log \det(\mb I-\mb K_z), \quad \text\quad \Uk(\tau, \gamma, x)=\partial_{xx}\log \det(\mb I-\Kk)
$$
are solutions to the second Kadomtsev-Petviashvili (KP-II) equation\footnote{The scaled function $v(\tau, \gamma, x)=2 u(12\tau, \pm 2\gamma, x)$ changes the equation to more standard form $3 \lambda v_{\gamma\gamma}+(v_{\tau}+ 6vv_x + v_{xxx})_x$ with $\lambda=1$.  The equation with $\lambda=-1$ is called the KP-I equation.}
\begin{equation}\label{kp_equation}
	12u_{\gamma\gamma}+\left(12u_{\tau}+12uu_x+u_{xxx} \right)_x=0.
\end{equation}
\end{cor}

For $u=\Uk$, Corollary~\ref{thm:integrable_structure} was obtained recently by Quastel and Remenik \cite{Quastel-Remenik19b}. The relation \eqref{eq:KPZ_GUE} tells us that $\Uk$ is a self-similar solution to \eqref{kp_equation},
\begin{equation*}
\Uk(\tau,\gamma,x)=\tau^{-3/2}\phi_0\left(\frac{x}{\tau^{1/3}}+\frac{\gamma^2}{4\tau^{4/3}}\right),
\end{equation*}
for some function $\phi_0$. Equation~\eqref{kp_equation}, in turn, gives us that $\phi_0=-\phi^2$, where $\phi$ solves the second Painlev\'e equation, from which one can recover the Painlev\'e formula of $F_{\mathrm{GUE}}$ \cite{Tracy-Widom94}.

The above result is new for $u=\Up$. In this case, the results should be interpreted as that the equations hold for $(\tau, \gamma, x)$ and $z$ in a neighborhood of $(\tau_0,\gamma_0,x_0)$ and $z_0$ in which $\det (\I -\mb K_{z}) \neq 0$. 
The complement of the zero set of this Fredholm determinant is an open set of full measure since the Fredholm determinant is an analytic function of these four variables. 
Assuming that the solution exists for all $x$ and decays to $0$ sufficiently fast as $x\to\infty$, the function $F$ can be written as 
\beqq
	F(x;\tau,\gamma)=\oint e^{xA_1(z)+\tau A_2(z)+2B(z)} E(x, \tau, \gamma; z)\frac{d z}{2\pi i z}, 
	\quad E(\tau, \gamma,x; z):= \exp\left(\int_x^\infty (y-x) \Up(\tau, \gamma,y; z) d y\right)
\eeqq
where $\Up(\tau, \gamma, x; z)$ is a \emph{complex} solution of the KP equation with the initial condition determined by $z$. 
This is a periodic analogue of the Painlev\'e formula of $F_{\mathrm{GUE}}$. 

The solutions $\Up,\Pp$ and $\Rp$ above are infinite dimensional solitons (with complex velocities) since the associated Riemann-Hilbert problem is discrete; see Section \ref{sec:integrable_structure}. This way, $F$ is a superposition of (the tau function of the) solutions integrated over a parameter $z$ which determines the initial/boundary condition of the solution. In contrast, the purely continuous nature of the associated Riemann-Hilbert problem for $\Uk,\Pk,\Rk$ indicate that these solutions are not solitons. 

The KP equation was first introduced by Kadomtsev and Petviashvili \cite{KP-1970} in 1970 and it is an universal model for the study of two-dimensional shallow water waves that generalize the mKdV equation to two spatial dimensions. The connection between the KP equation and the KPZ universality class was first observed by Quastel and Remenik in \cite{Quastel-Remenik19b}. They considered the one-time/multi-location marginals of the KPZ fixed point on the infinite line with general initial condition, 
and proved that a log derivative of the distribution can be expressed in terms of a matrix KP equation. 
This result was used to find large deviation results and also further extended in \cite{LeDoussal-2020}.  
As mentioned before, if we consider the one-point function and step initial condition, Quastel-Remenik's result is Corollary~\ref{thm:integrable_structure} for $u=\Uk$.\footnote{The paper \cite{Quastel-Remenik19b} used the notation $\gamma\mapsto 2\gamma$ so that the KP equation takes the form $3u_{\gamma\gamma}+\left(12u_{\tau}+12uu_x+u_{xxx} \right)_x=0$.} In the context of the periodic TASEP, Prolhac \cite{Prolhac19} noticed that $\partial_{xx}\det(\mb I-\mb K_z)$ is ``a reminiscent of soliton solutions for the KP equation''. Corollary~\ref{thm:integrable_structure} establishes a precise connection. 

We used the connection to an IIKS-integrable operator to derive integrable differential equations for the step initial condition. 
This connection extends to the flat initial condition as well. 
However, it is not clear if it extends to general initial conditions. 
Nonetheless, the authors in \cite{Quastel-Remenik19b, LeDoussal-2020} used differential identities of the kernel to derive the KP equation for general initial conditions. 
As pointed out in \cite{Quastel-Remenik19b}, this method had appeared in several papers in the past, among which are the papers  \cite{PoppeSattinger-1988,Poppe1989} that derive the KP equation for a general class of kernels.
It is straightforward to check that the general result of \cite{Poppe1989} is applicable to the 
kernel $\mb T_{-\gamma}\mb T_{\gamma}$ for the periodic case and gives an alternative proof for Corollary \ref{thm:integrable_structure}. Indeed, it can be used to derive the KP equation for general initial conditions for the periodic case and we obtain the following result. 

\begin{thm}\label{thm:kp_general}
Let $\Kzic$ be the kernel for a general initial condition\footnote{The initial condition is assumed to satisfy certain technical conditions.} for the periodic case obtained in \cite{Baik-Liu19}; see Section \ref{sec:generalIC}. 
The function
$$
u=U_{\mathrm{ic}}(\tau,\gamma,x):=\partial_{xx}\det\log(\mb I-\Kzic)
$$
solves the KP equation \eqref{kp_equation}.
\end{thm}

The analysis in this subsection can be extended to multi-time, multi-location distributions and the results will be announced elsewhere.

\subsection{Asymptotic properties of $\Fp$} 

We also study in detail the distribution \eqref{eq:Fdeffirst} in various asymptotic regimes.

\subsubsection{Large time limit}

The function $F$ depends non-trivially on $\tau$. 
We consider the large $\tau$ limit first. 

\begin{thm}[Large $\tau$ limit]\label{thm:gaussian_degeneration}
For each fixed $x\in \R$ and $\gamma\in \R$,
$$
\lim_{\tau \to \infty} F\left(-\tau + \frac{\pi^{1/4}}{\sqrt{2}}x\tau^{1/2};\tau,\gamma\right)=\frac{1}{\sqrt{2\pi}}\int_{-\infty}^x e^{-\frac{y^2}{2}}dy.
$$
\end{thm}

This result was conjectured in \cite[(4.14)]{Baik-Liu16} and the above theorem confirms it. 
This is a natural result to expect since heuristically
the model degenerates to a one dimensional random growth model along the time direction  
when $t\gg L^{3/2}$ in the periodic TASEP. 
However, the proof using the formula of $F$ turns out to be technical. 
The leading non-trivial contribution to the Fredholm determinant comes from the trace of $\mb K_z$. 
We then still need to analyze the integral over $z$. 
To evaluate this integral asymptotically using the method of steepest-descent, it turned out that we need to consider an analytic continuation of polylog functions on a Riemann surface and the main contribution comes from a boundary point on a new sheet of the Riemann surface. 
Recently motivated by the same function $F(x;\tau,\gamma)$, Prolhac \cite{Prolhac19} studied Riemann surfaces associated to general polylog functions.
In this paper, we carry out the analytic continuation directly for polylog functions of positive half integer index.

\subsubsection{Small time limit}

Next result is the small $\tau$ limit. 

\begin{thm}[Small $\tau$ limit for $\gamma=0$] \label{thm:tracy_widom_degeneration}
For every $x_0\in \R$,
there exist constants $C,c, \epsilon>0$ and $\tau_0>0$ such that
$$
\left|F (\tau^{1/3}x;\tau, 0 )-F_{\mathrm{GUE}}(x)\right|\leq C e^{-\epsilon \tau^{-2/3}}e^{-c(x-x_0)},\quad 0<\tau <\tau_0, \; x\geq x_0.
$$
In particular,
$$
\lim_{\tau \to 0} F (\tau^{1/3}x;\tau, 0 ) = F_{\mathrm{GUE}}(x)
$$
uniformly for $x\geq x_0$.
\end{thm}

When $\tau$ is small, the periodicity effect becomes small, and hence in the limit $\tau\to 0$ we expect that the periodic model reduces to the non-periodic model on the line. 
Recall that the initial condition is given by \emph{periodic} step initial condition. 
Hence, if the model were the TASEP on the infinite line, 
the locations at which the initial density profile changes from $0$ to $1$ generate shocks. 
At a shock location, the large time limit of the one-point distribution is not given by $F_{\mathrm{GUE}}$ but $F_{\mathrm{GUE}}^2$ (see, for example, \cite{Ferrari-Nejjar15}). 
Based on this observation, it was conjectured in \cite[(4.13)]{Baik-Liu16} that 
\begin{equation}\label{eq:TW_limit_conj} 
	\lim_{\tau\to 0} F\left(\tau^{1/3}x - \frac{\gamma^2}{4\tau};\tau,\gamma\right) = \begin{cases} F_{\mathrm{GUE}}(x), \qquad &-\frac12<\gamma< \frac12, \\
	(F_{\mathrm{GUE}}(x))^2, \qquad &\gamma=\frac12. 
	\end{cases}
\end{equation}
The above result confirms this conjecture when $\gamma=0$. 
The proof consists of showing that a conjugated version of the operator $\mb K_z$ converges to the Airy operator, and that $e^{xA_1(z)+\tau A_2(z)+2B(z)}$ converges to $1$ when $z$ is scaled appropriately. 
The convergence to the Airy kernel becomes substantially complicated when $\gamma\neq 0$. 
Although we believe that our arguments could be improved to analyze the case when $\gamma$ is sufficiently small, it is not yet clear how to extend the analysis to the case when $|\gamma|$ is close to $1/2$.

\subsubsection{Right tail estimate} 

We also study the right tail of the distribution function $F$. Recall that it is periodic in $\gamma$ with period $1$, so it is enough to consider $\gamma\in (-1/2,1/2]$.

\begin{thm}[Right tail estimate]\label{thm:large_tail}
For every fixed $\tau>0$ and $\gamma$, there is a constant $c>0$ such that 
	\beqq
	1-F(x; \tau, \gamma) = \left( 1-F_{\mathrm{GUE}} \left( \frac{x}{\tau^{1/3}}+\frac{\gamma^2}{4\tau^{4/3} }\right)  \right) \left( 1+ \Boh(e^{-cx^{1/2}} ) \right) \quad \text{for $-\frac12 <\gamma<\frac12$}
	\eeqq
and 
	\beqq
	1-F(x; \tau, \gamma) = 2\left( 1-F_{\mathrm{GUE}} \left( \frac{x}{\tau^{1/3}}+\frac{\gamma^2}{4\tau^{4/3} }\right)  \right) \left( 1+ \Boh(e^{-cx^{1/2}} ) \right) \quad \text{for $\gamma= \frac12$}
	\eeqq
as $x\to \infty$. 
\end{thm}

The above result implies that
\begin{equation*}
\begin{split}
&1-F\left(\tau^{1/3}x-\frac{\gamma^2}{4\tau};\tau,\gamma\right)
	=\begin{dcases}
\left( 1-F_{\mathrm{GUE}} \left( x\right)  \right) \left( 1+ \Boh(e^{-cx^{1/2}} ) \right), & -\frac12 <\gamma<\frac12,\\
\left( 1-F_{\mathrm{GUE}} \left( x\right)^2  \right) \left( 1+ \Boh(e^{-cx^{1/2}} ) \right), &\gamma=\frac12,
\end{dcases}
\end{split}
\end{equation*}
as $x\to\infty$.
This result is consistent with the conjectured small $\tau$ limit~\eqref{eq:TW_limit_conj}.

At this stage we cannot compare the above result with the large $\tau$ limit in Theorem \ref{thm:gaussian_degeneration},
as in order to make the parameters in the same form we would need to consider the left tail as $x\to -\infty$, which we do not pursue in this paper.

\subsection{Flat initial condition}

If the periodic TASEP starts with the flat initial condition, the limit of the one point distribution takes a slightly different form. 
All of the above results have an analogue for the flat initial condition which we discuss in Section~\ref{sec:flat_IC}.

\subsection{Organization of the paper}

In Section~\ref{sec:main} we state the definition of the function $F(x; \tau, \gamma)$ and obtain alternative 
representations for it. Lemma~\ref{lem:kernelK_kernelT} and Theorem~\ref{thm:kp_general} are also proven in Section~\ref{sec:main}.
Section~\ref{section:polylogs} discusses several analytic continuation properties of the polylog function, which we use in Section~\ref{sec:gaussian} to evaluate the large $\tau$ limit. These two sections are the most technical part of this paper. 
We obtain the small $\tau$ limit in Section~\ref{sec:tracy_widom} and the right tail estimate in Section~\ref{sec:large_tail}. 
The connection to integrable differential equations is discussed in Section~\ref{sec:integrable_structure}. 
The flat initial condition is discussed in Section \ref{sec:flat_IC}. 
Some proofs for auxiliary results used in Section~\ref{sec:integrable_structure} are given in Appendix \ref{sec:proofofdeformation} and \ref{sec:iiks_general}. 


\subsection*{Acknowledgments}

We would like to thank Peter Miller and Andrei Prokhorov for useful comments in earlier versions of this manuscript. The work of J.B. was supported in part by NSF grant DMS-1664531, DMS-1664692, and DMS-1954790.
The work of Z.L. was supported by the University
of Kansas Start Up Grant, the University of Kansas New Faculty General Research Fund, Simons
Collaboration Grant No. 637861, and NSF grant DMS-1953687.  The work of G.S. is supported by S\~ao
Paulo Research Foundation (FAPESP) under grants $\# 2019/16062-1$ and  $\# 2020/02506-2$.

\section{One point distribution} \label{sec:main}

\subsection{Periodic TASEP with step initial condition}

We first review the limit theorem of the one point distribution of the periodic TASEP.
Let $N$ and $L$ be the number of particles and fundamental period, respectively, where $0<N<L$.
Consider the situation when $L,N\to \infty$ with density of particles $N/L\to \rho \in (0,1)$.  
We take the periodic step initial condition; at time zero, \cb{each of the sites $-N+1, \cdots, -1, 0$ are occupied by a particle} and the sites $1, \cdots, L-N$ are empty, and this pattern is extended periodically with period $L$. 
We denote by $\mathcal{H}(j,t)$ the height function at site $j$ at time $t$. 
At $j=0$, it is defined to be $2$ times the number of particles that have moved from the site $0$ to $1$ up to time $t$. 
For $j\ge 1$, it is defined to be $\mathcal{H}(j,t)= \mathcal{H}(0,t) + \sum_{k=1}^j (1-2\eta_k(t))$ where $\eta_k(t)$ is the occupation variable that takes value $1$ if the site $j$ is occupied and value $0$ if it is empty at time $t$. 
For $j\le -1$, it is defined to be $\mathcal{H}(j,t)= \mathcal{H}(0,t)- \sum_{k=j+1}^{0} (1-2\eta_k(t))$. 

It was shown\footnote{It was also shown in a non-rigorous paper \cite{Prolhac16}.} in  \cite[Theorem 3.4]{Baik-Liu16}\footnote{Theorem 3.4 of \cite{Baik-Liu16} states the result in terms of the current, but it is easy to translate the results in terms of the height function. Furthermore, the limit in Theorem 3.4 of \cite{Baik-Liu16} is given as $F(\tau^{1/3} x;\tau,\gamma)$, but in this paper, we use a slightly different scaling, which has the effect of changing $\tau^{1/3}x$ to $x$. See Theorem 2.1 of \cite{Baik-Liu19} for the explicit statement in terms of the height function.} that with the scale of the position and time given by 
$$
s=\gamma L,\quad t=\tau \frac{L^{3/2}}{\sqrt{\rho(1-\rho)}}
$$
for parameters $\gamma\in (-\frac12,\frac12]$ and $\tau >0$, we have the limit theorem  
\begin{equation}\label{eq:limit_height_distr}
	\lim_{L\to\infty}\P\left( \frac{\mathcal{H}(t(1-2\rho)+ s,t)-(1-2\rho)s-(1-2\rho+2\rho^2)t  }{-2\rho^{1/2}(1-\rho)^{1/2}L^{1/2}}\leq x\right)=F(x;\tau,\gamma) 
\end{equation}
for any fixed $x\in \R$. The function $F(x; \tau, \gamma)$ is the one in \eqref{eq:Fdeffirst} and we describe it explicitly in the next subsection. 
Note that the height at time $t$ is evaluated at the location $j=t(1-2\rho) + s$, where $s$ measures the relative distance from the location $t(1-2\rho)$. 
The reason for this shift is that the periodic step initial condition generate shocks that travel at speed $1-2\rho$, and we consider the moving frame along the shock direction.
Also note that the time, relative position, and average-adjusted height have orders $\Boh(L^{3/2})$, $\Boh(L)$, and $\Boh(L^{1/2})$, which are consistent with the $3:2:1$ KPZ scale, and the time and period are related as $t=\Boh(L^{3/2})$ which is the relaxation time scale.

\subsection{Definition of the one point distribution function}

We define the function $F(x; \tau, \gamma)$ in \eqref{eq:Fdeffirst} explicitly. 
Let $\Li_s$ be the polylog function; see Section~\ref{section:polylogs} for their properties.
Set 
\begin{equation}\label{def:functions_A} 
	A_1(z):=-\frac{1}{\sqrt{2\pi}}\Li_{3/2}(z),\quad A_2(z):=-\frac{1}{\sqrt{2\pi}}\Li_{5/2}(z), 
\quad B(z):=\frac{1}{4\pi}\int_0^z \frac{(\Li_{1/2}(y))^2}{y}dy.
\end{equation}

To introduce the operator $\mb K_z$, first define the function 
\begin{equation}\label{def:Q_function}
Q(\xi):=\sqrt{\frac{2}{\pi}} \int_{-\infty}^\xi \Li_{1/2} (e^{-s^2/2}) ds \quad \text{for} \quad \frac{3\pi}{4}<\arg \xi<\frac{5\pi}{4}, 
\end{equation}
where the path of integration is contained in the sector $\arg(s)\in (3\pi/4,5\pi /4)$ so that $e^{-s^2/2}\in \C\setminus [1,\infty)$, and as such the polylog is well defined. 
Alternatively, the representation 
\begin{equation}\label{eq:alternative_definition_q}
Q(\xi)=-\int_{-i\infty}^{i\infty} \frac{\log(1-e^{-\xi^2/2}e^{u^2/2})}{u-\xi}\frac{du}{\pi i}, \quad \frac{3\pi}{4}<\arg \xi<\frac{5\pi}{4},
\end{equation}
also holds (see \cite[(4.8)]{Baik-Liu16}). The function $Q$ is analytic in the sector. 

We set, for $3\pi/4<\arg \xi<5\pi/4$, 
\begin{equation}\label{def:phi_function}
	\Phi(\xi)=\Phi(\xi; \tau, x)=-\frac{1}{3}\tau \xi^3+x\xi -Q(\xi)
\end{equation}
and define the integral operator $\mb K_z$ acting on functions over $\mathcal S_-$ via the kernel 
\begin{equation}\label{eq:initial_kernel}
	\bm K_z(\xi_1,\xi_2)
	=\bm K_z(\xi_1,\xi_2; \tau,\gamma, x)=\sum_{\eta \in \mathcal S_-} \frac{e^{\Phi(\xi_1)+\Phi(\eta)+\frac{\gamma}{2}(\xi_1^2-\eta^2)}}{\xi_1\eta (\xi_1+\eta)(\eta+\xi_2)}, 
	\qquad \xi_1, \xi_2\in \mathcal S_-, 
\end{equation}
where $\mathcal S_-$ is the discrete set defined in \eqref{def:bethe_root_set}. 
Note that the kernel depends on $z$ since the set $\mathcal S_-$ depends on $z$. 
Due to the growth properties of $\Phi$, it is direct to check that the series definition of the Fredholm determinant $\det(\mb I-\mb K_z)$ is well-defined even though the kernel does not decay fast enough as $\xi_2\to \infty$ on $\mathcal S_-$. 
It is also direct to check that $\det(\mb I - \mb K_z)$ is analytic in $0<|z|<1$. 

\begin{definition}
The function in \eqref{eq:Fdeffirst} is defined by 
\begin{equation}\label{def:limiting_distribution}
	F(x;\tau,\gamma)=\oint e^{xA_1(z)+\tau A_2(z)+2B(z)}\det(\mb I-\mb K_z)\frac{dz}{2\pi i z}
\end{equation}
where the contour is a circle $|z|=R$ with $0<R<1$, oriented counterclockwise. 
\end{definition}

It was shown in \cite{Baik-Liu16} that $F(x; \tau, \gamma)$ is a distribution function, and it satisfies the spatial periodicity $F(x; \tau, \gamma+1)=F(x; \tau, \gamma)$. 

The above kernel can be conjugated to a symmetric kernel (using the same notation)
\begin{equation}\label{eq:initial_kernel2}
	\bm K_z(\xi_1,\xi_2)
	=\sum_{\eta \in S_-} \frac{e^{\frac12 \Phi(\xi_1)+\Phi(\eta)+\frac12\Phi(\xi_2)+\frac{\gamma}{4}(\xi_1^2-2\eta^2+ \xi_2^2)}}{\sqrt{-\xi_1}\sqrt{-\eta} (\xi_1+\eta)(\eta+\xi_2)\sqrt{-\eta}\sqrt{-\xi_2}}, \qquad \xi_1, \xi_2\in \mathcal S_-,
\end{equation}
without changing the Fredholm determinant, and this kernel properly defines a trace class operator $\mb K_z:\ell^2(\mathcal S_-)\to \ell^2(\mathcal S_-)$ (see Subsection~\ref{subsection:prooflemma}).
This kernel is symmetric but not self-adjoint since it is complex-valued. 

\subsection{Proof of Lemma~\ref{lem:kernelK_kernelT}}\label{subsection:prooflemma}

Lemma \ref{lem:kernelK_kernelT} claims that $\det(\mb I-\mb K_z)_{\ell^2(\mathcal S_-)}=\det(\mb I-\mb T_{-\gamma}\mb T_{\gamma})_{L^2(0,\infty)}$ with $\mb T_{\gamma}$ being the operator on $L^2(0,\infty)$ defined by the kernel \eqref{eq:defTgh}. 
The rest of this subsection proves this claim. 

From the formula \eqref{eq:initial_kernel2}, $\mb K_z$ is the product of two operators 
\begin{equation}\label{eq:split_kernel_1}
	\mb K_z=\mJ \mJ^T
\end{equation}
where $\mJ:\ell^2(\mathcal S_-)\to \ell^2(\mathcal S_-)$ has the kernel 
\begin{equation}\label{eq:kernel_K1}
	\bJ(\xi,\eta)=\frac{e^{\frac{1}{2}(\Phi(\xi)+\Phi(\eta))+\frac{\gamma}{4}(\xi^2-\eta^2)}}{\sqrt{-\xi} (\xi+\eta)\sqrt{-\eta}} ,\qquad \xi, \eta\in \mathcal S_-.
\end{equation}
The operator $\mJ^T:\ell^2(\mathcal S_-)\to \ell^2(\mathcal S_-)$ is obtained from $\mJ$ from the kernel
\begin{equation} \label{eq:Jsymr}
	\bJ^T(\eta, \xi; \gamma)= \bJ(\xi, \eta; \gamma) = \bJ(\eta, \xi; -\gamma) .
\end{equation}
Using the identity $\frac{1}{\eta+\xi}=- \int_0^\infty e^{s(\eta+\xi)}ds $ that holds for $\eta, \xi\in \mathcal S_-$, we can 
write 
\begin{equation}\label{def:operators_K1_K2}
	\mJ=-\mb G_{-\gamma} \mb G_{\gamma}^T,
\end{equation}
where $\mb G_\gamma: L^2(0,\infty)\to \ell^2(\mathcal S_-)$ is the operator defined by the kernel
\begin{equation*}
	\bm G_{\gamma}(\xi,s)  =\frac{e^{\frac{1}{2}\Phi(\xi)-\frac{\gamma}{4}\xi^2+s\xi}}{\sqrt{-\xi}}
	\qquad \text{for $\xi\in \mathcal S_-$ and $s\in (0,\infty)$. }
\end{equation*}
It is immediate that $\mb G_\gamma$ is Hilbert-Schmidt. Consequently, $\mJ,\mJ^T$ and $\mb K_z$ are all trace class operators. 
From the decompositions above, 
$$
	\det(\mb I-\mb K_z) = \det(\mb I - \mJ \mJ^T)
	=\det (\mb I - \mb G_{-\gamma}\mb G_{\gamma}^T \mb G_{\gamma} \mb G_{-\gamma}^T)
	=\det (\mb I - \mb G_{\gamma}^T \mb G_{\gamma}\mb G_{-\gamma}^T \mb G_{-\gamma} ).
$$
Setting $\mb T_\gamma=\mb G_{-\gamma}^T\mb G_{-\gamma}$, we find that $\mb T_\gamma$ is trace class and its kernel is \eqref{eq:defTgh}. This proves Lemma~\ref{lem:kernelK_kernelT}.

\subsection{Integrable operator formula}

There is another representation for $\det(\mb I-\Kk)$ as well as $\det(\mb I-\mb K_z)$. 
These representations will be used in Section~\ref{sec:integrable_structure} when we derive integrable differential equations. 

For $|z|<1$, recall the set $\mathcal S_-$ in \eqref{def:bethe_root_set}, and also let  $\Lambda_-$ be an unbounded oriented contour from $e^{-i \theta_1}\infty$ to $e^{i \theta_2}\infty$ for some $\theta_1, \theta_2\in (\pi/2, 5\pi/6)$. 
Define (see Figure \ref{fig:bethe_roots})
\begin{equation}\label{def:Lambda_bethe} \begin{split}
	&\mathcal S=\mathcal S(z):=\mathcal S_+\cup\mathcal S_- \quad \text{with $\mathcal S_+:=-\mathcal S_-$, and}\\
	&\Lambda:=\Lambda_+\cup\Lambda_- \quad \text{with $\Lambda_+:=-\Lambda_-$.} \\
\end{split}\end{equation}
\begin{figure}\centering
\begin{minipage}{.4\textwidth}
\centering
\begin{tikzpicture}[scale=0.5]
\begin{scope}
	\clip (0, 0) circle (4.1);
	\path[fill=gray!20] (-0.1,0) -- (-0.1,5) -- (-6.92,4) -- cycle;
	\path[fill=gray!20] (-0.1,0) -- (-0.1,-5) -- (-6.92,-4) -- cycle;
	\path[fill=gray!20] (0.1,0) -- (0.1,5) -- (6.92,4) -- cycle;
	\path[fill=gray!20] (0.1,0) -- (0.1,-5) -- (6.92,-4) -- cycle;
\end{scope}

\draw [line width=0.4mm,gray!50] (-4.3,0)--(4.3,0);
\draw [line width=0.4mm,gray!50] (0,-4.3)--(0,4.3);

\draw[domain=-2.8:2.8,smooth,variable=\y,black,thick]  plot ({-(\y*\y+1)^(1/2)},{\y});
\draw[domain=-2.8:2.8,smooth,variable=\y,black,thick]  plot ({(\y*\y+1)^(1/2)},{\y});
\end{tikzpicture}
\end{minipage}%
\begin{minipage}{.4\textwidth}
\centering
\begin{tikzpicture}[scale=0.5]

\draw [line width=0.3mm,lightgray] (-4.3,0)--(4.3,0);
\draw [line width=0.3mm,lightgray] (0,-4.3)--(0,4.3);

\draw[thick,domain=-3:3,smooth,variable=\y,lightgray,thick]  plot ({(\y*\y+1)^(1/2)},{\y});
\draw[thick,domain=-3:3,smooth,variable=\y,lightgray,thick]  plot ({-(\y*\y+1)^(1/2)},{\y});

\fill[shift={(-2.5pt,-2.5pt)}] (1.06573,0.368479) rectangle ++(5pt,5pt) node [above,shift={(0pt,0pt)}] {};
\fill[shift={(-2.5pt,-2.5pt)}] (2.0173,1.75199) rectangle ++(5pt,5pt) node [above,shift={(0pt,0pt)}] {};
\fill[shift={(-2.5pt,-2.5pt)}] (2.68227,2.48889) rectangle ++(5pt,5pt) node [above,shift={(0pt,0pt)}] {};

\fill[shift={(-2.5pt,-2.5pt)}] (2.53213,-2.3263) rectangle ++(5pt,5pt) node [above,shift={(0pt,0pt)}] {};
\fill[shift={(-2.5pt,-2.5pt)}] (1.81494,-1.51459) rectangle ++(5pt,5pt) node [above,shift={(0pt,0pt)}] {};

\fill (-1.06573,-0.368479) circle[radius=2.5pt] node [above,shift={(0pt,0pt)}] {};
\fill (-2.0173,-1.75199) circle[radius=2.5pt] node [above,shift={(0pt,0pt)}] {};
\fill (-2.68227,-2.48889) circle[radius=2.5pt] node [above,shift={(0pt,0pt)}] {};
\fill (-2.53213,2.3263) circle[radius=2.5pt] node [above,shift={(0pt,0pt)}] {};
\fill (-1.81494,1.51459) circle[radius=2.5pt] node [above,shift={(0pt,0pt)}] {};
\end{tikzpicture}
\end{minipage}
\caption{The left picture is the set $\Lambda=\Lambda_+\cup\Lambda_-$. In the right picture, the sets $\mathcal S_+$ and $\mathcal S_-$ are represented by the squares and the dots, respectively. The dashed curves are the hyperbolas. Compare this figure with Figure~\ref{fig:LambdaandS}.}
\label{fig:bethe_roots}
\end{figure}
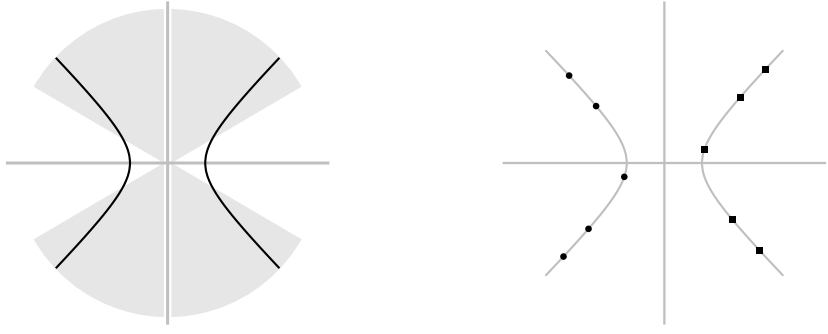
In addition, set
\begin{equation} \label{def:functionV}
V(u)=V(u;\tau,\gamma,x):=-\frac{\tau}{3}u^3+\frac{\gamma}{2}u^2+xu,\quad u\in \C,
\end{equation}
and denote by $\sigma_3$ the third Pauli matrix, 
$$
\sigma_3:=\begin{pmatrix}
1 & 0 \\ 0 & -1
\end{pmatrix}, \quad \text{so in particular, }\quad
e^{\pm \frac{1}{2}V(u)\sigma_3}=
\begin{pmatrix}
e^{\pm\frac{1}{2}V(u)} & 0 \\ 0 & e^{\mp\frac{1}{2}V(u)}
\end{pmatrix}.
$$

\begin{lem} \label{lem:integrability_operator_H}
Let $\chi_+$ and $\chi_-$ be the indicator functions of $\mathcal S_+$ and $\mathcal S_-$, respectively, and define vector functions 
$$
\vec f(u)=e^{-\frac{1}{2}V(u)\sigma_3}\vec f_0(u) \quad \text{and}\quad \vec g(u)=e^{\frac{1}{2}V(u)\sigma_3}\vec g_0(u),
$$
with
$$
\vec f_0(u)=
\begin{pmatrix}
\dfrac{e^{\frac{1}{2}Q(u)}}{\sqrt{u}}\chi_+(u) \\ 
\dfrac{e^{-\frac{1}{2}Q(u)}}{\sqrt{-u}}\chi_-(u)
\end{pmatrix}
\quad \text{and}\quad 
\vec g_0(u)=
\begin{pmatrix}
-\dfrac{e^{-\frac{1}{2}Q(u)}}{\sqrt{-u}}\chi_-(u) \\ 
\dfrac{e^{\frac{1}{2}Q(u)}}{\sqrt{u}}\chi_+(u)
\end{pmatrix} .
$$
Then
$$
\det(\mb I-\mb K_z)_{\ell^2(\mathcal S_-)}=\det (\mb I-\mb H)_{\ell^2(\mathcal S)}
$$ 
where $\mb H$ is the trace class operator acting on $\ell^2(\mathcal S)$ with kernel
\begin{equation}\label{def:kernel_H}
\bm H(u,v)=\frac{\vec f(u)^T \vec g(v)}{u-v}, \quad u\neq v, \quad \text{and}\quad \bm H(u,u)=0.
\end{equation}
\end{lem}

\begin{proof}
Recall that $\mb K_z= \mJ \mJ^T$ in \eqref{eq:split_kernel_1} where $\mJ :\ell^2(\mathcal S_-)\to \ell^2(\mathcal S_-)$ has the kernel $\bJ$ given by \eqref{eq:kernel_K1}. 
Noting that $\mathcal S_+=-\mathcal S_-$, we set $\mb W: \ell^2(\mathcal S_+)\to \ell^2(\mathcal S_-)$ to be the reflection operator defined by $(\mb W h)(u)=h(-u)$
and write $\mb K_z= \mJ_1 \mJ_2$, where $\mJ_1=\mJ \mb W:\ell^2(\mathcal S_+)\to \ell^2(\mathcal S_-)$ and $\mJ_2= \mb W^T\mJ^T: \ell^2(\mathcal S_-)\to \ell^2(\mathcal S_+)$. Since $\mathcal S_-$ and $\mathcal S_+$ are disjoint, we have $\ell^2(\mathcal S_-)\oplus \ell^2(\mathcal S_+)= \ell^2(\mathcal S_-\cup \mathcal S_+)$, and 
\begin{align*}
	\det(\mb I- \mb K_z)_{\ell^2(\mathcal S_-)}
	= \det\left( \mb I-\begin{pmatrix} 0 & \mJ_1 \\ \mJ_2 & 0 \end{pmatrix}\right)_{\ell^2(\mathcal S_-)\oplus \ell^2(\mathcal S_+)} 
	= \det \left( \mb I - \mb H \right)_{\ell^2(\mathcal S_-\cup \mathcal S_+)}
\end{align*}
where $\mb H$ is the operator on $\ell^2(\mathcal S_-\cup \mathcal S_+)$ with the kernel 
\beqq \begin{split}
	\bm H(u,v)
	&=  \chi_+(u) \bJ_2(u,v) \chi_-(v)  + \chi_-(u) \bJ_1(u,v) \chi_+(v) \\
	&= \chi_+(u) \frac{e^{\frac12 \Phi(v) -\frac12 \Phi(u) + \frac{\gamma}4 (v^2-u^2)} }{\sqrt{u}\sqrt{-v}(v-u) } \chi_-(v)  + \chi_-(u) \frac{e^{\frac12 \Phi(u) - \frac12 \Phi(v) + \frac{\gamma}4 (u^2-v^2)}}{\sqrt{-u}\sqrt{v}(u-v) } \chi_+(v) 
\end{split} \eeqq
for $u, v\in \mathcal S_+\cup \mathcal S_-$. Because $\mJ$ is trace class and $\mb W$ is bounded, the operator $\mb H$ is trace class.
The result now follows by inserting the formula $\Phi(\xi)=-\frac{1}{3}\tau \xi^3+x\xi -Q(\xi)$ (see \ref{def:phi_function}). 
\end{proof}

There is an analogue to the previous result for $\det(\mb I-\Kk)$. For the case when $\tau=1$ and $\gamma=0$, the next result is a particular instance of a known representation for 
the Airy$_2$ process \cite{Bertola-Cafasso-2012b, Bertola-Cafasso-2012a}.

\begin{lem}\label{lem:integrability_operator_F}
Let $\chi_{\Lambda_+}$ and $\chi_{\Lambda_-}$ be the characteristic functions of the contours $\Lambda_+$ and $\Lambda_-$ in \eqref{def:Lambda_bethe}. Define the vector functions
$$
\vec a(u)=e^{-\frac{1}{2}V(u)\sigma_3}\vec a_0(u),\quad \vec b(u)=e^{\frac{1}{2}V(u)\sigma_3}\vec b_0(u),
\quad \text{with}\quad
a_0(u)=
\begin{pmatrix}
\chi_{\Lambda_+}(u) \\ 
-\chi_{\Lambda_-}(u)
\end{pmatrix},
\quad 
b_0(u)=
\begin{pmatrix}
\chi_{\Lambda_-}(u) \\ 
\chi_{\Lambda_+}(u)
\end{pmatrix}.
$$
Then,
$$
\det(\mb I-\Kk)_{L^2(0,\infty)}=\det(\mb I-\mb F)_{L^2(\Lambda)}
$$
where $\mb F:L^2(\Lambda)\to L^2(\Lambda)$ acts with kernel
\begin{equation}\label{def:kernel_F}
\bm F(u,v)=\frac{1}{2\pi i}\frac{\vec f(u)^Tg(v)}{u-v},\; u\neq v \quad \text{and}\quad \bm F(u,u)=0.
\end{equation}
\end{lem}

\begin{proof}
The definition of the kernel $\bm A_\gamma$ in \eqref{eq:KPZ_shiftedairykernel} gives the decomposition
$\mb A_{\gamma}=\mb B_\gamma^T\mb B_\gamma$, 
where $\mb B_\gamma:L^2(0,\infty)\to L^2(\Lambda_-)$ acts with kernel
$$
\bm B_\gamma(\eta,v)=\frac{1}{\sqrt{2\pi i}}e^{\frac{1}{2}V(\eta)+\eta v},\quad \eta\in \Lambda_-, v>0,
$$
where $V$ is given in \eqref{def:functionV}. 
Manipulating the Fredholm determinants,
$$
	\det(\mb I-\Kk )_{L^2(0,\infty)} =\det(\mb I-\mb A_{-\gamma}\mb A_\gamma)_{L^2(0,\infty)} 
	=\det(\mb I- \mb B_\gamma \mb B_{-\gamma}^T\mb B_{-\gamma}\mb B_\gamma^T )_{L^2(\Lambda)} 
	= \det(\mb I-\mb B\mb B^T)_{L^2(\Lambda)},
$$
where $\mb B=\mb B_{\gamma}\mb B_{-\gamma}^T$. The rest of the proof follows exactly as in the proof of Lemma~\ref{lem:integrability_operator_H}: the role of $\mb K_z$ is played by $\mb B\mb B^T$ and the operator $\mb J$ is now $\mb B$.
\end{proof}

The form \eqref{def:kernel_H} or \eqref{def:kernel_F} of the kernels is saying that the operators $\mb F$ and $\mb H$ are IIKS-integrable, which is a class of operators that was first singled out by Its, Izergin, Korepin and Slavnov \cite{IIKS}. In Section~\ref{sec:integrable_structure} we discuss and explore this structure in detail.

\subsection{The periodic TASEP with general initial condition: proof of Theorem~\ref{thm:kp_general}} \label{sec:generalIC}

In \cite{Baik-Liu19b} an extension of \eqref{eq:limit_height_distr} was obtained for the periodic TASEP with general initial condition satisfying certain technical assumptions, which were verified for step, flat, and step-flat initial conditions. 
The limiting distribution takes the form
\begin{equation*}
\Fic(x;\tau,\gamma)=\oint E_{\mathrm{ic}}(z) e^{xA_1(z)+\tau A_2(z)+2B(z)}\det(\mb I-\Kzic) \frac{dz}{2\pi i z}
\end{equation*}
for a prefactor $E_{\mathrm{ic}}(z)$ which depends on the initial condition and also on the variable $z$ but is independent of the parameters $x,\tau$ and $\gamma$. The operator $\Kzic:\ell^2(\mathcal S_-)\to \ell^2(\mathcal S_-)$ is a generalization of $\mb K_z$ and acts with kernel\footnote{The original formulation in \cite{Baik-Liu19b} is a modified version of \eqref{eq:initial_kernel}, which can then be conjugated to \eqref{eq:kernel_ic}.}
\begin{equation}\label{eq:kernel_ic}
	\bKzic(\xi_1,\xi_2)
	=\sum_{\eta \in S_-} \chi_{\mathrm{ic}}(-\eta,\xi_2)\frac{e^{\frac12 \Phi(\xi_1)+\Phi(\eta)+\frac12\Phi(\xi_2)+\frac{\gamma}{4}(\xi_1^2-2\eta^2+ \xi_2^2)}}{\sqrt{-\xi_1}\sqrt{-\eta} (\xi_1+\eta)(\eta+\xi_2)\sqrt{-\eta}\sqrt{-\xi_2}}, \qquad \xi_1, \xi_2\in \mathcal S_-,
\end{equation}
where $\chi_{\mathrm{ic}}(-\eta,\xi_2)=\chi_{\mathrm{ic}}(-\eta,\xi_2;z)$ depends on the initial condition and on the parameter $z$ but not on $x,\gamma,\tau$.

Similarly as in \eqref{eq:split_kernel_1}, we write
$$
\Kzic=\mb J\mb J^{\mathrm{ic}},
$$
where $\mb J:\ell^2(\mathcal S_-)\to \ell^2(\mathcal S_-)$ is as in \eqref{eq:split_kernel_1} and $\mb J^{\mathrm{ic}}:\ell^2(\mathcal S_-)\to\ell^2(\mathcal S_-)$ is defined by the kernel
$$
\bm J^{\mathrm{ic}}(\xi,\eta)=\chi_{\mathrm{ic}}(-\xi,\eta)\bm J^T(\xi,\eta)=\chi_{\mathrm{ic}}(-\xi,\eta)\bm J(\eta,\xi),\qquad \xi,\eta\in \mathcal S_-.
$$
We take advantage of the same decomposition \eqref{def:operators_K1_K2} and manipulate
$$
\det(\mb I-\Kzic)=\det(\mb I+\mb G_{-\gamma}\mb G_\gamma^T \mb J^{\mathrm{ic}})=\det(\mb I-\mb T^{\mathrm{ic}}),
$$
where $\mb T^{\mathrm{ic}}:L^2(0,\infty)\to L^2(0,\infty)$ is the product $\mb T^{\mathrm{ic}}=-\mb G_\gamma^T \mb J^{\mathrm{ic}}\mb G_{-\gamma}$ and has kernel
$$
\bm T^{\mathrm{ic}}(u,v)=-\sum_{\xi,\eta\in \mathcal S_-}\frac{\chi_{\mathrm{ic}}(-\eta,\xi)}{\xi(\xi+\eta)\eta} e^{\Phi(\xi)+\Phi(\eta)+\frac{\gamma}{2}(\xi^2-\eta^2)+u\eta+v\xi  },\qquad u,v>0.
$$

It is convenient to rescale the kernel and consider the auxiliary operator  
\begin{equation}\label{eq:kernel_ic_changevar}
\widehat{\mb T}^{\mathrm{ic}}=\widehat{\mb T}^{\mathrm{ic}}(\tau,\gamma,x)=-\mb T^{\mathrm{ic}}(-3\tau,2\gamma,x)
\end{equation}
which has explicit kernel
$$
\widehat{\bm T}^{\mathrm{ic}}(u,v) =\widehat{\bm T}^{\mathrm{ic}}(u,v;\tau,\gamma,x)=\sum_{\eta,\xi\in \mathcal S_-}\frac{\chi_{\mathrm{ic}}(-\xi,\eta)e^{Q(\xi)+Q(\eta)}}{\xi(\xi+\eta)\eta}e^{\tau(\xi^3+\eta^3)+\gamma(\xi^2-\eta^2)+(u+x)\xi + (v+x)\eta},\quad u,v>0.
$$
In particular, the differential identities
$$
\partial_{x}\widehat{\bm T}^{\mathrm{ic}}=\partial_u \widehat{\bm T}^{\mathrm{ic}}+\partial_v \widehat{\bm T}^{\mathrm{ic}},\quad 
\partial_{\gamma}\widehat{\bm T}^{\mathrm{ic}}=\partial_{uu}\widehat{\bm T}^{\mathrm{ic}}-\partial_{vv}\widehat{\bm T}^{\mathrm{ic}}, \quad
\partial_{\tau}\widehat{\bm T}^{\mathrm{ic}}=\partial_{uuu}\widehat{\bm T}^{\mathrm{ic}}+\partial_{vvv}\widehat{\bm T}^{\mathrm{ic}}.
$$
are of straightforward verification. By \cite[Theorem~3.1]{Poppe1989}, these identities imply that $\widehat u=\partial_{xx}\log \det(\mb I+\widehat{\mb T}^{\mathrm{ic}} )$ satisfies the KP equation in the form
\begin{equation}\label{eq:rescaled_KP}
3\widehat u_{\gamma\gamma}+\left(\widehat u_{xxx}+12\widehat u\widehat u_x-4\widehat u_\tau  \right)_x=0.
\end{equation}

In the definition of $\widehat{\bm T}^{\mathrm{ic}}(u,v;\tau,\gamma,x)$ we should only consider $\tau<0$ to ensure the series is convergent. On the other hand, the KP equation is an evolution equation in the time variable $\tau$, so naturally with $\tau>0$. Nevertheless, we are safe to apply \cite[Theorem~3.1]{Poppe1989} here also for $\tau<0$, as the arguments therein are of algebraic nature, relying on $\tau$ solely as a variable without physical meaning and no restriction on its sign.
Having in mind \eqref{eq:kernel_ic_changevar}, Theorem~\ref{thm:kp_general} now follows from \eqref{eq:rescaled_KP}.

\section{Analytic continuation of polylog functions}\label{section:polylogs}

The formula \eqref{def:limiting_distribution} of $F(x; \tau, \gamma)$ involves polylog functions $\Li_{1/2}(z), \Li_{3/2}(z)$ and $\Li_{5/2}(z)$. 
In the first subsection, we summarize some basic properties of polylog functions used in this paper. 
In the remaining subsections, we establish some analytic continuation properties of polylog functions and their combinations. 
The results of this section will be used only in Subsection \ref{sec:asymppolyintg} for the large $\tau$ limit of $F(x; \tau, \gamma)$.

Throughout this paper, $\log z$ denotes the principal branch of the logarithmic function with branch cut along $(-\infty, 0]$, and $z^\alpha$ denotes $z^\alpha=e^{\alpha \log z}$. 

\subsection{Polylog functions}

For $\re s>0$, the polylog function is defined by the series
\begin{equation}\label{def:series_polylog}
\Li_s(z):=\sum_{k=1}^\infty \frac{z^k}{k^s},\quad |z|<1,
\end{equation}
and it extends analytically using the integral representation
\begin{equation}\label{eq:integral_repr_polylog}
\Li_s(z)=\frac{z}{\Gamma(s)}\int_0^\infty \frac{t^{s-1}}{e^t-z}dt,\quad z\in \C\setminus [1,\infty).
\end{equation}
From this representation, Paulsen \cite[Proposition~3]{paulsen_polylog} showed that if $s$ is not an integer, then 
\begin{equation}\label{eq:expansion_polylog}
\Li_s(z)=\Gamma(1-s)(-\log  z)^{s-1}+\zeta(s)+\sum_{n=1}^\infty (z-1)^n \sum_{m=1}^{n} \frac{S_n^{(m)}\zeta(s-m)}{n!}\quad  \text{as } z\to 1,
\end{equation}
where $\zeta$ is the Riemann zeta function and $S_n^{(m)}$ are the Stirling numbers of the first kind.

Observe that the function $f(t)= (-t)^{s-1}$ has branch cut along $(0,+\infty)$ and it satisfies the identity 
$$
(-t)^{s-1}_+-(-t)^{s-1}_-=2i |t|^{s-1}\sin(\pi s),\quad t>0,
$$
where $f_{\pm}(t)$ denotes the limit of $f(t\pm i\epsilon)$ as $\epsilon\searrow 0$. 
Thus, for $s\in \C$ satisfying $s\notin\Z$ and $\re(s)>0$, we can rewrite the integral representation \eqref{eq:integral_repr_polylog} as a complex contour integral given by 
\begin{equation}\label{eq:integral_repr_polylog_2}
\Li_s(z)=-\frac{\Gamma(1-s)z}{2\pi i}\int_\Clambda \frac{(-t)^{s-1}}{e^t-z}dt,\quad z\in \C\setminus [1,\infty), 
\end{equation}
where we used the reflection formula $\Gamma(s)\sin(\pi s)=\frac{\pi}{\Gamma(1-s)}$. 
Here, $\Clambda$ is an unbounded contour that starts and ends at $+\infty$, lies in the strip $\{t\in \C: |\im t|<\pi\}$,  encircles the positive axis with counterclockwise orientation, and separates the point $\log z$ from the interval $[0,\infty)$, see  Figure~\ref{fig:contours_analytic_continuation}. 

\begin{figure}[t]
		\centering
		\begin{tikzpicture}[scale=0.7]
\draw [line width=0.4mm,lightgray] (-2,0)--(5,0) node [pos=1,right,black] {$\R$};
\draw [line width=0.4mm,lightgray] (1,-3)--(1,3) node [pos=1,above,black] {$i\R$};
\draw [line width=0.4mm,lightgray,dashed] (-2,-2.5)--(5,-2.5) node [pos=1,right,black] {$\R-\pi i$};
\draw [line width=0.4mm,lightgray,dashed] (-2,2.5)--(5,2.5) node [pos=1,right,black] {$\R+\pi i$};
\fill (1,0) circle[radius=2.5pt] node [below right,shift={(0pt,0pt)}] {$0$};
\fill (1.5,1) circle[radius=2.5pt] node [above right,shift={(-2pt,-2pt)}] {$\log z$};

\draw [line width=0.45mm] (1,0) to (5,0) node (b) {};
\path [draw=black,thick,postaction={mid3 arrow={black,scale=1.5}}]	(5,0.5) 
	to [out=180,in=-10] (1.5,0.5)
	to [out=-10+180,in=90] (0.3,0)
	to [out=90+180,in=180,edge node={node [pos=0.8,below] {$\Clambda$}}] (5,-0.6);
\path [draw=black,thick,postaction={mid3 arrow={black,scale=1.5}}]	(5,1.5) 
    to [out=180,in=20,edge node={node [pos=0.4,above] {$\widetilde\Clambda$}}] (0,1.5)
	to [out=20+180,in=90] (-1.5,0)
	to [out=90+180,in=180] (5,-1.5);

		\end{tikzpicture}
		\caption{The contours $\Clambda$ and $\widetilde\Clambda$ used in the integral representations of $\Li_s$ and $F_s$.}\label{fig:contours_analytic_continuation}
	\end{figure}
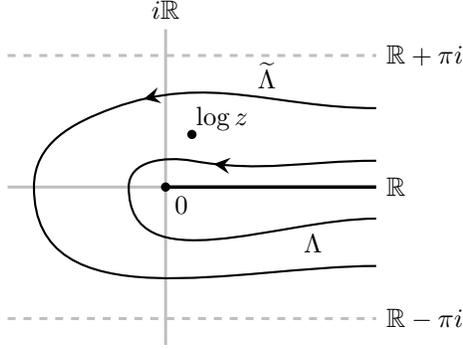

\subsection{Analytic continuation}\label{subsec:riemann_surface}

Define the set 
\begin{equation*}
 \RS_1=\C\setminus \left((-\infty,0]\cup [1,\infty) \right).
\end{equation*}
Let $\RS_2$ be another copy of $\RS_1$ and define a (non-compact two sheeted) Riemann surface $\RS$ by gluing $\RS_1$ and $\RS_2$ along {\it only} $[1,\infty)$, see Figure~\ref{fig:riemann_surface_S}. 
The only branch point\footnote{The Riemann surface $\RS$ here is bordered. We could extend $\RS$ to a non-compact Riemann surface without boundaries which is the fundamental covering surface for the polylog. In this extended surface $0^{(2)}$ would also be a branch point. Hence we sometimes refer to $0^{(2)}$ loosely as a branch point.} of the surface is $z=1$ and the local parameter near $z=1$ is given by $z-1=w^{2}$, $w\in \C$. For $j=1,2,$ we occasionally use the notation $p^{(j)}$ to denote the point on the sheet $\mathcal R_j$ that projects onto $p\in \mathcal \C$. This notion is well defined as long as $p\notin (-\infty,0)\cup (1,\infty)$. In particular,  $1^{(1)}=1^{(2)}$ so we simply write $1^{(j)}=1\in \mathcal R$ in this case.

\begin{figure}\centering
\begin{tikzpicture}
[x  = {(1cm,0cm)},
                    y  = {(0.707107cm,0.707107cm )},
                    z  = {(0 cm,1 cm)},
                    scale = 0.6,
                    color = {black}]

%
%
\begin{scope}[canvas is xy plane at z=0]
  \draw[white,->] (-5,0) -- (5,0) node (x-axis) [right] {$x$};
  \draw[white,->] (-5,0) -- (-5,2.8) node (y-axis) [above] {$y$};
\end{scope}
 
\begin{scope}[canvas is yz plane at x=-5]
  \draw[white,->] (0,-5) -- (0,3) node (z-axis) [above] {$z$};
\end{scope} 
%
%
%
\begin{scope}[canvas is xy plane at z=0]
\path[draw,line width=0.4mm] (-4,-1.8) to (4,-1.8) to [edge node={node [pos=0.4,right,shift={(4pt,0pt)}] {}}] (4,1.8);
%
%
\path [draw,line width=0.6mm,red] (-4,0) to [edge node={node [pos=0.5,above,shift={(2pt,-2pt)}] {}}] (0,0);
\path [draw,line width=0.6mm,blue] (1.5,0) to [edge node={node [pos=0.5,above,shift={(2pt,-2pt)}] {}}] (4,0);

\node at (4.5,0) [above right] {$\RS_1$};
\end{scope}
%

%
%
\begin{scope}[canvas is xy plane at z=-3]
\path[draw,line width=0.4mm,black] (-4,-1.8) to (4,-1.8) to [edge node={node [pos=0.4,right,shift={(4pt,0pt)}] {}}] (4,1.8);
%
%
\path [draw,line width=0.6mm,red] (-4,0) to [edge node={node [pos=0.25,above,shift={(2pt,-2pt)}] {}}] (0,0);
\path [draw,line width=0.6mm,blue] (1.5,0) to [edge node={node [pos=0.5,above,shift={(2pt,-2pt)}] {}}] (4,0);

\node at (4.5,0) [above right] {$\RS_2$};
\end{scope}
%
%
%
\foreach \x in {0.5,1,...,3}
\draw[line width=0.3mm,dashed,lightgray] (1+\x,0,-3)--(1+\x,0,0);

\fill [black] (0,0,-3) circle[radius=3pt] node [below] {$0$}; 
\fill [black] (0,0,0) circle[radius=3pt] node [below] {$0$};
\fill [black] (1.5,0,-3) circle[radius=3pt] node [below] {$1$}; 
\fill [black] (1.5,0,0) circle[radius=3pt] node [below] {$1$};
\fill (0,0,0) circle[radius=3pt] node [below] {$0$};

%

\end{tikzpicture}
\caption{The two sheets $\RS_1$ and $\RS_2$ that constitute the Riemann surface $\RS$. They are glued along the interval $[1,+\infty)$ but not along $(-\infty,0]$.}
\label{fig:riemann_surface_S}
\end{figure}
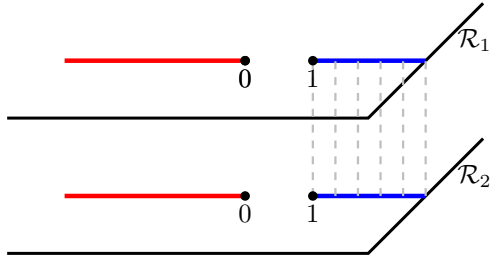

Define the function  
\begin{equation}\label{def:coordinate_u0}
u_0(z):=-(-2\log z)^{1/2}
,\quad z\in\RS_1 .
\end{equation}
From the choice of the branch of the logarithm and the square root, $u_0(z)$ is analytic in $\RS_1$. 
This function is one of the Bethe roots, 
see \eqref{def:bethe_root_set}. 
This particular Bethe root will play an important role in the large $\tau$ limit.
Note that $u_{0+}(x)=-u_{0-}(x)$ for $x>1$. 

\begin{lem}\label{lem:URS0zero}
The function 
\begin{equation}\label{def:coordinate_u0_sheet_2}
	\URS_0(z):=
\begin{cases}
u_0(z),& z\in \RS_1, \\
-u_0(z),& z\in \RS_2 \\
\end{cases}
\end{equation}
is analytic on $\RS$. 
The only zero of $\URS_0$ is the branch point $z=1$, which is a simple zero. As a consequence, $\frac{1}{\URS_0(z)}$ is a meromorphic function on $\RS$ with a simple pole at $z=1$ and no other poles.
\end{lem}

\begin{proof}
The analytic properties of $\URS_0$ are straightforward to check.
The fact that $\URS_0$ has a simple zero at $z=1$ is easily verified using the local coordinates $z-1=w^2$.
\end{proof}

Now consider the polylog $\Li_s$. 
For the rest of this section, we assume that $s$ is a positive half integer,
$$	s\in \Z+ \frac12 , \qquad s>0, $$
which is the case relevant to our purposes. 
We use the formula \eqref{eq:integral_repr_polylog_2} to extend the polylog to the Riemann surface $\RS$. 
To this end, we introduce an auxiliary function 
$$
P_s(z)=\int_{\widetilde \Clambda}\frac{(-t)^{s-1}}{e^t-z}dt, \quad z\in \C\setminus (-\infty,0],
$$
where $\widetilde \Clambda$ is a contour going from $+\infty +i0$ to $+\infty-i0$, encircling both the segment $[0,+\infty)$ and the point $\log z$, and being contained in the horizontal strip $|\im t|<\pi$; see Figure~\ref{fig:contours_analytic_continuation}. 
The condition $z\notin (-\infty, 0]$ implies that $\log z$ is in the horizontal strip $|\im t|<\pi$, and hence we can indeed choose such a contour $\widetilde \Clambda$ contained in the horizontal strip $|\im t|<\pi$.
Note that unlike  $\Li_s(z)$, the function $P_s(z)$ is analytic across the segment $[1,\infty)$.
Given a point $z\in \RS_1$, deforming the contour $\Clambda$ in \eqref{eq:integral_repr_polylog_2} to $\widetilde \Clambda$ we obtain the identity  
$$
\int_\Clambda \frac{(-t)^{s-1}}{e^t-z}dt=P_s(z)-2\pi i \res_{t=\log z}\left( \frac{(-t)^{s-1}}{e^t-z}\right) = P_s(z)-2\pi i \frac{(-\log z)^{s-1}}{z}.
$$
Hence, writing $\log z$ in terms of $u_0(z)$ in \eqref{def:coordinate_u0}, we get 
\begin{equation}\label{eq:identity_analytic_continuation_0}
\Li_s(z)=-\frac{\Gamma(1-s)}{2\pi i} zP_s(z)+2^{1-s} \Gamma(1-s)(- u_0(z))^{2s-2} \quad \text{for $z\in \RS_1$.}
\end{equation}
This formula implies the following extension result. 

\begin{lem}
Assume that $2s$ is a positive odd integer. The function 
\begin{equation} \label{eq:LRSdef}
	\LRS_s(z) := \begin{cases}
	\Li_{s}(z) \quad &\text{for $z\in \RS_1$,} \\
	\Li_s(z) - 2^{2-s}\Gamma(1-s)(-u_0(z))^{2s-2} \quad &\text{for $z\in \RS_2$,}
	\end{cases}
\end{equation}
is analytic on $\RS$ for $s\neq \frac12$. For $s=\frac12$,  $\LRS_{1/2}$ is meromorphic with a simple pole at $z=1$ and no other poles.
\end{lem}

\begin{proof}
We set $\LRS_s(z)= -\frac{\Gamma(1-s)z}{2\pi i} P_s(z)+2^{1-s} \Gamma(1-s)(- \URS_0(z))^{2s-2}$ where $\URS_0(z)$ given by \eqref{def:coordinate_u0_sheet_2}. Using $(-1)^{2s-2}=-1$ and \eqref{eq:identity_analytic_continuation_0}, we observe that this function is equal to \eqref{eq:LRSdef} on each sheet. For $s=1/2$, we note that $(- \URS_0(z))^{2s-2}= - \frac1{\URS_0(z)}$. 
\end{proof}

\subsection{Extension of the functions $\ff_1$ and $\ff_2$} \label{sec:f12exRS}

In the asymptotic analysis of $F(x; \tau, \gamma)$ as $\tau\to \infty$ in the next section, we need to extend several functions involving polylog to the Riemann surface $\RS$. These functions are defined below as $\ff_1$, $\ff_2$, and $\EE$ and appear later in \eqref{eq:ff12E}. 
In this subsection we discuss the analytic extension and properties of $\ff_1$ and $\ff_2$, and postpone the discussion of $\EE$ to the next subsection.

The functions $\ff_1(z)$ and $\ff_2(z)$ are given by 
\begin{equation*} \begin{split}
	&\ff_1(z) := \frac1{\sqrt{2\pi}} \Li_{3/2}(z) - \frac1{\sqrt{2\pi}} \Li_{5/2}(z)  -2 u_0(z) - \frac23 u_0(z)^3, 
	\quad \ff_2(z) :=   2u_0(z) - \frac1{\sqrt{2\pi}} \Li_{3/2}(z) .
\end{split} \end{equation*}
Define 
\begin{equation*}
\begin{split}
	\ffRS_1(z) := \frac1{\sqrt{2\pi}} \LRS_{3/2}(z) - \frac1{\sqrt{2\pi}} \LRS_{5/2}(z)  -2 \URS_0(z) - \frac23 \URS_0(z)^3
\end{split} 
\end{equation*}
and
\begin{equation*}
\begin{split}
	\ffRS_2(z) :=  2\URS_0(z) - \frac1{\sqrt{2\pi}} \LRS_{3/2}(z) .
\end{split} 
\end{equation*}
These functions are analytic on $\RS$.
They satisfy $\ffRS_1=\ff_1$ and $\ffRS_2=\ff_2$ on the first sheet $\RS_1$. 
For later use, we state the formula on the second sheet explicitly. 
Using \eqref{def:coordinate_u0_sheet_2} and \eqref{eq:LRSdef}, and $\Gamma(-1/2)=-2\sqrt{\pi}$ and $\Gamma(-3/2)=4\sqrt{\pi}/3$, we find the following.

\begin{lem} \label{lem:ffRS12second}
The functions $\ffRS_1$ and $\ffRS_2$ are analytic on $\mathcal R$ and satisfy
\begin{equation*}\begin{split}
	\ffRS_1(z) = \frac1{\sqrt{2\pi}} \Li_{3/2}(z) - \frac1{\sqrt{2\pi}} \Li_{5/2}(z) 
	\quad \text{and}\quad 
	\ffRS_2(z)=   - \frac1{\sqrt{2\pi}} \Li_{3/2}(z) 
	\quad \text{for $z\in \RS_2$.} 
\end{split} \end{equation*}
In addition, the functions $\ffRS_1$ and $\ffRS_2$ have continuous boundary values $\ffRS_1(z\pm i0)$ and $\ffRS_2(z\pm i0)$ for $z$ on the interval $(-\infty,0)$ in the boundary of $\mathcal R_1$.
\end{lem}

\subsection{Extension of the function $\EE$} \label{sec:EextRS}

We now consider the function 
$$
	\EE(z):= e^{2B(z) -2Q(u_0(z))}.
$$
Changing the variables in the definition \eqref{def:Q_function} of $Q$ and using $u'_0(t)=-1/(tu_0(t))$,
$$
	Q(u_0(z))=\sqrt{\frac{2}{\pi}}\int_{-\infty}^{u_0(z)}\Li_{1/2}(e^{-y^2/2})dy
	= - \sqrt{\frac{2}{\pi}}\int_0^{z} \frac{\Li_{1/2}(t)}{t u_0(t) } dt \quad \text{for $0<|z|<1$ with $z\notin(-1,0)$,} 
$$
where the integration contour for the second integral is in the same domain as $z$. 
Combined with the definition \eqref{def:functions_A} of $B$, we have 
\begin{equation}\label{eq:BmQf}
	2B(z)-2Q(u_0(z)) 
	= \frac{1}{2\pi}\int_0^{z} \frac{\Li_{1/2}(t)}{t} \left(\Li_{1/2}(t)+\frac{2^{5/2}\sqrt{\pi}}{u_0(t)}\right)dt 
\end{equation}
for $0<|z|<1$ with $z\notin(-1,0)$.

Define the function 
\begin{equation}\label{eq:BQexpRS}
	\ggRS(z):= \frac{\LRS_{1/2}(z)}{2\pi z} \left(\LRS_{1/2}(z)+\frac{2^{5/2}\sqrt{\pi}}{\URS_0(z)}\right) \quad \text{for } z\in \RS.
\end{equation}
On the first sheet $\mathcal R_1$ this function agrees with the integrand of \eqref{eq:BmQf}. 
It is analytic on $\RS$ except possibly at the branch point $z=1$ (note that $z=0$ is not on $\RS$). 

We look at the limit as $z\to 1$. Using \eqref{eq:expansion_polylog} and \eqref{eq:LRSdef}, we find that 
\begin{equation} \label{eq:ggRSasn1}
	\ggRS(z)= - \frac{3}{u_0(z)^2} + \Boh(u_0(z)^{-1}) = \frac{3}{2(z-1)} + \Boh((z-1)^{-1/2}) 
\end{equation}
as $z\to 1$. In terms of the local coordinates $z-1=w^2$, 
$$
	\ggRS(z) dz = \left( \frac{3}{w} + \Boh(1) \right) dw, 
$$
which implies that $\res_{z=1} \ggRS(z)dz =3$. 
Due to the pole of $\ggRS$, the integral $\int_{0^{(1)}}^z \ggRS(t) dt$ depends on the choice of the path on $\RS$. 
However, since the residue of $\ggRS$ at the pole is an integer, $\exp\left(\int_{0^{(1)}}^z \ggRS(t) dt\right)$ does not depend on the choice of path.

\begin{lem}\label{lem:EERSzero}
The function
\begin{equation*}
	\EERS(z):= \exp \left( \int_{0^{(1)}}^z \ggRS(t) dt \right) = \exp \left( \int_{0^{(1)}}^z \frac{\LRS_{1/2}(z)}{2\pi z} \left(\LRS_{1/2}(z)+\frac{2^{5/2}\sqrt{\pi}}{\URS_0(z)}\right) dt \right) 
\end{equation*}
is well defined and analytic on $\RS$.
We have $\EERS(z)= E(z)$ on $z\in \RS_1$ satisfying $0<|z|<1$. 
Furthermore, $\EERS$ has a zero of order $3$ at the branch point $z=1$. 
Finally, $\EERS(z)$ has continuous boundary values $\EERS(z\pm i0)$ for $z$ on the interval $(-\infty,0)$ in the boundary of $\mathcal R_1$.
\end{lem}

\begin{proof}
For $z\neq 1$, 
the value of $\EERS(z)$ does not depend on the choice of path of the integral on the Riemann surface since the residue of $\ggRS$ at the pole $1$ is an integer. Being an integral of an analytic function, it is analytic possibly except at $z=1$. 
As $z\to 1$, we have from \eqref{eq:ggRSasn1} 
$$
\int_{0^{(1)}}^z \ggRS(t) dt  = \frac{3}{2}\log(z-1)+\Boh(1), 
$$
which implies that $\EERS(z) = (z-1)^{3/2} f(z) = w^3 f(1+w^2)$ for an analytic function $f$ and the local coordinates $z-1=w^2$. 
This analysis shows that $\EERS$ is analytic near $z=1$ and it has a zero of order $3$ at $z=1$.

The continuity of the boundary values of $\EERS$ follow from continuity of the boundary values for $\Li_{1/2}$ and $u_0$ away from $z=0$.
\end{proof}

\subsection{Behavior of the function $\EERS(z)$ near $z=0$ on the second sheet}

We conclude this section with the behavior of $\EERS(z)$ as $z\to 0^{(2)}$ on the second sheet $\RS_2$. 
The following result will be used in the proof of Lemma~\ref{lem:integral_gamma5} in the next section. 

\begin{lem}\label{prop:local_expansion_prefactor_gaussian}
We have 
$$
	\EERS(z) =- 4u_0(z)^4 (1+\Boh(z^2)) \quad \text{as $z\to 0^{(2)}$.} 
$$
\end{lem}

\begin{proof}
Fix a small real number $\epsilon\in (0,1)$. 
We evaluate the exponent $\int_{0^{(1)}}^z \ggRS(t) dt$ of $\EERS$ using a specific path that consists of several pieces: 
\begin{equation*}
	\int_{0^{(1)}}^z \ggRS(t) dt = \int_{0^{(1)}}^{\epsilon^{(1)}} \ggRS(t) dt + \int_{\epsilon^{(1)}}^{(1-\epsilon)^{(1)}} \ggRS(t) dt+\int_{(1-\epsilon)^{(1)}}^{(1-\epsilon)^{(2)}} \ggRS(t) dt+\int_{(1-\epsilon)^{(2)}}^{\epsilon^{(2)}} \ggRS(t) dt +\int_{\epsilon^{(2)}}^z \ggRS(t) dt.
\end{equation*}
The contours for the second and fourth integrals are straight line segments. 
The first and last integrals are over contours such that $|t|$ stays small. 
The third integral is over a contour that satisfies that $|t-1|$ stays small. 
Note that the first four integrals do not depend on $z$. 

Consider the third integral. By \eqref{eq:BQexpRS} and using the local coordinates $1-t=w^2$, 
\begin{equation*}
	\lim_{\epsilon\to 0} \int_{(1-\epsilon)^{(1)}}^{(1-\epsilon)^{(2)}} \ggRS(t) dt = \lim_{\epsilon\to 0}  \int_{\sqrt{\epsilon}}^{-\sqrt{\epsilon}}  \frac{3}{w}  dw = \pm 3 \pi i, 
\end{equation*}
where the sign depends on the direction that the path goes around the branch point $z=1$. Since $\EERS(z)$ is the exponential of the integral and $e^{3\pi i}= e^{-3\pi i}=-1$, both signs give the same contribution to $\EERS(z)$. 

Now consider the first and the last integrals. From \eqref{eq:BQexpRS} and \eqref{eq:LRSdef}, 
\begin{equation} \label{eq:ssRSsh1}
	\ggRS(z)= \frac{1}{2\pi z}  \Li_{1/2}(z)  \left( \Li_{1/2}(z) +  \frac{2^{5/2}\sqrt{\pi}}{u_0(z)}  \right) \quad \text{for $z\in \RS_1$}
\end{equation}
and 
\begin{equation}\label{eq:ssRSsh2}
	\ggRS(z)= \frac{1}{2\pi z} \left( \Li_{1/2}(z) +  \frac{2^{3/2}\sqrt{\pi}}{u_0(z)}  \right) \left( \Li_{1/2}(z) -  \frac{2^{3/2}\sqrt{\pi}}{u_0(z)}  \right)
	 \quad \text{for $z\in \RS_2$.}
\end{equation}
Since $\Li_{1/2}(z)= z+ O(z^2)$ (see \eqref{def:series_polylog}), 
\begin{equation*}
	\ggRS(z)
	= -\frac2{\sqrt{\pi}(-\log z)^{1/2}} + \Boh(1)\quad \text{as } z\to 0^{(1)}
\end{equation*}
and 
$$
	\ggRS(z)=- \frac{4}{z u_0(z)^2} + \frac{\Li_{1/2}(z)^2}{2\pi z} =\frac{2}{z\log z}+\Boh(z) \quad\text{as } z\to 0^{(2)}.
$$
Hence, 
\begin{equation*}
	\lim_{\epsilon\to 0} \int_{0^{(1)}}^{\epsilon^{(1)}} \ggRS(t) dt
	= -\frac{2}{\sqrt{\pi}} \lim_{\epsilon\to 0}  \int_0^{\epsilon} \frac{1}{(-\log z)^{1/2}} dz  = 0, 
\end{equation*}
and for small enough $|z|$ and $\epsilon$, 
$$
	\int_{\epsilon^{(2)}}^z \ggRS(t) dt = 2\int_{\epsilon}^z \frac{dt}{t\log t}+ \Boh(z^2) +\Boh(\epsilon^2)
	= 2\log(-\log z)-2\log(-\log \epsilon)+ \Boh(z^2) +\Boh(\epsilon^2) .
$$

Finally, for the second and the fourth integrals, we have from \eqref{eq:ssRSsh1} and \eqref{eq:ssRSsh2}, 
$$
	 \int_{\epsilon^{(1)}}^{(1-\epsilon)^{(1)}} \ggRS(t) dt+\int_{(1-\epsilon)^{(2)}}^{\epsilon^{(2)}} \ggRS(t) dt 
	 = \frac{2^{3/2}}{\sqrt{\pi}}\int_\epsilon^{1-\epsilon} \left(\Li_{1/2}(t)+\frac{\sqrt{2\pi}}{u_0(t)}\right)\frac{dt}{tu_0(t)} .
$$
A direction computation shows that 
\begin{equation*} \begin{split}
	 \frac{2^{3/2}}{\sqrt{\pi}}\int_\epsilon^{1-\epsilon} \frac{\sqrt{2\pi}}{t u_0(t)^2} dt
	 &= -2 \int_{\epsilon}^{1-\epsilon} \frac{dt}{t\log t} = -2 \log(-\log(1-\epsilon))+ 2\log (-\log \epsilon)) .
\end{split} \end{equation*} 
We postpone the computation of the remaining part of the integral. 

Combining all together and taking $\epsilon\to 0$ first, we find that 
$$
	\EERS (z) =-  (-\log z)^2 (1+\Boh(z^2)) e^{C} 
$$
as $z\to 0^{(2)}$, where
$$
	C= \lim_{\epsilon\to 0} \left[ \frac{2^{3/2}}{\sqrt{\pi}}\int_\varepsilon^{1-\varepsilon} \Li_{1/2}(t) \frac{dt}{tu_0(t)} - 2 \log(-\log(1-\epsilon)) \right].
$$

We now find the value of $C$ using the lemma below. 
We may replace $\log(-\log(1-\epsilon))$ in the integral in the formula by $\log (\epsilon)$, and replace $\epsilon$ in the lower limit of the integral by $0$ because the integral is convergent at $t=0$. After that we change the variable $t$ to $u_0(t)= -(-2\log t)^{1/2}$ and find 
$$
	C= \lim_{\epsilon\to 0} \left[ - \frac{2^{3/2}}{\sqrt{\pi}}\int_{-\infty}^{-(-2\log(1-\varepsilon))^{1/2}} \Li_{1/2}(e^{-t^2/2}) dt - 2 \log \epsilon \right]. 
$$
Since $\Li_{1/2}(e^{-t^2/2})=\Boh(t^{-1})$ as $t\to 0$, which follows from \eqref{eq:expansion_polylog}, we may replace the upper limit of the integral by $-\sqrt{2\epsilon}$. 
Changing $\sqrt{2\epsilon}=\delta$, and using the evenness of the integrand, we find that $C= 4\log 2$ from the lemma below. This completes the proof. 
\end{proof}

\begin{lem} We have 
$$
	\lim_{\delta\to 0^+}\int_{\delta}^\infty \Li_{1/2}(e^{-t^2/2})dt+\sqrt{2\pi} \log(\delta)=- \sqrt{\frac{\pi}{2}}\log 2.
$$
\end{lem}

\begin{proof}
Using the power series formula \eqref{def:series_polylog} for the polylog and changing variables, 
\begin{equation*}
\begin{split}
	\int_{\delta}^\infty \Li_{1/2}(e^{-t^2/2}) dt
	=\sum_{k= 1}^\infty \frac{1}{k}\int_{\sqrt{k}\delta}^\infty e^{-t^2/2}dt
	= \int_\delta^\infty e^{-t^2/2}h\left(\left[ t^2\delta^{-2}\right]\right)dt
\end{split}
\end{equation*}
where $h(n):=1+\frac12+\cdots+\frac1{n}$ is the $n$-th partial sum of the harmonic series. 
We have 
\begin{equation*}
	h(n)=\log n +\gamma +\epsilon_n ,
\end{equation*}
where $\gamma$ is the Euler-Mascheroni constant, and $\epsilon_n= O(n^{-1})$ as $n\to \infty$. 
Set $\tilde{\epsilon}_y:=\epsilon_{[y]}+\log[y]-\log y$ and note that there is a constant $C>0$ such that $|\tilde{\epsilon}_y|\le \frac{C}{y}$ for all $y\ge 1$. Hence, 
$$
	\int_{\delta}^\infty \Li_{1/2}(e^{-t^2/2}) dt = \int_{\delta}^\infty e^{-t^2/2} \left( \log (t^2 \delta^{-2})  +\gamma +\tilde\epsilon_{t^2\delta^{-2}} \right)  dt.
$$
Note that
$$
	\int_{0}^\infty e^{-t^2/2} d t= \sqrt{ \frac{\pi}{2}} \qquad\text{and}\qquad
	\int_{0}^{\infty}e^{-t^2/2} \log (t^2)dt =-\sqrt{ \frac{\pi}{2}} (\gamma+\log 2).
$$
On the other hand, 
\begin{equation*}
	\left| \int_{\delta}^{\infty}e^{-t^2/2} \tilde\epsilon_{t^2\delta^{-2}} dt \right|
	\le C\delta^2 \left( \int_\delta^1 \frac1{t^2} dt + \int_1^\infty e^{-t^2/2} dt \right) = O(\delta). 
\end{equation*}
Combining the above equations, we obtain the result. 
\end{proof}

\section{Large time limit}\label{sec:gaussian}

Recall the distribution function $F$ in \eqref{def:limiting_distribution}. In this section we consider the large $\tau$ limit of $F(x_\tau; \tau, \gamma)$ with
\begin{equation}\label{eq:x_scaling_large_tau}
	x_\tau= -\tau +\frac{\pi^{1/4}}{\sqrt{2}}\rx \tau^{1/2}, \qquad \rx\in \R \; \mbox{fixed}, 
\end{equation}
and prove Theorem \ref{thm:gaussian_degeneration}. 
This section is split into four subsections. 
In the first subsection, we analyze the asymptotics of the Fredholm determinant in the integrand of \eqref{def:limiting_distribution} and show that all but the first two terms of the Fredholm series expansion are exponentially small. 
In the second subsection, we show that the main contribution to $F(x_\tau; \tau, \gamma)$ comes from the integral involving the trace of the operator $\mb K_z$ and then show that the integral can be expressed as a single integral involving polylog functions. 
In the third subsection we evaluate the integral using the method of steepest descent and, finally, in the fourth subsection we combine all the ingredients to prove Theorem \ref{thm:gaussian_degeneration}.

\subsection{Asymptotic analysis of the Fredholm determinant}

The formula \eqref{def:limiting_distribution} of $F(x_\tau; \tau, \gamma)$ is given by an integral of the variable $z\in\D\setminus\{0\}$ where $\D$ denotes the open unit disk. 
It is convenient to write
$$
	z=e^{\log|z|} e^{i\theta_0(z)/2}, \quad \theta=\theta_0(z)=\begin{dcases}
	2\arg z, & \arg z \in [0,\pi],\\
	-2(2\pi -\arg z),& \arg z\in (\pi,2\pi).
	\end{dcases}
$$
Note that $\theta_0(z)\in (-2\pi,2\pi]$. Using the above notation, the set $\mathcal S_-$ (recall \eqref{def:bethe_root_set}), which is the set of solutions of $e^{-\xi^2/2}=z$ with $\re \xi<0$, can be enumerated as $\mathcal S_-=\{u_k\}_{k\in \Z}$ where 
\begin{equation*}
 	u_k= u_k(z) 
 	=-(-2\log|z|+i\theta_k)^{1/2},\qquad \theta_k:=-\theta_0(z)+4\pi k.
\end{equation*}
In particular
$$
	u_0(z)= -(-2\log|z|+i\theta_0)^{1/2} =  -(-2\log z)^{1/2}
$$
agrees with the previous definition in \eqref{def:coordinate_u0}.
We denote the real part and the imaginary part of $u_k$ by $x_k$ and $y_k$: 
\begin{equation*}
	u_k= u_k(z)=x_k+iy_k = x_k(z)+iy_k(z) .
\end{equation*}
These values $x_k$ and $y_k$ satisfy the relations
\begin{equation}\label{eq:relations_real_imaginary_u}
	x_k^2-y_k^2=-2\log|z|, \qquad 2x_ky_k=\theta_k.
\end{equation}
Solving for $x_k^2$, we have 
\begin{equation}\label{eq:identity_xk}
	\left(x_k(z)\right)^2=-\log|z| +\sqrt{(\log|z|)^2+\theta_k^2/4}.
\end{equation}
Also note that 
\begin{equation}\label{eq:GL_ineq_thetak}
	\theta_k^2\ge 4\left(2k-\frac{\theta_0}{2\pi}\right)^2\pi^2\geq  4(2|k|-1)^2\pi^2 \qquad \text{for $k\neq 0$.} 
\end{equation}

Recall the formula \eqref{def:phi_function} of the function $\Phi$ appearing in \eqref{eq:initial_kernel}. With the change $x\mapsto x_\tau$ in \eqref{eq:x_scaling_large_tau}, it becomes 
\begin{equation} \label{eq:leading_term_Phi}
	\Phi_\tau(\xi) := \Phi(\xi; z, x_\tau, \tau) = -\tau \left( \frac13 \xi^3+\xi\right) + \tau^{1/2} \frac{\pi^{1/4} \rx}{\sqrt{2}} \xi - Q(\xi)
\end{equation}
where $Q$ is defined in \eqref{def:Q_function}. Using~\eqref{eq:relations_real_imaginary_u}, we write
\begin{equation}\label{eq:relation_phi_p}
\re \Phi_\tau(u_k)=-\frac{\tau}{3}p_z(x_k)+\tau^{1/2} \frac{\pi^{1/4}\rx }{\sqrt{2}}x_k-\re Q(u_k)
\end{equation}
for $u_k=x_k+iy_k\in \mathcal S_-$, where
$$
p_z(w):=-2w^3+3w(1-2\log|z|). 
$$

In the next few results, we estimate $\re \Phi_\tau(u)$ for $u$ on $\mathcal S_-$. These results involve constants which are given numerically to several digits and will be later combined to estimate a Fredholm determinant from its series. The exact value of these constants, as they appear here, could be relaxed at the cost of estimating the very first few terms of the series for the Fredholm determinant separately, taking advantage of the alternating signs that appear when computing determinants. However, we opt for giving these precise values of the constants and use only the first term of the series.

\begin{lem}
There is a constant $\delta\in (0,1)$ such that  
\begin{equation*} 
	p_z(x_k(z))\ge \begin{dcases}
	- 1.42,&k=0,\\
	5.81|k|^{3/2}, & k\ne 0,\\
	\end{dcases}
\end{equation*}
for $1-\delta\le |z|\le 1$. 
\end{lem}

\begin{proof}
From the definition of $\mathcal S_-$ we have $x_k(z)<0$. Thus, for all $|z|$ that satisfies the inequality $\log |z|\ge \frac12(1-3\cdot\sqrt[3]{2\cdot 0.71\cdot 0.71})\approx -0.00136$, 
\beqq
	p_z(x_k(z))\ge -2x_k(z)^3 +3\sqrt[3]{2\cdot 0.71\cdot 0.71} x_k(z)\ge -1.42 . 
\eeqq
Now, consider $k\ne 0$. 
Using \eqref{eq:identity_xk}, recalling $x_k(z)<0$,and combining with \eqref{eq:GL_ineq_thetak}, 
\begin{equation*}
	x_k(z)\le 
-\sqrt{-\log|z| +(2|k|-1)\pi}, \qquad k\ne 0.
\end{equation*}
The function $p_z(w)$ decreases in $w\in (-\infty, w_c)$ with $w_c=-\frac{1}{2}\sqrt{1-2\log|z|}$. 
For $k\neq 0$, we see that $-\sqrt{-\log|z| +(2|k|-1)\pi}<w_c$, and hence, 
\beqq
\begin{split}
	p_z(x_k(z))&\ge p_k(-\sqrt{-\log|z| +(2|k|-1)\pi})\\
	&=\left(4\log |z| + (4|k|-2)\pi-3\right)\sqrt{-\log|z|+(2|k|-1)\pi}\\
	&\ge (4|k|-2.9565)\sqrt{2|k|-1}\pi^{3/2} 
	\geq 1.0435\pi^{3/2}|k|^{3/2}
\end{split}
\eeqq
provided $\log|z|>\frac\pi4(3/\pi -0.9565)\approx-0.0012$. This completes the proof. 
\end{proof}

\begin{cor}\label{prop:estimate_real_phi_uk_u0}
There exists a constant $\delta\in (0,1)$ such that for every $\delta_1\in (0,\delta)$, we can find $\tau_0>0$ so that 
\begin{equation*}
\re \Phi_\tau( u_k(z)) \le \begin{dcases}
0.4734\tau, & k=0,\\
-1.9366|k|^{3/2}\tau, & k\ne 0,
\end{dcases} 
\end{equation*}
for all $1-\delta\le |z|\le 1-\delta_1$ and $\tau\ge\tau_0$. 
\end{cor}

\begin{proof}
We use the formula \eqref{eq:relation_phi_p} and the last lemma. 
By \eqref{eq:identity_xk} and \eqref{eq:GL_ineq_thetak}, for every $\epsilon\in (0,1)$, there is a positive constant $C>0$ such that $|x_k(z)|\le C|k|^{1/2}$ for all $k\neq 0$ and $\epsilon\le |z|\le 1$. Using this and the fact that $\rx$ is a fixed constant, we find that  the term $\frac{\pi^{1/4}\rx \tau^{1/2}}{\sqrt{2}}x_k$ of the formula  \eqref{eq:relation_phi_p} is $\Boh(|k|^{1/2} \tau^{1/2})$ uniformly in $\epsilon\le |z|\le 1$. The remaining term of the formula, $\re Q(u_k(z))$, is uniformly bounded in compacts of $|z|<1$ as can be seen from the formula~\eqref{eq:alternative_definition_q}. Hence, we obtain the result for $k\neq 0$. The case $k=0$ is similar.  
\end{proof}

We also need estimates on the points in $\mathcal S_-$. 

\begin{lem} \label{lem:uuplusu}
For every fixed $\delta_1\in (0,1)$, there exists $c>0$ such that 
\begin{equation*}
	|u_k(z)|\geq c\sqrt{|k|+1}  \quad\text{and} \quad |u_j(z)+u_k(z)|\geq c\sqrt{|j|+|k|+2 } 
\end{equation*}
for all $0<|z|<1-\delta_1$ and $k,j\in \Z$.
\end{lem}

\begin{proof}
From~\eqref{eq:identity_xk}, $x_k^2\ge -\log|z| +|\theta_k|/2> -\log(1-\delta_1) +|\theta_k/2|>0$. 
From \eqref{eq:GL_ineq_thetak}, $|\theta_k|\ge 2(2|k|-1)\pi$ for $k\ne 0$. 
Thus, there exists $c>0$ such that $|u_k|^2\ge x_k^2 \ge c^2(|k|+1)$ for all $k$. 
This implies the first inequality. 
The second inequality follows from the first inequality since $x_j, x_k<0$, and hence 
\begin{equation*}
	|u_j+u_k| \geq  -x_j-x_k\geq c\sqrt{|j|+1}+c\sqrt{|k|+1}\geq c\sqrt{|j|+|k|+2}.
\end{equation*}
\end{proof}

The kernel of the operator $\mb K_z$ is given by \eqref{eq:initial_kernel}. 
With the change $x\mapsto x_\tau$ in \eqref{eq:x_scaling_large_tau}, the kernel becomes 
\begin{equation*}
	\bm K_z(\xi_1,\xi_2)=\sum_{u \in S_-} \frac{e^{\Phi_\tau(\xi_1)+\Phi_\tau(u)+\frac{\gamma}{2}(\xi_1^2-u^2)}}{\xi_1u (\xi_1+u)(u+\xi_2)}
	\quad \text{for $\xi_1, \xi_2\in \mathcal S_-$,}
\end{equation*}
where $\Phi_\tau$ is defined in \eqref{eq:leading_term_Phi}. 
We now estimate this kernel. 

\begin{lem}\label{lem:kernel_sum_1}
There is a constant $\delta\in(0,1)$ such that for every $\delta_1\in (0,\delta)$, 
we can find $\tau_0>0$ and $C>0$ so that 
$$
	\sum_{k=1}^N |\bm K_z(\xi, u_{\ell_k})|^2\leq ( C \log N ) e^{2\re \Phi_\tau(\xi)+0.9468\tau}
$$
for all $\xi\in \mathcal S_-$, for all $N\ge 2$ and distinct integers $\ell_1,\hdots , \ell_N$, and for all $\tau\geq \tau_0$ and $1-\delta\le |z|\le 1-\delta_1$. 
\end{lem}

\begin{proof}
Let $\delta,\delta_1$ and $\tau_0$ be the constants from Corollary \ref{prop:estimate_real_phi_uk_u0}.
Every $\xi\in \mathcal S_-$ satisfies $e^{- \xi^2/2}= z$, and hence, $|e^{-\gamma \xi^2 /2}|= |z|^\gamma$. 
Thus, using Lemma \ref{lem:uuplusu}, 
\begin{equation} \label{eq:Kinltakak}
	|\bm K_z(\xi,u)| \leq \frac{e^{\re \Phi_\tau (\xi)}}{c^3}f(u) \quad \text{where} \quad
	f(u) := \sum_{j\in \Z}\frac{e^{\re \Phi_\tau(u_j)}}{|u_j+u|}
\end{equation}
for $\xi, u\in \mathcal S_-$, where the constant $c>0$ is from Lemma \ref{lem:uuplusu}. 
For $u=u_\ell$, Corollary \ref{prop:estimate_real_phi_uk_u0} and Lemma~\ref{lem:uuplusu} imply that 
\begin{equation*}
	f(u_\ell)  
	 \leq \frac{e^{0.4734\tau}}{c \sqrt{|\ell|+2}} +\sum_{j\ne 0} \frac{e^{-1.9366|j|^{3/2}\tau}}{c\sqrt{|\ell|+|j|+2}  }.
\end{equation*}
for all $\tau\ge \tau_0$.
Adjusting the value of $\tau_0$ if needed, the sum is smaller than the single term, and hence, 
\begin{equation*}
	f(u_\ell)\le \frac{2e^{0.4734\tau}}{c \sqrt{|\ell|+2}}.
\end{equation*}
Inserting the above inequalities into \eqref{eq:Kinltakak} and noting that $\sum_{k=1}^N \frac1{|\ell|+2}\le \log N$
since $\ell_1, \cdots, \ell_N$ are distinct integers, we complete the proof. 
\end{proof}

We are now ready to prove the main result of this subsection. 

\begin{prop}\label{thm:localization_fred_det}
Change $x\mapsto x_\tau$ as \eqref{eq:x_scaling_large_tau}. 
There is a constant $\delta\in(0,1)$ such that for every $\delta_1\in(0,\delta)$,
$$
	\det(\mb I-\mb K_z)=1-\Tr \mb K_z + \Boh(e^{-0.51\tau}) 
	\quad \text{as } \tau \to \infty,
$$
uniformly for $1-\delta\le |z|\le 1-\delta_1$. 
\end{prop}

\begin{proof}
From the series expansion of Fredholm determinants, 
\begin{equation*}
 	\det(\mb I-\mb K_z)=1-\Tr \mb K_z +\sum_{N=2}^\infty \frac{(-1)^N}{N!}D_N
	\quad \text{with}
	\quad D_N:= \sideset{}{^*} \sum_{\ell_1,\hdots,\ell_N \in \Z} \det (\bm K_z(u_{\ell_j},u_{\ell_k}))_{j,k=1}^N \end{equation*}
where the superscript $^*$ means that the sum is taken over $\ell_j$'s that are distinct. 
For distinct $\ell_j$'s, the Hadamard's inequality and Lemma~\ref{lem:kernel_sum_1} give us that
\begin{equation*}
	\left| \det (\bm K_z(u_{\ell_j},u_{\ell_k}))_{j,k=1}^N \right| 
	\leq \prod_{j=1}^N \left(\sum_{k=1}^N |\bm K_z(u_{\ell_j},u_{\ell_k})|^2 \right)^{1/2}
	\leq  (C\log N)^{N/2} e^{0.4734\tau N} \prod_{j=1}^N e^{\re\Phi_\tau (u_{\ell_j})} .
\end{equation*}
Hence, 
$$
	|D_N| \leq (C\log N)^{N/2} e^{0.4734\tau N} \sideset{}{^*} \sum_{\ell_1,\hdots,\ell_N \in \Z}  \prod_{j=1}^N e^{\re\Phi_\tau (u_{\ell_j})}.
$$
Considering the case when one of the indices $\ell_j$ is zero separately and using a symmetry, 
\begin{align*}
	\sideset{}{^*} \sum_{\ell_1,\hdots,\ell_N \in \Z}  \prod_{j=1}^N e^{\re\Phi_\tau (u_{\ell_j})}
	& = \sideset{}{^*} \sum_{\ell_1,\hdots,\ell_N \neq 0}  \prod_{j=1}^N e^{\re\Phi_\tau (u_{\ell_j})} 
	+ N e^{\re\Phi_\tau (u_{0})} \sideset{}{^*} \sum_{\ell_2,\hdots,\ell_N\neq 0}  \prod_{j=2}^N e^{\re\Phi_\tau (u_{\ell_j})}\\
	& \le \bigg( \sum_{\ell\neq 0} e^{\re\Phi_\tau (u_{\ell})} \bigg)^N
	+ N e^{\re\Phi_\tau (u_{0})} \bigg(  \sum_{\ell\neq 0} e^{\re\Phi_\tau (u_{\ell})} \bigg)^{N-1}.
\end{align*}
Now, Corollary \ref{prop:estimate_real_phi_uk_u0} implies that 
$$
	\sum_{\ell\in \Z\setminus \{0\}}e^{\re\Phi_\tau(u_{\ell})}\leq 2\sum_{\ell=1}^\infty e^{-1.9366\ell^{3/2}\tau }
	\le  \frac{2 e^{-1.9366\tau}}{1-e^{-1.9366\tau}}  
	\quad \text{and}\quad 
e^{\re\Phi_\tau (u_{0})}\le e^{0.4734\tau}
$$
for all large enough $\tau$. Therefore, 
$$
\sideset{}{^*} \sum_{\ell_1,\hdots,\ell_N \in \Z}  \prod_{j=1}^N e^{\re\Phi_\tau (u_{\ell_j})}
 \leq C_1^N e^{-1.9366\tau N}e^{2.41\tau}
$$
for a new constant $C_1>0$. Thus, adjusting the constant $\tau_0$ if needed, we find that for a yet new constant $C_2>0$
$$
	|D_N|	\le  (\sqrt{C_2\log N})^N e^{-1.4632\tau N + 2.41\tau}
$$
for $\tau\ge \tau_0$ and $N\ge 2$. From this and using $N!\geq (N/e)^N$, we conclude that
\begin{equation*}
	\left| \sum_{N=2}^\infty \frac{(-1)^N}{N!}D_N \right|
	\le \sum_{N=2}^\infty  \left( \frac{e\sqrt{C_2\log N}}{N} \right)^N e^{-1.4632\tau N + 2.41\tau}
	\le C_3 \sum_{N=2}^\infty e^{-1.4632\tau N + 2.41\tau}
\end{equation*}
for a new constant $C_3>0$, and hence we find that the sum is $\Boh(e^{-0.51\tau})$. 
\end{proof}

\subsection{From a Fredholm determinant integral to a polylog integral}

Proposition~\ref{thm:localization_fred_det} implies the following result for the distribution function.

\begin{prop}\label{prop:Flgtu}
We have
\begin{equation*}
	F\left(x_\tau;\tau,\gamma\right)=1-\oint_{|z|=R} e^{x_\tau A_1(z)+\tau A_2(z)+2B(z)} \Tr(\mb K_z) \; \frac{dz}{2\pi i z}+\Boh(e^{-0.002 \tau})
\end{equation*}
as $\tau\to \infty$ for any $R\in (0,1)$.  
\end{prop}

\begin{proof}
Since the integrand in the result is analytic in $0<|z|<1$, it is enough to prove the statement for $R=1-\delta$ where $\delta>0$ is the constant from Proposition~\ref{thm:localization_fred_det}. 
We insert the result of Proposition~\ref{thm:localization_fred_det} into the formula \eqref{def:limiting_distribution} of $F\left(x_\tau;\tau,\gamma\right)$.
Since $A_1(z), A_2(z)$, and $B(z)$ are analytic in the unit disc, and they are all $0$ at $z=0$, by the residue theorem, we obtain the term $1$. 
We now consider the error term.
Recalling \eqref{eq:x_scaling_large_tau},
$$
	\max_{|z|=R} \re\left(x_\tau A_1(z)+\tau A_2(z)+2B(z)\right)=\frac{\tau}{\sqrt{2\pi}}  \max_{|z|=R}\re( \Li_{3/2}(z)-\Li_{5/2}(z))  +\Boh(\tau^{1/2}).
$$
Using the series representation \eqref{def:series_polylog} of polylog functions,  
$$
	|\Li_{3/2}(z)-\Li_{5/2}(z)|\leq \sum_{k=1}^\infty |z|^k(k^{-3/2}-k^{-5/2})\le \zeta(3/2)-\zeta(5/2)<1.271. 
$$
Here $\zeta$ is the Riemann zeta function.
Hence, 
\begin{equation}\label{eq:estimate_factor}
	\max_{|z|=R} \re\left(x_\tau A_1(z)+\tau A_2(z)+2B(z)\right)<0.5071\tau
\end{equation}
for all large enough $\tau$, and
\begin{equation*}
	\oint_{|z|=R} e^{ x_\tau A_1(z)+\tau A_2(z)+2B(z)} \Boh(e^{-0.51\tau}) dz = \Boh(e^{-0.002\tau})
\end{equation*}
and we obtain the result. 
\end{proof}

Since the kernel of $\mb K_z$ is given by a sum, $\Tr \mb K_z$ is a double sum. 
We show that the main contribution to the trace comes from a single term of the sum. 

\begin{lem}\label{prop:large_tau_estimate_trace}
With the change $x\mapsto x_\tau$ in \eqref{eq:x_scaling_large_tau},
there is a constant $\delta\in(0,1)$ such that for every $\delta_1\in(0,\delta)$,
$$
	\Tr \mb K_z = \frac{e^{2\Phi_\tau(u_0(z))}}{4u_0(z)^4} +\Boh(e^{-1.4632\tau})  \quad \text{as $\tau\to \infty$,}
$$
uniformly for $1-\delta\le |z|\le 1-\delta_1$.
\end{lem}

A direct calculation shows that $2\re \Phi_\tau(u_0(-1\pm i0))\approx -3.8795 \tau$, so the term singled out above is neglectible when compared with the error term for $z$ close to $-1$. However, the importance of the formula above comes from its uniformity: when $z$ moves away from $-1$, the formula is still valid, and for $z$ sufficiently away from $-1$ the contribution from the term $e^{2\Phi(u_0(z))}$ above becomes dominant over the error term. 

Later we will use the asymptotic formula above to compute the integral in Proposition~\ref{prop:Flgtu}. The major contribution to that integral will come from a critical point $z=z_c$ of the exponent of the integrand and, at this critical point, the contribution that will arise from the term $e^{2\Phi_\tau(u_0(z_c))}$ will indeed be larger than the error term.

Also, the asymptotic formula above has to be interpreted carefully when considering $z<0$, as the term $\Tr \mb K_z$ is analytic in $0<|z|<1$, but $u_0(z)$ is analytic only in $\C\setminus ((-\infty, 0]\cup [1, \infty))$. Nonetheless, it is easy to check from the definition that $u_0(z)$ and $\frac1{u_0(z)}$ are bounded in any compact subset of 
the open set $\{z  : \re z< 0\}$, including on the line segment $(-\infty, 0)$, and admit continuous boundary values when $z$ approaches $(-\infty,0)$ from above or below (although these boundary values do not coincide). Therefore, as will be done later without further mention, 
these asymptotics can be used to integrate $\Tr \mb T_z$ along any contour that intersects $(-\infty,0)$ only at its endpoints. 

\begin{proof}[Proof of Lemma~\ref{prop:large_tau_estimate_trace}.]
By the definition of the kernel, 
\begin{align*}
	\Tr \mb K_z=\sum_{u\in \mathcal S_-}\bm K_z(u,u) & =\sum_{k\in \Z}\sum_{j\in \Z}\frac{e^{\Phi_\tau(u_k)+\Phi_\tau(u_j)+\frac{\gamma}{2}(u_k^2-u_j^2)}}{u_k u_j (u_j+u_k)^2}	= \frac{e^{2\Phi_\tau(u_0)}}{4u_0^4}  + \text{REST}
\end{align*}
where we have set 
\begin{equation*}
	\text{REST}= \sum_{(k,j)\neq (0,0) }\frac{e^{\Phi_\tau(u_k)+\Phi_\tau(u_j)+\frac{\gamma}{2}(u_k^2-u_j^2)}}{u_k u_j (u_j+u_k)^2}.
\end{equation*}
We estimate \text{REST}. 
Let $\tau_0>0$ be the constant from Corollary \ref{prop:estimate_real_phi_uk_u0}.
Using $|e^{\frac{\gamma}2 u_k^2}|=|z|^\gamma$ and Lemma \ref{lem:uuplusu},  
\begin{equation*} \begin{split}
	|\text{REST}| 
	\le   \frac1{c^4}\left[ e^{\re \Phi_\tau(u_0)} \sum_{j\neq 0}e^{\re \Phi_\tau(u_j)}+\sum_{k\neq 0} e^{\re\Phi_\tau(u_k)} \sum_{j\in \Z} e^{\re\Phi_\tau(u_j)}\right] .
\end{split} \end{equation*} 
From Corollary \ref{prop:estimate_real_phi_uk_u0} we have $e^{\re \Phi_\tau(u_0)}\le e^{0.4734\tau}$, 
\begin{equation*}
	\sum_{k\neq 0} e^{\re\Phi_\tau(u_k)} \le 2 \sum_{k=1}^\infty e^{-1.9366|k|^{3/2} \tau}\le \frac{2 e^{-1.9366\tau}}{1-e^{-1.9366\tau}}
\end{equation*}
and
$$
	\sum_{j\in \Z} e^{\re\Phi_\tau(u_j)} \le e^{0.4734\tau}+ \frac{2 e^{-1.9366\tau}}{1-e^{-1.9366\tau}}
$$
for all $\tau\ge \tau_0$. Hence, $|\text{REST}| = \Boh(e^{-1.4632\tau})$ and we obtain the result. 
\end{proof}

We arrived at the main result of this subsection. To state it, we define the functions 
\begin{equation} \label{eq:ff12E} \begin{split}
	&\ff_1(z) := \frac1{\sqrt{2\pi}} \Li_{3/2}(z) - \frac1{\sqrt{2\pi}} \Li_{5/2}(z)  -2 u_0(z) - \frac23 u_0(z)^3, \\
	&\ff_2(z) :=   2u_0(z) - \frac1{\sqrt{2\pi}} \Li_{3/2}(z) ,  \\
	& \EE(z) := e^{2B(z)-2Q(u_0(z))}
\end{split} \end{equation}
for $z\in \C\setminus ((-\infty, 0]\cup [1, \infty))$, where $B$ and $Q$ are given in \eqref{def:functions_A} and \eqref{def:Q_function}. 

\begin{cor}
There is a constant $\epsilon\in (0,1)$ such that 
\begin{equation}\label{eq:relation_F_polylog_integral}
	F(x_\tau;\tau,\gamma)=1- \frac1{8\pi i} \int_{\Gamma} \frac{\EE(z) e^{\tau \ff_1(z)+ \tau^{1/2} \hat x \ff_2(z)}}{z u_0(z)^4} dz+\Boh(e^{-0.002 \tau})  \quad \text{with $\hat x:= \frac{\pi^{1/4}\rx}{\sqrt{2}}$}
\end{equation}
as $\tau \to \infty$, 
where $\Gamma$ is any simple contour in the domain $\C\setminus ((-\infty, 0]\cup [1, \infty))$ that starts at $a-i0$ and ends at $b+i0$, for any 
$a,b\in (-1-\epsilon, -1+\epsilon)$. 
\end{cor}
\begin{proof}
Observe that the integrand of \eqref{eq:relation_F_polylog_integral} is analytic in $\C\setminus ((-\infty, 0]\cup [1, \infty))$. 
Let $R$ for the contour in Proposition \ref{prop:Flgtu} be any number in $[1-\delta, 1)$ where $\delta$ is the constant from Lemma \ref{prop:large_tau_estimate_trace}.
We regard the contour as an arc that starts at $-R-i0$ and ends at $-R+i0$. 
We insert the result of Lemma~\ref{prop:large_tau_estimate_trace} into the formula of Proposition \ref{prop:Flgtu}. 
By the definitions, $e^{x_\tau A_1(z)+\tau A_2(z)+2B(z)}$ times $e^{2\Phi_\tau(u_0(z))}$ is $E(z) e^{\tau \ff_1(z)+ \tau^{1/2} \hat x \ff_2(z)}$, 
and hence we obtain the result when the contour is an arc that starts at $-R-i0$ and ends at $-R+i0$. 
Since the integrand of \eqref{eq:relation_F_polylog_integral} is analytic in $\C\setminus ((-\infty, 0]\cup [1, \infty))$, 
we can deform the contour to any other simple contour in the same domain but with the same end points $-R-i0$ and $-R+i0$. 

We now show that there is $\epsilon>0$ such that for any $-1-\epsilon< a_1< a_2< -1+\epsilon$, the integral over the line segment from $a_1-i0$ to $a_2 -i0$ or the line segments from $a_1+i0$ to $a_2+i0$ are exponentially small in $\tau$. This allows us to change the end points and still have the formula \eqref{eq:relation_F_polylog_integral} valid, thus concluding the proof.
Noting $u_0(-1\pm i0)= -\frac{1\mp i}{\sqrt{2}}\sqrt{2\pi}$ and using the numerical values $\Li_{3/2}(-1)=-0.7651\cdots$ and $\Li_{5/2}(-1)=-0.8671\cdots$, we see that
\begin{equation*}
	\re \ff_1(-1\pm i0) =  \frac{1}{\sqrt{2\pi}}\Li_{3/2}(-1) -\frac{1}{\sqrt{2\pi}}\Li_{5/2}(-1) +2\sqrt{\pi}  -\frac{4\pi\sqrt{\pi}}{3} 
	= - 3.8388\cdots. 
\end{equation*}
Thus, by continuity of the integrand up to the boundary $(-\infty, 0)\pm i0$, we find that the integrand is $\Boh (e^{-3\tau})$ in a neighborhood of $z=-1$, as we wanted.
\end{proof}

\subsection{Asymptotic analysis of polylog integral} \label{sec:asymppolyintg}

We now evaluate the integral on the right-hand side of \eqref{eq:relation_F_polylog_integral}  as $\tau\to \infty$ using the method of steepest descent.
It turns out that the main contribution to the integral comes from a point on the Riemann surface on which the integrand admits analytic extension.

Recall the Riemann surface $\RS$ introduced in Section \ref{subsec:riemann_surface}. 
Define the one-form
\begin{equation}\label{eq:oneform_omega}
	\omega_\tau = \frac{\EERS(z) e^{\tau \ffRS_1(z)+ \tau^{1/2} \hat x \ffRS_2(z)}}{8\pi i z \URS_0(z)^4} dz
\end{equation}
on $\RS$, 
where $\EERS, \ffRS_1, \ffRS_2, \URS_0$ are extensions of $\EE, \ff_1,\ff_2, u_0$ defined in Section \ref{section:polylogs}. 
Using this notation, the equation \eqref{eq:relation_F_polylog_integral} becomes 
\begin{equation}\label{eq:Finomega}
	F(x_\tau;\tau,\gamma)=1- \int_{\Gamma} \omega_\tau+\Boh(e^{-0.002 \tau}) 
\end{equation}
where $\Gamma$ is a contour in the first sheet $\RS_1$, which we now take to start at $-1-i0$ and end at $-1+i0$. 
We will see that as $\tau\to \infty$, the main contribution to the integral comes from a neighborhood of order $\tau^{-1/2}$ of the point $z=0^{(2)}$ on the second sheet $\RS_2$. 
We deform $\Gamma$ to a new contour $\Gamma'=  \Gamma_1 \cup \Gamma_2\cup \Gamma_3$ defined as follows (see Figure~\ref{fig:deformation_contour_riemann_surface}).
\begin{itemize}
\item $\Gamma_1=\gamma_1\cup \overline{\gamma_1}$ where $\gamma_1$ is a contour  that starts from the point $-1-i0$ in the lower half plane, moves in the lower half plane of $\RS_1$ until it hits a point on the branch cut on $(1,\infty)$, and then moves to the upper half plane in the second sheet $\RS_2$ until it ends at a point  in the upper imaginary axis. 
Concretely, we use the contour 
$$
	\gamma_1(t)=e^{-\frac{1}{2}(t^2-\pi)+ i \sqrt{\pi} t} \quad 
	\text{for $-\sqrt{\pi}\leq t\leq \sqrt{\pi}/2$, } 
$$
where the part for $-\sqrt{\pi}\le t\le 0$ lies on $\RS_1$ and the part for $0\le t\le \sqrt{\pi}/2$ lies on $\RS_2$. 
Note that $\gamma_1(-\sqrt{\pi})=-1-i0$ and $\gamma_1(0)=e^{\pi/2}>0$. Also, $\gamma_1(\sqrt{\pi}/2)= ie^{\frac{3\pi}{8}}$ on $\RS_2$. 

\item $\Gamma_2=\gamma_2\cup \overline{\gamma_2}$, where $\gamma_2$ is the line segment on the imaginary axis in $\RS_2$ joining the point $\gamma_2(\sqrt{\pi}/2)= ie^{3\pi/8} $ and the point $i\tau^{-5/12}$. 

\item $\Gamma_3$ is a contour in $\RS_2$ joining the point $i\tau^{-5/12}$ and $-i\tau^{-5/12}$. We take it to be the union of straight line segments with a semicircle, namely 
$$\Gamma_3=\{z= \tau^{-1/2} e^{i\theta} |-\pi/2\le \theta\le \pi/2 \} \cup\{ z=iy| \tau^{-1/2}\le |y|\le \tau^{-5/12}\},$$ 
oriented downwards. As we will show, this part gives the main contribution to the integral. 
\end{itemize}

In the above, the notation $\overline{\gamma_j}$ denotes the contour in $\RS$ whose intersection with $\RS_k$ is the complex conjugation of $\gamma_j\cap \RS_k$.

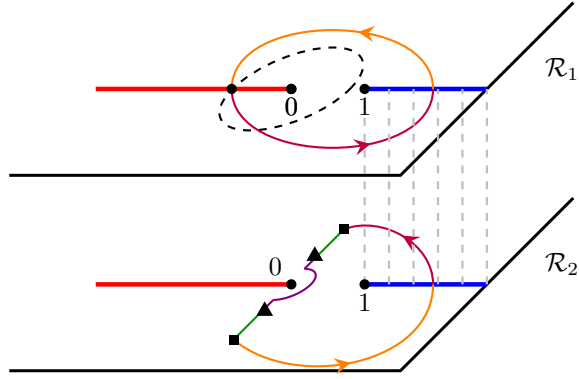
\begin{figure}\centering
\begin{tikzpicture}
[x  = {(1cm,0cm)},
                    y  = {(0.707107cm,0.707107cm )},
                    z  = {(0 cm,1 cm)},
                    scale = 0.65,
                    color = {black}]
                    
%
%
\begin{scope}[canvas is xy plane at z=0]
  \draw[white,->] (-5,0) -- (5,0) node (x-axis) [right] {$x$};
  \draw[white,->] (-5,0) -- (-5,2.8) node (y-axis) [above] {$y$};
\end{scope}
 
\begin{scope}[canvas is yz plane at x=-5]
  \draw[white,->] (0,-5) -- (0,3) node (z-axis) [above] {$z$};
\end{scope} 
%
%
%
\begin{scope}[canvas is xy plane at z=0]
\path[draw,line width=0.4mm] (-4,-2.5) to (4,-2.5) to [edge node={node [pos=0.4,right,shift={(4pt,0pt)}] {}}] (4,2.5);
%
%
\path [draw,line width=0.6mm,red] (-4,0) to [edge node={node [pos=0.5,above,shift={(2pt,-2pt)}] {}}] (0,0);
\path [draw,line width=0.6mm,blue] (1.5,0) to [edge node={node [pos=0.5,above,shift={(2pt,-2pt)}] {}}] (4,0);
%
%
\draw[black,thick,dashed] (0,0) circle[radius=1.2]; 
\draw[thick,purple,postaction={mid6 arrow={purple,scale=1.5}}] (-1.22,0) to [out=-90,in=-90] (1.55+1.35,0); 
\draw[thick,orange,postaction={mid4 arrow={orange,scale=1.5}}] (1.55+1.35,0) to [out=90,in=90] (-1.22,0); 

\end{scope}
%

%
%
\begin{scope}[canvas is xy plane at z=-4]
\path[draw,line width=0.4mm,black] (-4,-2.5) to (4,-2.5) to [edge node={node [pos=0.4,right,shift={(4pt,0pt)}] {}}] (4,2.5);
%
%
\path [draw,line width=0.6mm,red] (-4,0) to [edge node={node [pos=0.25,above,shift={(2pt,-2pt)}] {}}] (0,0);
\path [draw,line width=0.6mm,blue] (1.5,0) to [edge node={node [pos=0.5,above,shift={(2pt,-2pt)}] {}}] (4,0);
%
%
\draw[thick,orange,postaction={rmid arrow={orange,scale=1.5}}] (1.55+1.35,0) to [out=-90,in=-40] (1.55-1.6,-1.6); 
\draw[thick,purple,postaction={mid arrow={purple,scale=1.5}}] (1.55+1.35,0) to [out=90,in=20] (1.55-1.6,1.6); 
%
%
\draw[thick,black!40!green] (1.55-1.6,1.6) to (1.55-1.6,0.7);
\draw[thick,black!40!green] (1.55-1.6,-1.6) to (1.55-1.6,-0.7);
%
%
\draw[thick,violet] (1.55-1.6,0.7) to (1.55-1.6,0.45);
\draw[thick,violet] (1.55-1.6,-0.7) to (1.55-1.6,-0.45);
\draw[thick,violet] (1.55-1.6,-0.45) arc(-90:90:0.45);
\end{scope}
%
%
%
\foreach \x in {0.5,1,...,3}
\draw[line width=0.3mm,dashed,lightgray] (1+\x,0,-4)--(1+\x,0,0);

\fill [black] (0,0,-4) circle[radius=3pt] node [above left] {$0$}; 
\fill [black] (0,0,0) circle[radius=3pt] node [below] {$0$};
\fill [black] (1.5,0,-4) circle[radius=3pt] node [below] {$1$}; 
\fill [black] (1.5,0,0) circle[radius=3pt] node [below] {$1$};

\node[fill=black,regular polygon, regular polygon sides=4,inner sep=1.3pt] at (1.55-1.6,-1.6,-4) {};
\node[fill=black,regular polygon, regular polygon sides=4,inner sep=1.3pt] at (1.55-1.6,1.6,-4) {};
\node[fill=black,regular polygon, regular polygon sides=3,inner sep=1.3pt] at (1.55-1.6,-0.7,-4-0.05) {};
\node[fill=black,regular polygon, regular polygon sides=3,inner sep=1.3pt] at (1.55-1.6,0.7+0.05,-4+0.05) {};

\fill (-1.22,0) circle[radius=3pt];
\fill (0,0,0) circle[radius=3pt] node [below] {$0$};

\node [above right] at (5,0,0) {$\RS_1$};
\node [above right] at (5,0,-4) {$\RS_2$};

\end{tikzpicture}
\caption{
	The dashed curve in the first sheet is the original contour $\Gamma$. The solid curve is the contour $\Gamma'=\Gamma_1\cup\Gamma_2\cup\Gamma_3$. 
	The contour $\Gamma_3$ is the arc on the second sheet that connects the two marked triangle points. $\Gamma_2$ consists of two straight line segments each connecting one marked triangle point and one marked square point.
	The remaining part of the solid curve is $\Gamma_1$. }
\label{fig:deformation_contour_riemann_surface}
\end{figure}

We now evaluate the integral of $\omega_\tau$ on each part of $\Gamma'=\Gamma_1\cup\Gamma_2\cup\Gamma_3$. 

\begin{lem}\label{lem:integral_gamma5}
We have 
$$
	\lim_{\tau\to \infty} \int_{\Gamma_3}\omega_\tau =  \frac1{\sqrt{2\pi}} \int_\rx^\infty e^{-\frac12 u^2} du . 
$$
\end{lem}

\begin{proof}
The contour $\Gamma_3$ is on the second sheet $\RS_2$. 
From the definition \eqref{eq:oneform_omega} of $\omega_\tau$, Lemma \ref{lem:ffRS12second}, and Lemma~\ref{prop:local_expansion_prefactor_gaussian}, 
\begin{equation*}
	\int_{\Gamma_3}\omega_\tau
	=- \frac1{2\pi i} \int_{\Gamma_3} e^{H_\tau(z)} (1+\Boh(z^2)) \frac{dz}{z}
\end{equation*}
with 
\begin{equation*}
	H_\tau(z):= \frac{\tau}{\sqrt{2\pi}} \left( \Li_{3/2}(z)-\Li_{5/2}(z) \right) - \frac{\tau^{1/2}\hat x}{\sqrt{2\pi}} \Li_{3/2}(z)
	= \frac{\tau}{8 \sqrt{\pi}} z^2 + \tau\;\Boh(z^3) - \frac{\tau^{1/2}\hat x}{\sqrt{2\pi}} z+\tau^{1/2} \Boh(z^2)
\end{equation*}
where the error terms are independent of $\tau$. 
Using $|z|\le \tau^{-5/12}$ for the error term and changing the variables $z= \tau^{-1/2} s$, we find that 
\begin{equation*}
	\int_{\Gamma_3}\omega_\tau 
	=- \frac1{2\pi i} \int_{\tau^{1/2} \Gamma_3} e^{\frac{1}{8 \sqrt{\pi}} s^2  - \frac{\hat x}{\sqrt{2\pi}} s+ \Boh(\tau^{-1/4})}  \frac{ds}{s} (1+\Boh(\tau^{-5/6})).
\end{equation*}
The new contour $\tau^{1/2}\Gamma_3$ consists of the line segment from $-i\tau^{1/12}$ and $-i$, the line segment $i$ and $i\tau^{1/12}$, and the half-circle,  $\{z= e^{i\theta} |-\pi/2\le \theta\le \pi/2 \}$. 
It is oriented downwards. 
Note that the pole $s=0$ is located to the left of the contour. 
Reversing the orientation, recalling $\hat x= \frac{\pi^{1/4}\rx}{\sqrt{2}}$, and changing the variables $s=i 2\pi^{1/4} u$,  we find that 
\begin{equation*}
	\lim_{\tau\to \infty} \int_{\Gamma_3}\omega_\tau 
	= \frac1{2\pi i} \int_{-i\infty}^{i\infty}  e^{\frac{1}{8 \sqrt{\pi}} s^2  - \frac{\hat x}{\sqrt{2\pi}} s}  \frac{ds}{s} 
	= \frac1{2\pi i} \int_{-\infty}^\infty e^{-\frac12 u^2-i\rx u} \frac{du}{u},
\end{equation*}
where the contour of the middle integral goes around the pole $s=0$ through the right half plane, and the contour of the last integral goes around the pole $u=0$ through the lower half plane. Changing variables $u\mapsto u-i\rx$, it is easy to show that the last integral is a function of $\rx$ whose derivative is $-\frac1{\sqrt{2\pi}} e^{-\rx^2/2}$ and tends to zero as $\rx\to +\infty$. Hence, we find that
\begin{equation*}
	\frac1{2\pi i} \int_{-\infty}^\infty e^{-\frac12 u^2-i\rx u} \frac{du}{u} = \frac1{\sqrt{2\pi}} \int_\rx^\infty e^{-\frac12 u^2} du
\end{equation*}
and we obtain the result. 
\end{proof}

\begin{lem}\label{lem:integral_gamma3}
We have 
$$
\lim_{\tau\to\infty} \int_{\Gamma_2}\omega_\tau = 0 .
$$
\end{lem}

\begin{proof}
Since $\Gamma_2=\gamma_2\cup\overline{\gamma_2}$, by symmetry, it is enough to consider $\gamma_2$ which is the line segment in $\RS_2$ from the point $i e^{3\pi/8}$ to $i\tau^{-5/12}$. 
From the definition of $\omega_\tau$ and Lemmas \ref{lem:URS0zero} and \ref{lem:ffRS12second},
\begin{equation*}
	\int_{\gamma_3} \omega_\tau =
	\frac{i}{8\pi}\int_{\tau^{-5/12}}^{e^{3\pi/8}} \frac{\EERS(iy)}{y u_0(iy)^4}e^{-\frac{\tau}{\sqrt{2\pi}} \left(\Li_{5/2}(iy)-\Li_{3/2}(iy)\right)- \frac{\tau^{1/2} \hat x}{\sqrt{2\pi}} \Li_{3/2}(iy)} dy .
\end{equation*}
By the integral representation \eqref{eq:integral_repr_polylog},
\begin{equation*}
	\re\left(\Li_{5/2}(iy)-\Li_{3/2}(iy)\right)
	= \frac{y^2}{\sqrt{\pi}}\int_0^\infty \frac{1}{e^{2t}+y^2}\left(2t^{1/2}- \frac{4}{3}t^{3/2}\right)dt .
\end{equation*}
The integrand changes its sign at $t=3/2$. Splitting the integral into two, numerical evaluations show that, for $0\le y\le e^{3\pi/8}$, 
\begin{equation*} \begin{split}
	\re\left(\Li_{5/2}(iy)-\Li_{3/2}(iy)\right)
	&\ge \frac{y^2}{\sqrt{\pi}} \int_0^{\frac{3}{2}}\frac{1}{e^{2t}+e^{3\pi/4}} \left(2t^{1/2}-\frac{4}{3}t^{3/2}\right)dt 
	- \frac{y^2}{\sqrt{\pi}}\int_{\frac{3}{2}}^\infty \frac{1}{e^{2t}}\left(\frac{4}{3}t^{3/2}-2t^{1/2}\right)dt \\
	&= \frac{y^2}{\sqrt{\pi}} (( 0.06729\cdots) -  (0.02600 \cdots)) > (0.02) y^2 .
\end{split} \end{equation*}
On the other hand, the function $y\mapsto \Li_{3/2}(iy)$ is continuous in the interval $[0,e^{3\pi/8}]$ and has the behavior $\re \Li_{3/2}(iy)= \re (iy  -\frac{y^2}{2\sqrt{2}}+ O(y^3))= -\frac{y^2}{2\sqrt{2}}+ O(y^3)$ as $y\to 0$. 
Hence, there is a constant $c_1>0$ such that
$$
	|\re \Li_{3/2}(iy)|\leq c_1 y^2 \quad \text{for every } y\in [0, e^{3\pi/8}].
$$
In addition, by Lemma \ref{prop:local_expansion_prefactor_gaussian} and continuity, there is $c_2>0$ such that $|\EERS(iy)|\le c_2 |u_0(iy)|^4$ on the contour. 
Combining together, we find that  there is a constant $c_3>0$ such that 
\begin{equation*} \begin{split}
	\left| \int_{\Gamma_2} \omega_\tau \right| 
	\le c_2 \int_{\tau^{-5/12}}^{e^{3\pi/8}} \frac{e^{-c_3 \tau y^2}}{y} dy
	\le c_2 \int_{\tau^{1/12}}^\infty \frac{e^{-c_3 u^2}}{u} du
\end{split} \end{equation*}
and we obtain the result.
\end{proof}

\begin{lem}\label{lem:GL_int_gamma1}
We have 
$$
\lim_{\tau\to\infty} \int_{\Gamma_1} \omega_\tau = 0.
$$
\end{lem}

\begin{proof}
By symmetry, it is enough to consider only the part $\gamma_1$ of $\Gamma_1$ that starts on the lower half plane on $\RS_1$ and ends on the upper half plane on $\RS_2$. We specified $\gamma_1$ explicitly by its parametrization
$$
	\gamma_1(t)=e^{-\frac{\eta(t)^2}{2}}=e^{-\frac{1}{2}(t^2-\pi)+ i \sqrt{\pi} t} \quad 
	\text{where $\eta(t)=t- i \sqrt{\pi}$ for $-\sqrt{\pi}\leq t\leq \sqrt{\pi}/2$, } 
$$
and we take $\gamma_1(t)$ to be on the first sheet $\RS_1$ for $-\sqrt{\pi}\le t< 0$, and on the second sheet $\RS_2$ for $0< t\le \sqrt{\pi}/2$. 
We note that $\gamma_1(-\sqrt{\pi})=-1-i0$, $\gamma_1(0)=e^{\pi/2}$, and $\gamma_1(\sqrt{\pi}/2)=ie^{3\pi/8}$.
It is straightforward to check that 
$u_0(\gamma_1(t))= \eta(t)$ for $-\sqrt{\pi}\le t<0$ and $u_0(\gamma_1(t))= -\eta(t)$ for $0<t\le \sqrt{\pi}/2$.
From the definition \eqref{eq:oneform_omega} of $\omega_\tau$, 
\begin{equation}\label{eq:integral_gamma2_parametrization}
	\int_{\gamma_1}\omega_\tau 
	=\frac{1}{8\pi i} \int_{-\sqrt{\pi}}^{\sqrt{\pi}/2}  \frac{\EERS(\gamma_1(t))}{\eta(t)^4}
	e^{\tau \ffRS_1(\gamma_1(t)) + \tau^{1/2} \hat x \ffRS_2(\gamma_1(t)) } \gamma_1'(t) d t.
\end{equation}
We have $|\eta(t)^4|=|t- i \sqrt{\pi}|^4\ge \pi^2$. Each of the functions $|\gamma_1'(t)|$, $|\EERS(\gamma_1(t))|$ and $\re \ffRS_2(\gamma_1(t))$ is bounded along the contour: for $\gamma_1'$ this is clear from its definition whereas for $\re \ffRS_2$ and $\EERS$ this follows from Lemmas~\ref{lem:ffRS12second} and \ref{lem:EERSzero}. 
From the definition and Lemma \ref{lem:ffRS12second}, 
$$
	\ffRS_1(\gamma_1(t)) =
\begin{cases}
	\frac1{\sqrt{2\pi}} \Li_{3/2}(\gamma_1(t))-\frac1{\sqrt{2\pi}} \Li_{5/2}(\gamma_1(t))-2\eta(t) -\frac{2}{3}\eta(t)^3,& t<0, \\
	\frac1{\sqrt{2\pi}} \Li_{3/2}(\gamma_1(t))-\frac1{\sqrt{2\pi}} \Li_{5/2}(\gamma_1(t)), & t\geq 0.
\end{cases}
$$
Using this explicit expression one can check numerically that the function $\re \ffRS_1(\gamma_1(t))$ for $-\sqrt{\pi}\le t\le \sqrt{\pi}/2$ has its maximum at $t=\sqrt{\pi}/2$, and the maximum value is negative, see Figure \ref{fig:polylog_gamma2}. One could also rigorously prove that $\re \ffRS_1(\gamma_1(t))\le -0.03$ in this interval by expanding the difference $\Li_{3/2}(\gamma_1(t))-\Li_{5/2}(\gamma_1(t))$ in series and viewing the expression for $\ffRS_1$ above as a polynomial in $\eta(t)$ of degree $6$ plus a small and controllable error. Since this proof is straightforward but cumbersome, and the claim is clear from Figure~\ref{fig:polylog_gamma2}, we do not provide the explicit calculations. 
Hence, the integrand of \eqref{eq:integral_gamma2_parametrization} is $O(e^{-c\tau})$ for some $c>0$, and we obtain the result. 

\begin{figure}
\centering
\includegraphics[scale=.45]{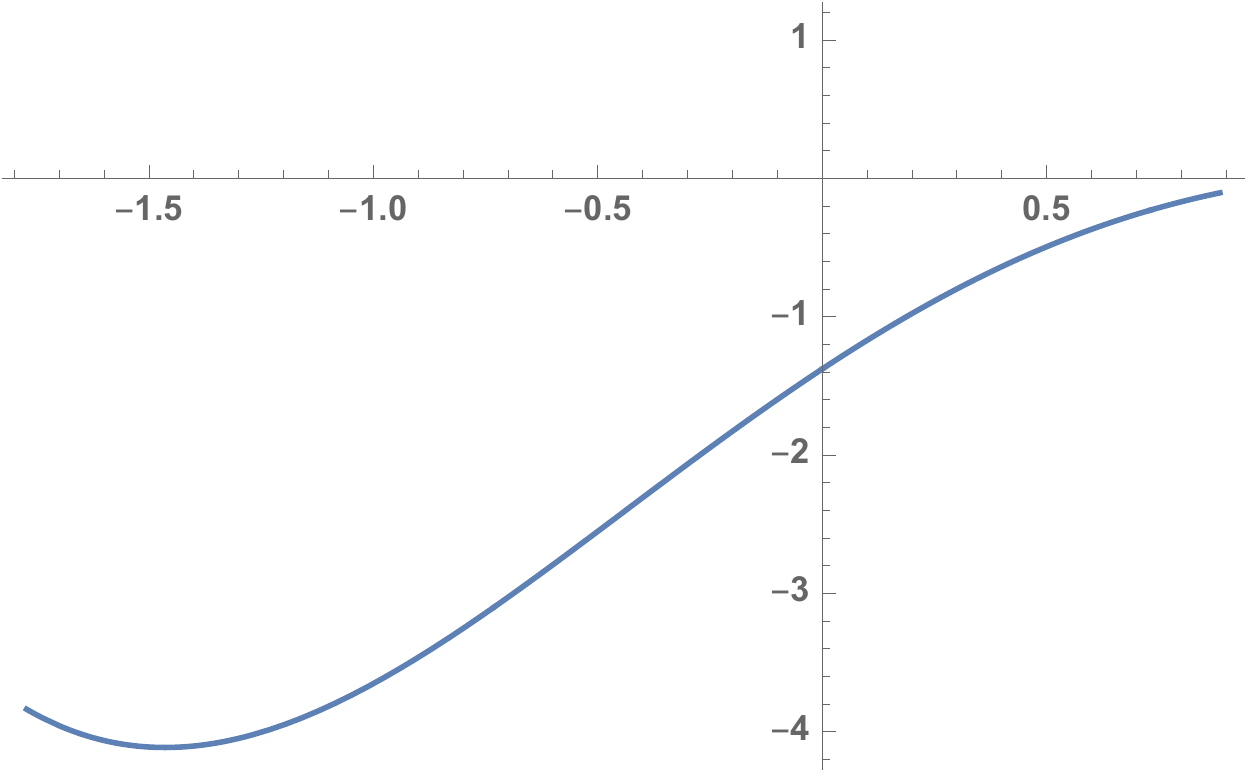}
\caption{The plot of $\re \ffRS_1(\gamma_1(t))$ for $t\in [-\sqrt{\pi},\sqrt{\pi}/2]$: the maximum is 
$-0.104065\cdots$ at $t=\sqrt{\pi}/2$. }\label{fig:polylog_gamma2}
\end{figure}
\end{proof}

\subsection{Completion of the proof of Theorem~\ref{thm:gaussian_degeneration}}

The only thing left to check is that it is possible to deform the contour $\Gamma$ in \eqref{eq:Finomega} to $\Gamma'$. 
We thus need to check that $\omega_\tau$ is analytic on $\RS$. 
Since $\ffRS_1, \ffRS_2, \EERS, \URS_0$ are all analytic on $\RS$, the only possible singularities of $\omega_\tau$ come from the zeros of $\URS_0$. The only zero of $\URS_0$ is the branch point $z=1$, which is a simple zero (see Lemma \ref{lem:URS0zero}). 
Since $\EERS$ has a zero of order $3$ by Lemma \ref{lem:EERSzero}, and $dz = 2w dw$ has a simple zero using the local coordinate $z-1=w^2$, we find that $\omega_\tau$ has a removable singularity at $z=1$. 
(Putting aside the terminology of Riemann surfaces, this simply means that the function $\frac{\EE(z) e^{\tau \ff_1(z)+ \tau^{1/2} \hat x \ff_2(z)}}{8\pi iz u_0(z)^4}$ has an integrable singularity at the branch point $z=1$). 
Therefore, $\omega_\tau$ is analytic on $\RS$, and we can deform $\Gamma$ to $\Gamma'$ in the integral in \eqref{eq:Finomega}. The proof is then completed from Lemmas~\ref{lem:integral_gamma5}, \ref{lem:integral_gamma3} and \ref{lem:GL_int_gamma1}.

\section{Small time limit}\label{sec:tracy_widom}

We prove Theorem~\ref{thm:tracy_widom_degeneration}. 

\subsection{Small $\tau$ limit}

We set $\gamma=0$ as in Theorem~\ref{thm:tracy_widom_degeneration}.  
Recall $\Phi(\xi)= -\frac13 \tau \xi^3+ x\xi - Q(\xi)$ from \eqref{def:phi_function}. 
Scaling $s, t$ by $\tau^{1/3}$ in Lemma~\ref{lem:kernelK_kernelT}, the distribution function \eqref{def:limiting_distribution} can be written as 
\begin{equation*}
	F(\tau^{1/3}x;\tau, 0)=\oint e^{\tau^{1/3}xA_1(z)+\tau A_2(z)+2B(z)}\det(\mb I-\mb A_{x,\tau}^2)\frac{dz}{2\pi i z},
\end{equation*}
where $\mb A_{x,\tau}:L^2(0,\infty)\to L^2(0,\infty)$ is the operator with the kernel 
\begin{equation}\label{def:kernel_operator_A}
	\bm A_{x,\tau}(s,t) 
	= \tau^{1/3} \bm T_0(\tau^{1/3}s, \tau^{1/3}t)|_{x\mapsto \tau^{1/3}x}
		=\tau^{1/3} \sum_{\xi \in \mathcal S_-}\frac{1}{-\xi}e^{-\frac{1}{3}\tau \xi^3+\tau^{1/3}\xi (s+t+x)-Q(\xi)} 
\end{equation}
for $s, t>0$. 
Observe that $\mb A_{x,\tau}$ depends on $z$ since the set $\mathcal S_-$ depends on $z$.

We compare $F(\tau^{1/3}x;\tau, 0)$ with the GUE Tracy-Widom distribution, 
$F_{\mathrm{GUE}}(x)=\det(\mb I-\mb A_x^2)$, where $\mb A_x:L^2(0,\infty)\to L^2(0,\infty)$ has kernel
\begin{equation*}
	\bm A_x(s,t)=\ai(s+x+t), \qquad s,t>0.
\end{equation*}
Recall the integral representation of the Airy function
\begin{equation}\label{integral_representation_airy}
	\ai(z)=\frac{1}{2\pi i}\int_{\Clambda}e^{-\frac{u^3}{3}+z u}du,\qquad z\in \C,
\end{equation}
where $\Clambda$ is an unbounded contour from $\infty e^{-\theta i}$ to $\infty e^{\theta i}$ for any $\theta\in (\pi/2,5\pi/6)$.
 
We will show that if we scale $z$ appropriately with $\tau$ then in the limit $\tau\to 0$ the term $Q(\xi)$ vanishes and the kernel $\bm A_{x,\tau}(s,t)$ converges to the Airy kernel along with its Fredholm determinant. 

\begin{lem}\label{lem:estimate:scaled_Q}
For every $\epsilon\in (0,1)$, there is a constant $c>0$ such that 
\begin{equation}\label{eq:estimate_Q_tw_scaling}
	Q(\tau^{-1/3}\zeta)=\Boh(e^{-c\tau^{-2/3}}) 
\end{equation}
as $\tau \to 0$ uniformly for $\zeta$ satisfying $\re (\zeta^2) \ge 1-\epsilon$ and $\re \zeta<0$. 
\end{lem}

\begin{proof}
From \eqref{eq:alternative_definition_q},  after a change of variables, 
\begin{equation*}
	Q(\tau^{-1/3}\zeta)
	=-\int_{-i\infty}^{i\infty} \frac{\log(1-e^{-\zeta^2/(2\tau^{2/3})}e^{u^2/2})}{u-\tau^{-1/3}\zeta}\frac{du}{\pi i}
	= -\frac1{\pi i} \int_{-i\infty}^{i\infty} \log\left(1-e^{- \frac{\zeta^2-u^2}{2\tau^{2/3}} }\right)\frac{du}{u-\zeta}. 
\end{equation*}
Uniformly for $u\in i\R$ and $\re(\zeta^2)\ge 1-\epsilon$,  
$ \big| e^{- \frac{\zeta^2-u^2}{2\tau^{2/3}}} \big| \le \big|e^{- \frac{\zeta^2}{2\tau^{2/3}} } \big| 
= e^{- \frac{\re(\zeta^2)}{2\tau^{2/3}} } \le e^{- \frac{1-\epsilon}{2\tau^{2/3}} } \to 0$ as $\tau\to 0$, and hence, 
\beqq
	\left|\log \left(1-e^{- \frac{\zeta^2-u^2}{2\tau^{2/3}}}\right)\right|
	\leq 2 \big|e^{- \frac{\zeta^2-u^2}{2\tau^{2/3}}} \big|
= 2e^{-\frac{\re (\zeta^2- u^2)}{2\tau^{2/3}}}\leq 2e^{-\frac{1-\epsilon+|u|^2}{2\tau^{2/3}}} .
\eeqq 
Also note that $|u-\zeta|\ge |\re u -\re\zeta| =\sqrt{(\re \zeta)^2}\ge \sqrt{\re \zeta^2}\ge \sqrt{1-\epsilon}$ uniformly for $u\in i\R$ and $\zeta$ as assumed in the lemma. 
These two estimates complete the proof. 
\end{proof}

\begin{lem}\label{prop:convergence_kernels_tw_scaling}
For every $x_0\in \R$, there exist constants $C, c, \tau_0>0$ such that 
$$
	|\bm A_{x,\tau}(s,t)-\bm A_{x}(s,t)|\leq C e^{-c\tau^{-2/3}}  e^{-c(s+t)-c(x-x_0)} 
$$
for every $x\geq x_0$, $s, t>0$, $\tau\in (0, \tau_0)$, and every $z$ satisfying 
\begin{equation}\label{eq:scaling_z_tw}
	|z|=e^{-\frac{1}{2\tau^{2/3}}}.
\end{equation}
\end{lem}

\begin{proof}
Define the function 
$$
	f(w) =\frac{z}{z-e^{-w^2/2}}e^{-\frac{1}{3}\tau w^3+\tau^{1/3}w\left(x+s+t\right)-Q(w)}
$$
for $\re w<0$ in the sector $\frac{3\pi}4< \arg w <\frac{5\pi}4$. 
Recall that this sector contains the set $\mathcal S_-$ (see \eqref{def:bethe_root_set}) and that the function $Q$ is analytic in this sector. 
The fraction in the formula for $f$ has simple poles precisely at the set $\mathcal S_-$ and
$$
	\res_{\xi} f=\frac{1}{\xi}e^{-\frac{1}{3}\tau \xi^3+\tau^{1/3}\xi\left(x+s+t\right)-Q(\xi)} 
	\quad \text{for } \xi\in \mathcal S_-.
$$
Therefore, (see \eqref{def:kernel_operator_A}) we can write 
\begin{equation}\label{eq:kernel_residue}
	\bm A_{x,\tau}(s,t)
	=-\tau^{1/3}\sum_{\xi\in \mathcal S_-}\res_{\xi}f 
	= \frac{\tau^{1/3}}{2\pi i}\int_{\Clambda_\inn}f(w)dw-\frac{\tau^{1/3}}{2\pi i}\int_{\Clambda_\out}f(w)dw,
\end{equation}
where $\Clambda_\inn$ and $\Clambda_\out$ are two unbounded curves in $\C$, both from $\infty e^{-3\pi i/4}$ to $\infty e^{3\pi i /4}$, lying in the sector $\frac{3\pi}4<\arg w < \frac{5\pi}4$ such that $\mathcal S_-$ is contained in the strip between $\Clambda_\inn$ and $\Clambda_\out$. 
We also choose $\Clambda_\inn$ to be to the left of $\Clambda_\out$; see Figure~\ref{fig:TW_contours}.

\begin{figure}\centering
\begin{tikzpicture}[scale = 0.76]
\draw [line width=0.4mm,lightgray] (-5,0)--(0.5,0) node [pos=1,right,black] {$\R$};
\draw [line width=0.4mm,lightgray] (0,-3)--(0,3) node [pos=1,above,black] {$i\R$};
\draw[thick,domain=-3:3,smooth,variable=\y,black,dashed]  plot ({-(\y*\y+1)^(1/2)},{\y});
\draw[thick,postaction={mid4 arrow={red,scale=1.5}},domain=-3:3,smooth,variable=\y,red,mark position=0.7(g1)]  plot ({-(\y*\y+1)^(1/2)-1},{\y})  node {};
\node at (g1) [left,red] {$\Clambda_{\inn}$};
\draw[thick,postaction={mid4 arrow={red,scale=1.5}},domain=-3:3,smooth,variable=\y,red,  mark position=0.7(g2)]  plot ({-(\y*\y+1)^(1/2)+0.7},{\y}) node {};
\node at (g2) [right,red] {$\Clambda_{\out}$};
\fill (-1.06573,-0.368479) circle[radius=2.5pt] node [above,shift={(0pt,0pt)}] {};
\fill (-2.0173,-1.75199) circle[radius=2.5pt] node [above,shift={(0pt,0pt)}] {};
\fill (-2.68227,-2.48889) circle[radius=2.5pt] node [above,shift={(0pt,0pt)}] {};
\fill (-2.53213,2.3263) circle[radius=2.5pt] node [above,shift={(0pt,0pt)}] {};
\fill (-1.81494,1.51459) circle[radius=2.5pt] node [above,shift={(0pt,0pt)}] {};
\end{tikzpicture}
\caption{The dots are the points in $\mathcal S_-$ and sit in the dashed parabola determined by $|e^{-w^2/2}|=|z|$. The contours $\Clambda_\inn$ and $\Clambda_\out$ are represented in solid lines. }\label{fig:TW_contours}
\end{figure}
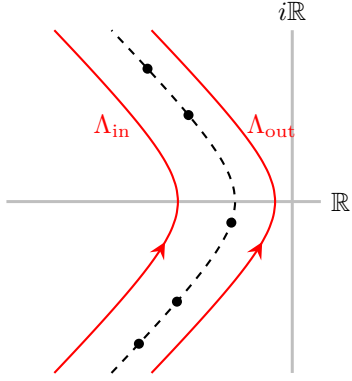

Let $z$ satisfy \eqref{eq:scaling_z_tw}. Changing the variables $w=\tau^{-1/3}\zeta$ we get the identity
\begin{equation}\label{eq:integral_f_reg}
	\tau^{1/3}\int_{\Clambda_{\inn,\out}} f(w)dw = \int_{\widehat \Clambda_{\inn,\out}}\frac{z}{z-e^{-\zeta^2/(2\tau^{2/3})}} e^{-\frac{1}{3}\zeta^3+\zeta\left(x+s+t\right)-Q(\tau^{-1/3}\zeta)} d\zeta,
\end{equation}
where the new contours $\widehat \Clambda_{\inn}$ and $\widehat \Clambda_{\out}$ are as follows. 
Since $|z|=e^{-\frac1{2\tau^{2/3}}}$, 
$$ 
	\tau^{1/3} \mathcal S_-\subset \{ u\in \C \; \mid \; \re u^2=1, \; \re u<0 \}.
$$
Thus, for every $\varepsilon\in (0,1)$, the scaled set $\tau^{1/3} \mathcal S_-$ lies between the contours 
\begin{equation*}
	\widehat \Clambda_\inn := \{ \zeta\in \C \; \mid \; \re \zeta^2=1+\varepsilon, \; \re \zeta<0 \}
	\quad \mbox{and}\quad 
	\widehat \Clambda_\out := \{ \zeta\in \C \; \mid \; \re \zeta^2=1-\varepsilon, \; \re \zeta<0 \}.
\end{equation*}
These contours lie in the sector $\frac{3\pi}4< \arg \zeta< \frac{5\pi}4$. 
We choose the contours with any fixed $\varepsilon\in (0,1)$. 

Note that since $|z|= e^{-\frac1{2\tau^{2/3}}}$, we have $|ze^{ \frac{\zeta^2}{2\tau^{2/3}}}|=e^{ \frac{ \re\zeta^2-1}{2\tau^{2/3}}}=e^{\frac{\varepsilon}{2\tau^{2/3}}}$ for $\zeta\in\widehat\Clambda_\inn$ 
and $|ze^{ \frac{\zeta^2}{2\tau^{2/3}}}|= e^{-\frac{\varepsilon}{2\tau^{2/3}}}$ for $\zeta\in\widehat\Clambda_\out$.
Thus, 
\begin{equation}\label{eq:estimate_quotient_gamma1_tw}
	\frac{z}{z-e^{-\zeta^2/(2\tau^{2/3})}}
	=\left(1-\frac{1}{ze^{\zeta^2/(2\tau^{2/3})}}\right)^{-1}
	=1+\Boh(e^{-\frac{\varepsilon}{2\tau^{2/3}}}),
	\qquad \zeta\in \widehat \Clambda_\inn,
\end{equation}
and
\begin{equation}\label{eq:estimate_quotient_gamma2_tw}
	\frac{z}{z-e^{-\zeta^2/(2\tau^{2/3})}}
	=-ze^{\zeta^2/(2\tau^{2/3})} \left(1-ze^{\zeta^2/(2\tau^{2/3})}\right)^{-1}
	=\Boh(e^{-\frac{\varepsilon}{2\tau^{2/3}}}),
	\qquad \zeta\in \widehat\Clambda_{\out}
\end{equation}
as $\tau\to 0$, uniformly for $\zeta$ on the respective contour and $z$ satisfying \eqref{eq:scaling_z_tw}.

Using $\Clambda=\widehat \Lambda_\inn$ in the integral formula of the Airy function in \eqref{integral_representation_airy} and \eqref{eq:integral_f_reg}, 
we obtain 
\begin{equation*} 
\begin{split}
	\left| \frac{\tau^{1/3}}{2\pi i}\int_{\Clambda_\inn}f(w)dw-\ai(s+x+t) \right| 
	&= \left| \frac{1}{2\pi i}\int_{\widehat \Clambda_\inn}\left(\frac{z}{z-e^{-\zeta^2/(2\tau^{2/3})}}e^{- Q(\tau^{-1/3}\zeta)}-1\right) e^{-\frac{1}{3}\zeta^3+\zeta\left(x+s+t\right) }d\zeta \right| . 
\end{split}
\end{equation*} 
Since $s, t>0$ and $x>x_0$, and $-\re\zeta =\sqrt{(\re \zeta)^2}\ge \sqrt{\re \zeta^2}= \sqrt{1+\epsilon}$, we have 
$$
\re \zeta\left(x+s+t\right) \le -\sqrt{1+\epsilon} (x-x_0+ s+t) + x_0\re \zeta,\quad \zeta\in \widehat \Clambda_{\inn}\cup \widehat \Clambda_{\out}.
$$
The integral $\int_{\widehat \Clambda_{\inn}} e^{\frac{1}{3} \re(\zeta^3)-  x_0 \re \zeta} |d \zeta| $ is convergent, 
and thus we find using \eqref{eq:estimate_quotient_gamma1_tw} and \eqref{eq:estimate_Q_tw_scaling} that there is a constant $c>0$ which is independent of the parameters, and another constant $C>0$ which depends only on $x_0$, such that  
\begin{equation*} \begin{split}
	\left| \frac{\tau^{1/3}}{2\pi i}\int_{\Clambda_\inn}f(w)dw-\ai(s+x+t) \right| 
	& \le C e^{-c\tau^{-2/3}} e^{-\sqrt{1+\epsilon}(x-x_0+ s+t)} 
\end{split} \end{equation*} 
as $\tau \to 0$. Similarly, we obtain using \eqref{eq:estimate_quotient_gamma2_tw} 
\begin{equation*} \begin{split}
	\left| \frac{\tau^{1/3}}{2\pi i}\int_{\Clambda_\out}f(w)dw  \right| 
	& \le C e^{-c\tau^{-2/3}} e^{-\sqrt{1-\epsilon}(x-x_0+ s+t)} .
\end{split} \end{equation*} 
The proof is now complete once we combine these estimates with \eqref{eq:kernel_residue}--\eqref{eq:integral_f_reg}.
\end{proof}

\begin{cor}\label{prop:convergence_fred_det_tw}
For every $x_0\in \R$, there exist constants $C,c, \tau_0>0$ such that
$$
	|\det\left(\mb I-\mb A^2_{x,\tau}\right)-\det\left(\mb I-\mb A^2_{x}\right)|\leq C e^{-c \tau^{-2/3}}e^{-c(x-x_0)}
$$
for all $x\geq x_0$, $\tau\in (0, \tau_0)$, and $z$ satisfying $|z|= e^{-\frac1{2\tau^{2/3}}}$. 
\end{cor}

\begin{proof}
Denote by $\norm{\cdot}_1$ the trace norm and by $\norm{\cdot}_2$ the Hilbert-Schmidt norm. Using
$$
	|\det (\mb I-\mb K_1)-\det (\mb I-\mb K_2)|\leq \norm{\mb K_1-\mb K_2}_1\exp\left({\norm{\mb K_1}_1+\norm{\mb K_2 }_1+1}\right)
$$
and $\|\mb \K_1 \mb K_2 \|_{1}\leq \|\mb K_1\|_2 \|\mb K_2\|_2 $, we have 
\begin{equation*} 
	|\det\left(\mb I-\mb A^2_{x,\tau}\right)-\det\left(\mb I-\mb A^2_{x}\right)| 
	\leq \norm{\mb A_{x,\tau}-\mb A_{x}}_2\left(\norm{\mb A_{x,\tau}}_2+\norm{\mb A_x}_2\right)
	\exp\left({\norm{\mb A_{x,\tau}}^2_2+\norm{\mb A_{x}}^2_2+1}\right)
\end{equation*}
Lemma~\ref{prop:convergence_kernels_tw_scaling} implies that
\beqq
	\norm{\mb A_{x,\tau}-\mb A_x}_2^2
	\le C^2 e^{-2c\tau^{-2/3}} e^{-2c(x-x_0)} \int_0^\infty \int_0^\infty e^{-2c(s+t)} dsdt 
	= \Boh(e^{-2c\tau^{-2/3}}) .
\eeqq
On the other hand, $\norm{\mb A_x}_2$ is uniformed bounded for $x\geq x_0$.
Therefore, we obtain the result. 
\end{proof}

\subsection{Completion of the proof of Theorem~\ref{thm:tracy_widom_degeneration}} \label{sec:comppfsmall}

The functions $A_1,A_2$ and $B$ in \eqref{def:functions_A} are analytic on the unit disk and $A_1(0)=A_2(0)=B(0)=0$. 
Thus, by Cauchy's Theorem,
$$
	F(\tau^{1/3}x;\tau,\gamma=0)-F_{\mathrm{GUE}}(x)
	=\oint e^{\tau^{1/3}xA_1(z)+\tau A_2(z)+2B(z)} \left(\det(\mb I-\mb A_{x,\tau}^2)-\det(\mb I-\mb A_x^2) \right)\frac{dz}{2\pi i z}.
$$
Taking the contour as the circle \eqref{eq:scaling_z_tw}, the proof now follows from Corollary~\ref{prop:convergence_fred_det_tw}

\section{Right tail estimate}\label{sec:large_tail}

In this section we prove Theorem~\ref{thm:large_tail}. 
It is convenient to reformulate the right tail of $F_{\mathrm{GUE}}(x)$ and $(F_{\mathrm{GUE}}(x))^2$ in terms of an integral involving the Airy function.
The Airy function $\mathrm{Ai}(x)$ has the integral representation~\eqref{integral_representation_airy} and
\begin{equation}\label{eq:Airy_asymptotics}
	\mathrm{Ai}(x) = \frac{1}{2\sqrt{\pi}x^{1/4}}e^{-\frac{2}{3}x^{3/2}}\left(1+\Boh(x^{-3/2})\right),\qquad x\to\infty.
\end{equation}
We define
\begin{equation}
\label{eq:def_CB_general}
\CB(x;\alpha):= \int_x^\infty (y-x)e^{\alpha y}\mathrm{Ai}(y)^2dy,\qquad \alpha\in\R,
\end{equation}
and denote in particular
\begin{equation*}
\CB(x)=\CB(x;0).
\end{equation*}

\begin{lem} \label{prop:right_tail_FGUE}
We have
		\begin{equation}
		\label{eq:asympt_CB}	\CB(x;\alpha)=\frac{1}{16\pi x^{3/2}}e^{\alpha x-\frac{4}{3}x^{3/2}} \left(1+ \frac{\alpha}{x^{1/2}}+\frac{3\alpha^2}{4x}+ \Boh(x^{-3/2})\right)
		\end{equation}
		as $x\to\infty$. Moreover, for any positive constants $a$ and $\beta$,
		\begin{equation}
		\label{eq:asympt_CB2}
		\int_{x}^{x+ax^\beta}(y-x)e^{\alpha y} \mathrm{Ai}(y)^2dy = \CB(x;\alpha) (1+\Boh(e^{-ax^{\beta+\frac12}})),\quad x\to\infty.
		\end{equation} 
\end{lem}

\begin{proof}
Inserting~\eqref{eq:Airy_asymptotics} in~\eqref{eq:def_CB_general}, and writing $y=x+s$, we have
	\begin{equation*}
	\begin{split}
	\CB(x;\alpha) 
	&= \frac{e^{\alpha x}}{4\pi} \int_0^\infty \frac{s}{(s+x)^{1/2}}e^{\alpha s} e^{-\frac{4}{3} (x+s)^{3/2}}ds \left(1+\Boh(x^{-3/2})\right)\\
	&= \frac{e^{\alpha x-\frac{4}{3}x^{3/2}}}{4\pi \sqrt{x} }  \int_{0}^1 s e^{\alpha s } 
	\frac{e^{-\frac{s^2}{2\sqrt{x}}}}{(1+s/x)^{1/2}} e^{-2\sqrt{x}s} ds \left(1+\Boh(x^{-3/2})\right)
	\end{split}
	\end{equation*}
where for the second equality we changed the integral domain from $(0, \infty)$ to $(0, 1)$ since the integral over $(1, \infty)$ only gives exponentially smaller error, and expanded $(x+s)^{3/2}$. 
Expanding the fraction in the integrand in power series and applying Watson's lemma we obtain~\eqref{eq:asympt_CB} after routine calculations. 
Using a similar calculation, 
	\begin{equation*}
	\begin{split}
	\int_{x+ax^\beta}^\infty (y-x)e^{\alpha y}\mathrm{Ai}(y)^2dy
	&= \frac{e^{\alpha x-\frac{4}{3}x^{3/2}}}{4\pi \sqrt{x} }  \Boh(e^{-cx^{\beta+\frac12}})
	\end{split}
	\end{equation*}
for any positive constant $c$ with $c<2a$. Combining with~\eqref{eq:asympt_CB}, we obtain~\eqref{eq:asympt_CB2}.
\end{proof}

Noting that $\Tr (\mb A_x^2) = \CB(x)$, the right tail of the GUE Tracy-Widom distribution satisfies
\begin{equation}\label{eq:right_tail_FGUE}
	1-F_{\mathrm{GUE}}(x) = \CB(x) \left(1+\Boh(e^{-\frac{4}{3}(1-\epsilon)x^{3/2}})\right)
\end{equation}
as $x\to\infty$, for any $\epsilon>0$ fixed. This implies that 
\begin{equation}\label{eq:right_tail_FGUE2}
	1-(F_{\mathrm{GUE}}(x))^2 =2\CB(x) \left(1+\Boh(e^{-\frac{4}{3}(1-\epsilon)x^{3/2}})\right) .
\end{equation}
		
\bigskip

Now we turn to the function $F(x;\tau,\gamma)$.
Recall Lemma~\ref{lem:kernelK_kernelT}, 
\begin{equation}\label{eq:limiting_distr_AB_decomposition}
	F(x;\tau,\gamma)=\oint e^{xA_1(z)+\tau A_2(z)+2B(z)}\det(\mb I-\mb \Tngamma\mb \Tgamma) \frac{dz}{2\pi iz}
\end{equation}
where the contour is any circle $|z|=R$ with $0<R<1$. 
We take $x\to +\infty$ with the choice of contour
\begin{equation*} 
	|z|=e^{-\frac{x}{2\tau}}.
\end{equation*}
We prove the following result. 

\begin{prop} \label{prop:LT_HSnorm_est}
For every $\gamma\in [-1/2, 1/2]$, there exists a  constant $0<\epsilon<1$  such that the Hilbert-Schmidt norm of $\mb \Tgamma$ satisfies
\begin{equation*}
	\|\mb \Tgamma\|_2\le e^{-\frac{2}{3}(1-\epsilon)x^{3/2}\tau^{-1/2} } \quad \text{as $x\to \infty$,} 
\end{equation*}
uniformly for $|z|=e^{-\frac{x}{2\tau}}$.
\end{prop}

We also need to evaluate the leading term asymptotics for $\Tr  ( \mb \Tngamma \mb \Tgamma )$. This will be done with a careful control of the Laurent expansion of the kernel for the product $\mb \Tngamma \mb \Tgamma$ at $z=0$. As a result we will see that for $|\gamma|<1$ the Laurent coefficient for $z^0$ is the leading one, whereas for $|\gamma|=1/2$ the coefficients for $z^{\pm 1}$ and $z^0$ all contribute to the same order as $x\to\infty$.
For the precise statement, set
\beqq
	\cb(\tau, \gamma, x):= \CB\left(\frac{x}{\tau^{1/3}}+\frac{\gamma^2}{4\tau^{4/3}}\right), \quad
	\cb_{\pm}(\tau, \gamma, x):= \CB\left(\frac{x}{\tau^{1/3}}+\frac{\gamma^2}{4\tau^{4/3}};\pm \frac{1}{2\tau^{2/3}}\right). 
\eeqq

\begin{prop}\label{prop:LT_trace_asympt}
For each $\gamma\in [-1/2, 1/2]$, there is $c>0$ such that the following result holds as $x\to \infty$ uniformly for $|z|=e^{-\frac{x}{2\tau}}$:  
if $|\gamma|<1/2$, 
\begin{equation*}
	\Tr  ( \mb \Tngamma \mb \Tgamma ) =  
	\cb(\tau, \gamma, x) \left(1+\Boh(e^{-cx^{1/2}})\right) ,
\end{equation*}
and if $\gamma=\pm 1/2$, 
\beqq
\begin{split}
	&\Tr ( \mb \Thalf \mb \Tnhalf )
	= \left(2 b(\tau, 1/2, x) 
	+ ze^{-\frac{1}{96\tau^2}} \cb_+(\tau, 1/2, x) 
	+ z^{-1}e^{\frac{1}{96\tau^2}} \cb_-(\tau, 1/2, x) \right) \left(1+\Boh(e^{-cx^{1/2}})\right) .
\end{split}
\eeqq
\end{prop}

Lemma~\ref{prop:right_tail_FGUE} implies that $\cb_{\pm}(\tau, 1/2, x)  = e^{\pm ( \frac{x}{2\tau} + \frac{1}{32\tau^2} ) } \cb(\tau, 1/2, x)(1+\Boh(x^{-1/2}))$.
Hence, for $|z|=e^{-\frac{x}{2\tau}}$, 
\beq \label{eq:threeequalorder}
	z\cb_+(\tau, 1/2, x) = \Boh(b(\tau, 1/2, x)),
	\qquad 
	z^{-1} \cb_-(\tau, 1/2, x) = \Boh(b(\tau, 1/2, x)),
\eeq
so the three terms in $\Tr ( \mb \Thalf \mb \Tnhalf )$ have the same order.

We prove the two propositions in the following subsections. Assuming them, we now obtain Theorem~\ref{thm:large_tail}.

\begin{proof}[Proof of Theorem~\ref{thm:large_tail}]
The functions $A_1,A_2$ and $B$ are analytic (see \eqref{def:functions_A}), and using \eqref{def:series_polylog} we see that  
\begin{equation*}
	E(z):= e^{xA_1(z)+\tau A_2(z)+2B(z)}=1 - \frac{(x+\tau)z}{\sqrt{2\pi}}  + \Boh(z^2) \qquad \text{as } z\to 0. 
\end{equation*}
Hence,  
\beq \label{eq:integral_prefactor}
	\oint E(z) \frac{dz}{2\pi iz}=1, 
	\qquad \oint zE(z) \frac{dz}{2\pi iz}=0, 
	\qquad \oint z^{-1} E(z) \frac{dz}{2\pi iz}=- \frac{x+\tau}{\sqrt{2\pi}}. 
\eeq
We also have $E(z)= \Boh(1)$ for $|z|=e^{-\frac{x}{2\tau}}$ as $x\to\infty$.
On the other hand, using  $\|\mb  \Tngamma\mb \Tgamma \|_1\le \|\mb \Tngamma\|_2 \|\mb \Tgamma\|_2$ and Proposition \ref{prop:LT_HSnorm_est}, 
\beqq 
	\det(\mb I - \mb\Tngamma \mb \Tgamma) - 1 + \Tr (\mb\Tngamma\mb \Tgamma) = \Boh(\|\mb\Tngamma\mb \Tgamma \|_1^2) = \Boh(  e^{-\frac{8(1-\epsilon)x^{3/2}}{3\tau^{1/2}}} )
\eeqq 
as $x\to \infty$ uniformly for $|z|=e^{-\frac{x}{2\tau}}$. 
Hence, 
\begin{equation}
\label{eq:F_trace}
	1- F(x;\tau,\gamma)= \oint_{|z|=e^{-\frac{x}{2\tau}}} E(z) \Tr (\mb \Tngamma\mb \Tgamma) \frac{dz}{2\pi iz} + O(e^{-\frac{8(1-\epsilon)x^{3/2}}{3\tau^{1/2}}}).
\end{equation}
Thus, for $|\gamma|<1/2$, Proposition~\ref{prop:LT_trace_asympt} and \eqref{eq:right_tail_FGUE} imply the theorem. 
For $\gamma=\pm 1/2$, Proposition~\ref{prop:LT_trace_asympt}, the equation \eqref{eq:integral_prefactor}, and \eqref{eq:threeequalorder} imply that  the integral in \eqref{eq:F_trace} is 
\begin{equation*}
\begin{split}
	\left( 2\cb(\tau, x, 1/2)  -\frac{(x+\tau)e^{\frac{1}{96\tau^2}}}{\sqrt{2\pi}}   \cb_-(\tau, x, 1/2)\right)
	\left(1+\Boh(e^{-cx^{1/2}})\right).
\end{split}
\end{equation*}
Since $\cb_{-}(\tau, 1/2, x)  = \Boh( e^{- \frac{x}{2\tau}} \cb(\tau, 1/2, x))$, we obtain the result from \eqref{eq:right_tail_FGUE2}. 
\end{proof}

In Section~\ref{sec:analysis_Tgamma}, we first analyze the kernel of $\mb T_\gamma$. We then prove Propositions~\ref{prop:LT_HSnorm_est} and~\ref{prop:LT_trace_asympt} in Sections~\ref{sec:proof_LT_HSnorm_est} and~\ref{sec:proof_LT_trace_asympt}, respectively.

\subsection{Analysis of the kernel $\bm T_\gamma$}\label{sec:analysis_Tgamma}

The kernel of $\mb T_\gamma$ is given by Lemma~\ref{lem:kernelK_kernelT} and we recall it here,
\begin{equation} \label{eq:T_to_CA}
	\bm T_\gamma(s,t)=\CT_{\gamma}(x+s+t), 
	\qquad 
	\CT_\gamma(y)= \CT_\gamma(y;z,\tau)=\sum_{\xi\in\mathcal S_-}\frac{e^{-\frac{\tau}{3} \xi^3 +\frac{\gamma}{2}\xi^2 +y\xi -Q(\xi)}}{-\xi} 
\end{equation}
for $s, t>0$. The function $Q$ is defined in~\eqref{def:Q_function}. In this subsection we find the asymptotic behavior of $\CT_\gamma(y)$ as $y\to\infty$. 

Fix $\tau>0$ and define the Airy-like function
\begin{equation} \label{eq:def_RA}
 \RA(y;\mu)=  \int_{\Lambda} e^{ -\frac{\tau}3 \xi^3 + \frac{\mu}{2}\xi^2+y\xi -Q(\xi) }\frac{d\xi}{2\pi i}
 \end{equation}
where $\Lambda$ is any simple unbounded contour from $\infty e^{5\pi i /4}$ to $\infty e^{3\pi i /4}$ contained in the sector $3\pi/4<\arg\xi<5\pi/4$, the latter condition being imposed so that $Q$ remains analytic along $\Lambda$.
Without the term $Q$, the above function is a scaled Airy function, 
\beq
\label{eq:aux_02}
    \int_{\Lambda} e^{ -\frac{\tau}3 \xi^3 +\frac{\mu}{2}\xi^2+y\xi  }\frac{d\xi}{2\pi i}
    =\frac{e^{ \frac{\mu^3}{12\tau^2} + \frac{\mu y}{2\tau} }}{\tau^{1/3}}  \ai\left(\frac{y}{\tau^{1/3}}+\frac{\mu^2}{4\tau^{4/3}}\right).
\eeq
The function $\CT_\gamma(y)$ in \eqref{eq:T_to_CA} looks like a discretization of $\RA(y;\gamma)$. The factor $\frac1{-\xi}$ is close to the spacing between the points in the discrete set $\mathcal S_-$, which is not uniformly spaced.
A precise comparison between $\CT_\gamma(y)$ and $\RA(y;\gamma)$ is in Proposition \ref{prop:asymptotics_CA} below. 

\subsubsection{Asymptotics of $\RA(y; \mu)$} 

\begin{lem}\label{lem:asymptotics_g_large_tail}
There are constants $c>0$ and $y_0>0$ such that
\begin{equation*}
    	\RA(y;\mu) 
	= \frac{e^{\frac{\mu^3}{12\tau^2} + \frac{\mu y}{2\tau} }}{\tau^{1/3}}  \ai\left(\frac{y}{\tau^{1/3}}+\frac{\mu^2}{4\tau^{4/3}}\right)
	\left( 1+ \Boh(e^{-cy} ) \right)
\end{equation*}
for all $y\ge y_0$ and $\mu\le 2\sqrt{\tau y}$.
\end{lem}

\begin{proof}
The proof follows from the standard steepest descent analysis, accounting for the factor $Q$ in the exponent. 
Introduce the change of variables $\xi\mapsto w$ for~\eqref{eq:def_RA} and~\eqref{eq:aux_02},
 $$
 	\xi= s w +\frac{\mu }{2\tau}, 	\qquad \text{with}\quad s:= \frac{1}{\sqrt{\tau}}\left(y+\frac{\mu^2}{4\tau}\right)^{1/2}>0.
 $$
In terms of the new variable, 
\begin{equation*}
 	\tau^{1/3} e^{ -\frac{\mu^3}{12\tau^2} - \frac{\mu y}{2\tau} } \RA(y;\mu) 
	- \ai ( \tau^{2/3}s^2  ) 
	=\frac{\tau^{1/3} s}{2\pi i}  \int_{\widehat \Lambda} e^{\tau s^3 \left(-\frac{1}{3}w^3+w\right)}(e^{-Q(\xi(w))}-1)dw.
\end{equation*}
To evaluate the integral as $s\to \infty$, we need to take a contour of steepest descent. 
The relevant critical point of the function $-w^3/3+w$ is $w=-1$ and its path of steepest descent is a contour from $\infty e^{4\pi i/3}$ to $\infty e^{2\pi i/3}$, it is on the hyperbola $3x^2-y^2=3$.
However, $Q(\xi(w))$ is analytic in the sector $3\pi/4<\arg \xi(w)<5\pi/3$, which is approximately $3\pi/4<\arg w<5\pi/3$ when $w$ becomes large, 
and consequently we cannot deform the contour to the path of steepest descent. 
Nonetheless, since the functions involved are explicit, we find that there is a contour in the sector that is ``steep enough" and passes the critical point $z=-1$. 
We use the following concrete choice: 
$$
 	\widehat \Lambda=\widehat\Lambda_1\cup \widehat\Lambda_2:=\{-1+bi\mid -0.2\le b\le 0.2\}\cup \{a+bi\mid a^2-b^2=-1.45a-0.49, \, a\le -1\}.
$$ 
This contour is locally the same as the path of steepest descent near the critical point $z=-1$. 
To verify that the method of steepest descent applies, it is enough to show that (1) $\re(-\frac{1}{3}w^3+w)$ decreases as $a=\re w$ decreases along $\widehat \Lambda_2$ and (2) the factor
$e^{-Q(\xi(w))}-1$ is exponentially small uniformly on the contour. 

For (1), it is direct to check that for $w=a+bi\in\widehat\Lambda_2$,
\beqq
 	\re \big(-\frac{w^3}{3}+w \big)=\frac{2}{3}a^3+1.45a^2 +1.49a
\eeqq
and its derivative is $2a^2+2.9a+1.49>0$ for all $a$. 
 	
For (2), we use the integral representation~\eqref{eq:alternative_definition_q} of $Q$. 
For $w=a+bi\in\widehat \Lambda_2$,
 	\begin{equation*}
 	\begin{split}
 	\re(\xi(w)^2)&=\frac{1}{\tau}\left(y+\frac{\mu^2}{4\tau}\right) \left(\lambda^2 +a(2\lambda -1.45) - 0.49) \right), 
 	\quad \lambda:=\frac{\mu}{2\tau s}=\frac{\mu}{\sqrt{\mu^2+4\tau y}}.
 	\end{split}
 	\end{equation*}
 	Note that the above function depends only on $a=\re(w)$, not on $b=\im(w)$. 
 	The assumption $\mu\le  2 \sqrt{\tau y}$ implies that $\lambda\le 1/\sqrt{2}$ 	which in turn gives $2\lambda-1.45<0$. 
 	Hence, the above function takes its minimum when $a=-1$, and for this value of $a$ we have $b=\pm 0.2$ since $a+ib$ is on $\widehat\Lambda_2$. 
 	
 	On the other hand, for $w=a+ ib \in\widehat\Lambda_1$, $\re(\xi(-1+bi)^2)$ is equal to $-s^2b^2$ plus a term which is independent of $b$. Hence, the minimum of $\re(\xi(w)^2)$ over $w\in \widehat\Lambda_1$ is achieved at the end points $w=-1\pm 0.2i$. 
 	
 	Thus, we found that $\re(\xi(w)^2) \ge \re(\xi(-1 \pm 0.2i)^2)$ for any $w\in \widehat\Lambda$. Now, noting that the function $f(v):= v^2-2v+0.96$ is a decreasing function for $v\le 1$, we have 
 	$f(\lambda)\ge f(\frac1{\sqrt{2}})= 1.46-\sqrt{2}>0$, and therefore, 
  	\begin{equation*}
 	\begin{split}
 	\re(\xi(w)^2) \ge \re(\xi(-1 \pm 0.2i)^2) 
	= s^2  f(\lambda) 
	\ge \left(1.46-\sqrt{2}\right)s^2 \quad \text{for $w\in \widehat\Lambda.$} 
 	\end{split}
 	\end{equation*}	
 	Combined with $\re(\xi(w))<0$ from the definition of $\xi(w)$, we find that $\xi(w)$ is in the sector $3\pi/4<\arg(\xi(w))<5\pi/4$ 
 	for all $w\in \widehat \Lambda$, and also, 
 	\begin{equation*}
 	|e^{-\xi(w)^2/2}|=e^{-\re(\xi(w)^2)/2}\le e^{-c(y+\frac{\mu^2}{4\tau})}
 	\end{equation*}
 with $c=(1.46-\sqrt{2})/\tau$. With arguments similar to the ones in Lemma~\ref{lem:estimate:scaled_Q} we obtain
\begin{equation} \label{eq:Qerrorbd}
 	|Q(\xi(w))|= \Boh ( e^{-c(y+\frac{\mu^2}{4\tau})})
\end{equation}
 	as $y\to \infty$, uniformly for $w\in\widehat\Lambda$. Hence $e^{-Q(\xi(w))}-1=\Boh ( e^{-c(y+\frac{\mu^2}{4\tau})})$ uniformly for $w\in\widehat\Lambda$.
	
Thus, the method of steepest-descent applies. Since the integral without the term $(e^{-Q(\xi(w))}-1)$ is the same as $\ai ( \tau^{2/3}s^2  )$, the asymptotic formula is the same as that of the airy function multiplied by an error from $(e^{-Q(\xi(w))}-1)$.
From  \eqref{eq:Qerrorbd} we find that 
\beqq
	\frac{\tau^{1/3} s}{2\pi i}  \int_{\widehat \Lambda} e^{\tau s^3 \left(-\frac{1}{3}w^3+w\right)}(e^{-Q(\xi(w))}-1)dw
	= \Boh (\ai(\tau^{1/3} s^2) e^{-c(y+\frac{\mu^2}{4\tau})}).
\eeqq
This completes the proof. 
\end{proof}

By \eqref{eq:Airy_asymptotics}, the above lemma implies that
\begin{equation} \label{eq:RA_asympotitcs}
	\RA(y;\mu)= \frac{1}{2\sqrt{\pi}\tau^{1/4}}\left(y+\frac{\mu^2}{4\tau}\right)^{-1/4}
	e^{\frac{\mu^3}{12\tau^2}+\frac{\mu y}{2\tau} -\frac{2}{3 \sqrt{\tau}}\left(y+\frac{\mu^2}{4\tau}\right)^{3/2}}
	\left( 1+ \Boh(y^{-3/2}) \right)
\end{equation}
for all $y\ge y_0$ and $\mu\le 2\sqrt{\tau y}$.

\subsubsection{Truncated series expansion of $\CT_\gamma(y)$}
 
Note that $\bm \CT_\gamma(y)$ depends on $\mathcal S_-$ which, in turn, depends on the complex variable $z$ that appears in the integral in \eqref{eq:limiting_distr_AB_decomposition}. To compare $\bm \CT_\gamma(y)$ with $\RA(y; \gamma)$, we start with the following formula. 

\begin{lem} 
There are positive constants $x_0$ and $C$ such that 
\beqq 
	\left| \bm \CT_\gamma(y) - \sum_{k=-K}^\infty z^{-k}\RA\left(y;\gamma-k\right) \right| 
	\le C \left( \frac{|z|}{0.9} \right)^K 
\eeqq
for all $y\geq x_0$, $0<|z|\le 0.8$, 
and $K\geq 1$. 
\end{lem}

The values $0.8$ and $0.9$ are chosen solely for convenience.
The proof will show that the condition $|z|\le 0.8$ can be changed to $|z|\le r$ for any fixed $r\in(0,1)$ with the bound changed to $\left(\frac{|z|}{R}\right)^K$ for any $R\in (r, 1)$. 

\begin{proof} 
Using the fact that $\mathcal S_-$ is a set of zeros of the equation $e^{-\xi^2/2}-z=0$, we can write  
(cf. the proof of Lemma \ref{prop:convergence_kernels_tw_scaling}) 
	\begin{equation*} 
		\CT_\gamma(y) = \left(\int_{\Lambda_\out} -\int_{\Lambda_\inn}\right)  \frac{z}{z-e^{-\xi^2/2}} e^{-\frac{\tau}{3}\xi^3 +\frac{\gamma}{2}\xi^2 +y\xi -Q(\xi)} \frac{d\xi}{2\pi i}
	\end{equation*}
	where we take $\Lambda_\out:=\{\xi: |e^{-\xi^2/2}|=0.9|z|, \, \re(\xi)<0\}$ and $\Lambda_\inn:=\{\xi: |e^{-\xi^2/2}|=0.9, \,  \re\xi <0\}$. Both contours lie in the sector $3\pi/4<\arg \xi <5\pi/4$ and they extend from $\infty e^{5\pi i/4}$ to $\infty e^{3\pi i/4}$. 
	The points of $\mathcal S_-$ lie between these two contours since $|e^{-\xi^2/2}|=|z|$ for all $\xi\in\mathcal S_-$ and $0.9|z|<|z|<0.9$.

Since $|z^{-1} e^{-\xi^2/2}|=0.9<1$ for $\xi\in \Lambda_\out$, the geometric series implies that
\begin{align*}
	\frac{1}{2\pi i}\int_{\Lambda_\out}\frac{z}{z-e^{-\xi^2/2}}e^{-\frac{\tau}{3}\xi^3 +\frac{\gamma}{2}\xi^2 +y\xi-Q(\xi)}d\xi 
&=\sum_{k=0}^\infty z^{-k}\int_{\Lambda_\out}e^{-\frac{\tau}{3}\xi^3 +\frac{\gamma-k}{2}\xi^2 +y\xi -Q(\xi)}\frac{d\xi}{2\pi i}.
\end{align*}
The integral is $\RA\left(y;\gamma-k\right)$. 
On the other hand, for $\xi \in \Lambda_\inn$, we write  
\begin{equation*}
	\frac{z}{z-e^{-\xi^2/2}} 
	=-\sum_{k=1}^{K}z^{k}e^{k\xi^2 /2}+ \left(\frac{z}{e^{-\xi^2/2}}\right)^{K} \frac{z}{z-e^{-\xi^2/2}} .
\end{equation*}
Since $|e^{-\xi^2/2}|=0.9$ for $\xi \in \Lambda_\inn$ and $|z|\le 0.8$, 
\begin{equation*}
\begin{split}
	&\left| \frac{1}{2\pi i} \int_{\Lambda_\inn} \left(\frac{z}{e^{-\xi^2/2}}\right)^{K} \frac{z}{z-e^{-\xi^2/2}} e^{-\frac{\tau}{3}\xi^3 +\frac{\gamma}{2}\xi^2 +y\xi -Q(\xi)}d\xi \right| 
	 \le \frac{4}{\pi} \left( \frac{|z|}{0.9} \right)^{K} (0.9)^{-\gamma}  \int_{\Lambda_\inn} e^{-\re (\frac{\tau}{3}\xi^3 +Q(\xi))} |d \xi| .
\end{split}
\end{equation*}
where we use the fact that $\re (\xi)<0$ 
on $\Lambda_\inn$. 
The last integral does not depend on $z$, $y$, or $K$. Hence, we obtain the result. 
\end{proof}

We use the above result when $|z|=e^{-\frac{x}{2\tau}}$ with $x\to \infty$. 
We will also choose $K=K(y)$ depending on $y$ in a specific way; see \eqref{eq:Kydef}.

\begin{cor} \label{lem:CAseriesf}
There are positive constants $x_0$ and $C$ such that 
\beq \label{eq:kernel_A_series_expansion}
	\left| \bm \CT_\gamma(y) - \sum_{k=-K}^\infty z^{-k}\RA\left(y;\gamma-k\right) \right| 
	\le C e^{-\frac{9K(y)x}{19\tau}}
\eeq
for all $0< |z|\le e^{-\frac{x}{2\tau}}$, $y\geq x\ge x_0$, and any choice of $K=K(y)\geq 1$. 
\end{cor}

\subsubsection{$\CT_\gamma(y)$ when $y$ is close to $x$}

We show that the main contribution to the sum in Corollary~\ref{lem:CAseriesf} comes from only one or two terms. 
From~\eqref{eq:RA_asympotitcs}, for $|z|=e^{-\frac{x}{2\tau}}$, 
\begin{equation}
\label{eq:asymptotics_RA}
\begin{split}
	&\left|z^{-k} \RA(y;\mu-k)\right|
	 = \frac{e^{\frac{\mu x}{2\tau} +\frac{2y^{3/2}}{3 \sqrt{\tau}}\fe\left(\frac{\mu-k}{2\sqrt{\tau y}};1-\frac{x}{y}\right)}}{\sqrt{2\pi}(4\tau y+(\mu-k)^2))^{1/4}} 
	\left(1+\Boh(y^{-3/2}) \right)
\end{split}
\end{equation}
as $y\to \infty$, uniformly for $\mu$ and $k$ satisfying $\mu-k \le  2\sqrt{\tau y}$, 
where 
\begin{equation} \label{eq:fffns}
	\fe(u;\lambda):=u^3 +\frac{3}{2}\lambda u -(1+u^2)^{3/2}.    
\end{equation}
This function $f(u; \lambda)$ is  strictly concave for every $\lambda$. It is strictly increasing when $\lambda\ge 1$, and when $\lambda<1$ it has the unique maximum
\begin{equation}
\label{eq:min_fe}
	\max_{u\in\R}\fe(u;\lambda)=\fe(u_c;\lambda)=-\frac{3}{2}(1-\lambda)^{1/2}+\frac{1}{2}(1-\lambda)^{3/2}, 
	\qquad u_c=\frac{\lambda}{2\sqrt{1-\lambda}} .
\end{equation}

We are interested in the case when $x\le y\le x+ \Boh(x^{1/2})$ and $x\to \infty$. 
In this case,  $\lambda= 1- \frac{x}{y} \to 0$ as $x\to \infty$, and hence the maximizer $u_c$ of $f(u;1-\frac{x}{y})$ is close  to $0$. 
We may expect that the main contribution to the sum in Corollary~\ref{lem:CAseriesf} comes from the term corresponding to $k$ at which the function $\fe (\frac{\mu-k}{2\sqrt{\tau y}};1-\frac{x}{y} )$ becomes the largest. 
Since $f$ is concave and $k$ is a discrete index, the maximum may occur either at one value of $k$ or two values of $k$. 
We will show that the main contribution to the sum comes from the term $k=0$ when $|\gamma|<1/2$, from the terms $k=1, 0$ when $\gamma=1/2$, and from the terms $k=0, -1$ when $\gamma=-1/2$. 
For this purpose, we first analyze the function $f$. 
We use the next result when $s=\Boh(x)\to \infty$. 

\begin{lem}\label{Claim:asymptotics_RA}
The function $f$ defined in \eqref{eq:fffns} satisfies the following inequalities. 
	\begin{enumerate}[(a)]
		\item For all $\alpha<1/2$, $s>0$, and  $\lambda\ge 0$, 
		\beqq
			\fe \left(\frac{\alpha}{s}; \lambda \right)-\fe \left(\frac{\alpha-1}{s}; \lambda \right)\ge \frac{3(1-2\alpha)}{2s^2} .
		\eeqq
		\item For all $\alpha>-1/2$, $s\ge \frac{3(1+2\alpha)^2 +1}{5(1+2\alpha)}$, and $0\le \lambda\le \frac{1+2\alpha}{12s}$, 
		\beqq
			\fe \left(\frac{\alpha}{s}; \lambda \right)-\fe \left(\frac{\alpha+1}{s}; \lambda \right) \ge \frac{1+2\alpha}{8s^2} .
		\eeqq
	\end{enumerate}	 
\end{lem}

\begin{proof}
(a) Since $u^3$ and $\frac32\lambda u$ are increasing functions of $u$, we have 
\beqq \begin{split}
	\fe \left(\frac{\alpha}{s}; \lambda \right)-\fe \left(\frac{\alpha-1}{s}; \lambda \right)
	&\ge \left(1+ \frac{(\alpha-1)^2}{s^2} \right)^{3/2} - \left(1+ \frac{\alpha^2}{s^2} \right)^{3/2}. 
\end{split} \eeqq
Using the inequality $(1+a)^{3/2}\ge 1+\frac{3a}{2}$, which holds for all $a\ge 0$, the above is equal to 
\beqq \begin{split}
	\left(1+ \frac{\alpha^2}{s^2}\right)^{3/2} \left(  \left(1+ \frac{1-2\alpha}{\alpha^2+s^2}\right)^{3/2} -1 \right)
	\ge \left(1+ \frac{\alpha^2}{s^2}\right) \frac{3(1-2\alpha)}{2(\alpha^2+s^2)}
	= \frac{3(1-2\alpha)}{2s^2}. 
\end{split} \eeqq

(b) From the formula of $f$, 
\beqq \begin{split}
	\fe \left(\frac{\alpha}{s}; \lambda \right)-\fe \left(\frac{\alpha+1}{s}; \lambda \right)
	= \frac{\alpha^3-(\alpha+1)^3}{s^3} - \frac{3\lambda}{2s} 
	+\big(1+ \frac{\alpha^2}{s^2}\big)^{3/2} \left(  \left(1+ \frac{1+2\alpha}{\alpha^2+s^2}\right)^{3/2} -1 \right) . 
\end{split} \eeqq
From the assumption on $s$, 
$$
	\frac{\alpha^3-(\alpha+1)^3}{s^3}= - \frac{3(1+2\alpha)^2 + 1}{4s^3} \ge -\frac{5(1+2\alpha)}{4s^2}. 
$$
Using the inequality $(1+a)^{3/2}\ge 1+\frac{3a}{2}$ for $a\ge 0$ again, 
\beqq
\begin{split}
	\fe \left(\frac{\alpha}{s}; \lambda \right)-\fe \left(\frac{\alpha+1}{s}; \lambda \right)
	\ge -\frac{5(1+2\alpha)}{4s^2} - \frac{3\lambda}{2s} + \frac{3(1+2\alpha)}{2s^2} 
\end{split}
\eeqq
From the condition $\lambda\le \frac{1+2\alpha}{12s}$ we conclude the proof.
\end{proof}

\begin{lem} \label{lem:Agk}
Fix $\tau>0$. 
\begin{enumerate}[(a)]
	\item For each fixed $\mu<1/2$, there exists $x_0>0$ such that 
	\beqq
		|z^{-k-1}\RA(y;\mu-k-1)| \le 2 e^{- \frac{1-2\mu+2k}{4\tau^{3/2}} y^{1/2}} | z^{-k} \RA(y;\mu-k)|
	\eeqq
	 for all $|z|=e^{-\frac{x}{2\tau}}$, $y\ge x\ge x_0$, and $k\ge 0$. 
	\item For each fixed $\mu> -1/2$, there exists $x_0>0$ such that 
	\beqq
		|z^{k+1}\RA(y;\mu+k+1)| \le 2 e^{- \frac{1+2\mu+2k}{48\tau^{3/2}} y^{1/2}} | z^{k} \RA(y;\mu+k)|
	\eeqq
	 for all $|z|=e^{-\frac{x}{2\tau}}$ and $x\ge x_0$, and for all $y$ and $k$ satisfying $x\le y\le x+ \frac{1+2\mu+2k}{24\sqrt{\tau}} \sqrt{x}$ and $0\le k\le K(y) :=\lfloor \frac{19}{12} \sqrt{\tau y} \rfloor$. 
\end{enumerate}
\end{lem}

\begin{proof}
(a) The formula \eqref{eq:asymptotics_RA} applies since $\mu-k  \le \mu\le 2\sqrt{\tau x}\le  2\sqrt{\tau y}$ holds if $x$ is large enough. 
Note that $(\mu-k)^2\le (\mu-k-1)^2$ for all $k\ge 0$ since $\mu<1/2$. 
Also noting that the error term in \eqref{eq:asymptotics_RA} is uniform in $y$ and $k$,  we find that for all large enough $y\ge x$, 
\beqq
	\frac{|z^{-k-1}\RA(y;\mu-k-1)|}{| z^{-k} \RA(y;\mu-k)|}
	\le 2 \exp \left\{ -\frac{2y^{3/2}}{3\tau^{1/2}} \left( \fe\left(\frac{\mu-k}{2\sqrt{\tau y}};1-\frac{x}{y}\right) - \fe\left(\frac{\mu-k-1}{2\sqrt{\tau y}};1-\frac{x}{y}\right) \right) \right\} .
\eeqq
Applying Lemma \ref{Claim:asymptotics_RA} (a) with $\alpha= \mu-k$, $s=2\sqrt{\tau y}$, and $\lambda= 1- \frac{x}{y}$, we obtain the result.

(b) Since $\mu+k +1 \le \mu+1 +K(y)\le  2\sqrt{\tau y}$ if $y\ge x$ and $x$ is large enough, the formula \eqref{eq:asymptotics_RA} applies. By the same argument as the part (a), the result follows from Lemma \ref{Claim:asymptotics_RA} (b) with $\alpha= \mu+k$, $s=2\sqrt{\tau y}$, and $\lambda= 1- \frac{x}{y}$ if these values satisfy the conditions of the lemma.  
Clearly, $\alpha>-1/2$ since $\mu>-1/2$ and $k\ge 0$. 
The condition  $0\le \lambda\le \frac{1+2\alpha}{12s}$ is satisfied since 
\beqq
	0\le 1-\frac{x}{y} = \frac{y-x}{y} \le \frac{(1+2\mu+2k)x^{1/2}}{24\sqrt{\tau}y}
	\le \frac{1+2\mu +2k}{24\sqrt{\tau y}} = \frac{1+2\alpha}{12s} 
\eeqq
Finally, to verify the condition on $s$ we observe that for $x$ large enough so that $\sqrt{\tau x}\ge 6(1+2\mu)+ \frac2{1+2\mu}$, 
\beqq
	\frac{3}{5}(1+2\alpha)+ \frac{1}{5(1+2\alpha)} 
	\le \frac{3}{5}(1+2\mu+2K)+ \frac1{5(1+2\mu)}
	= \frac{1}{10}\left(6(1+2\mu)+\frac{2}{1+2\mu}\right)+\frac{6}{5}K\leq s
\eeqq
since $K\le \frac{19}{12}\sqrt{\tau y}$ and $x\le y$. 
Hence, Lemma \ref{Claim:asymptotics_RA} (b) applies and we obtain the result. 
\end{proof}

The above estimates imply the following result. 

\begin{prop}[$\CT_\gamma(y)$ when $y$ is close to $x$] \label{prop:asymptotics_CA}
Fix $\tau>0$. 
For any fixed $\gamma\in (-1/2, 1/2)$,  
there exists $c>0$ such that 
	\begin{equation} \label{eq:asymptotics_CA}
	\CT_\gamma(y)= \RA(y;\gamma)  \left(1+\Boh(e^{-cy^{1/2}})\right) 
	\end{equation}
	as $x\to\infty$ uniformly for $y\in \left[x,x+\frac{(1+2\gamma)\sqrt{x}}{24\sqrt{\tau}}  \right]$ and  $|z|=e^{-\frac{x}{2\tau}}$. 
For $\gamma=\pm 1/2$, there is $c>0$ such that 
	\begin{equation*}
	\begin{split}
	\CT_{\pm1/2}(y)=& \left(\RA(y;\pm1/2) +z^{\mp 1}\RA(y;\mp 1/2)\right) \left(1+\Boh(e^{-cy^{1/2}})\right) + \Boh(e^{-\frac{9K(y)x}{19\tau}}) 
	\end{split}	
	\end{equation*}
	as $x\to\infty$ uniformly for $y\in \left[x,x+\frac{\sqrt{x}}{12 \sqrt{\tau} } \right]$ and  $|z|=e^{-\frac{x}{2\tau}}$. 
\end{prop}

\begin{proof}
We use the truncated series formula for $\CT_\gamma(y)$ given in Corollary \ref{lem:CAseriesf} with 
\beq \label{eq:Kydef}
	K(y)= \left \lfloor \frac{19}{12} \sqrt{\tau y} \right \rfloor.
\eeq
Consider first the case when $-1/2<\gamma<1/2$. 
For $k\ge 1$, using Lemma \ref{lem:Agk} (a) $k$ times, 
\beqq
	|z^{-k} \RA(y;\gamma-k)| \le (2 e^{-c_1 y^{1/2}})^{k} |  \RA(y;\gamma)|, 
	\qquad c_1:= \frac{1-2\gamma}{4\tau^{3/2}} 
\eeqq
for all $y\ge x$. 
For $-K\le k\le -1$, using Lemma \ref{lem:Agk} (b) $|k|$ times, 
\beqq
	|z^{-k} \RA(y;\gamma-k)| \le (2 e^{-c_2 y^{1/2}})^{|k|} | \RA(y;\gamma)|, 
	\qquad c_2:= \frac{1+2\gamma}{48\tau^{3/2}}
\eeqq
for $x\le y\le x+ \frac{(1+2\gamma)\sqrt{x}}{24\sqrt{\tau}}$. With \eqref{eq:kernel_A_series_expansion} in mind,
this implies that $k=0$ term gives the main contribution to the sum and with $c=\min\{c_1, c_2\}$, 
$$
	\CT_\gamma(y)= \RA(y;\gamma)  \left(1+\Boh(e^{-cy^{1/2}})\right) + \Boh(e^{-\frac{9K(y)x}{19\tau} }) 
$$
as $x\to\infty$ uniformly for $y\in \left[x,x+\frac{(1+2\gamma)\sqrt{x}}{24\sqrt{\tau}}  \right]$ and  $|z|=e^{-\frac{x}{2\tau}}$. 
We compare the two error terms. 
From  \eqref{eq:Kydef} and $y\ge x$, we see that $ \Boh(e^{-\frac{9K(y)x}{19\tau} })$ is 
(super-)exponentially smaller than $\RA(y;\gamma)$ (see \eqref{eq:RA_asympotitcs}). Hence, the additive error can be replaced by a multiplicative error and we obtain \eqref{eq:asymptotics_CA}. 

Consider now the case $\gamma=1/2$. 
For $k\ge 2$, we use Lemma \ref{lem:Agk} (a) $k-1$ times to obtain
\beqq
	|z^{-k} \RA(y;1/2-k)| \le (2 e^{-c_1 y^{1/2}})^{k-1} | z^{-1} \RA(y;-1/2)|, 
	\qquad c_1:= \frac{1-2\gamma+2}{4\tau^{3/2}} = \frac{1}{2\tau^{3/2}},
\eeqq
for all $y\ge x$. 
For $-K\le k\le -1$, we use Lemma \ref{lem:Agk} (b) $|k|$ times to obtain
\beqq
	|z^{-k} \RA(y;1/2-k)| \le (2 e^{-c_2 y^{1/2}})^{|k|} | \RA(y;1/2)|, 
	\qquad c_2:= \frac{1+2\gamma}{48\tau^{3/2}} = \frac{1}{24\tau^{3/2}}, 
\eeqq
for $x\le y\le x+ \frac{\sqrt{x}}{12\sqrt{\tau}} $. Proceeding as before and replacing the additive error by the multiplicative error we obtain the result in this case.

Finally, consider the case when $\gamma=-1/2$. For $k\ge 1$, we use \ref{lem:Agk} (a) $k$ times and estimate 
of $|z^{-k} \RA(y;1/2+k)|$ in terms of $|\RA(y;1/2)|$ for all $y\ge x$. For $-K\le k\le -2$, we use \ref{lem:Agk} (b) $|k|-1$ times and find an estimate of
$|z^{-k} \RA(y;1/2+k)|$ in terms of $|z \RA(y;-1/2)|$, which are valid for $x\le y\le x+ \frac{1+2\gamma+2}{24\sqrt{\tau}} \sqrt{x}
= x+\frac{\sqrt{x}}{12\sqrt{\tau}} $. We now obtain the result in the same way as before.
\end{proof}

\subsubsection{Uniform estimate of $\CT_\gamma(y)$ for all $y\ge x$}

Although valid for all $y\geq x$, the next result will only be used for $y\ge x+\Boh(x^{1/2})$. 

\begin{prop}[Uniform estimate of $\CT_\gamma(y)$ for $y\ge x$]	\label{prop:bound_CA}
Fix $-1/2\le \gamma \le 1/2$ and $\tau>0$. 
There exist positive constants $C$ and $x_0$ such that
\begin{equation*}
	\left|\CT_\gamma(y)\right| \le  C\max\left\{ y^{1/4} e^{ \frac{\gamma x}{2\tau} -\frac{\sqrt{x}(3y-x)}{3 \sqrt{\tau}}}, e^{-\frac{9K(y) x}{19\tau}}  \right\}, \qquad K(y) :=\left \lfloor \frac{19}{12} \sqrt{\tau y} \right\rfloor, 
\end{equation*}
for all $y\ge x\ge x_0$ and $|z|=e^{-\frac{x}{2\tau}}$. 
\end{prop}

\begin{proof}
We use \eqref{eq:kernel_A_series_expansion}. 
Since Lemma \ref{lem:Agk} (a) holds for all $y\ge x$, using the same argument of the proof of Proposition \ref{prop:asymptotics_CA} we find that for some $c_1>0$, 
\beqq
	\sum_{k=2}^\infty |z^{-k}\RA(y;\gamma-k)| \le |z^{-2}\RA(y;\gamma-2)|\left(1+\Boh(e^{-c_1y^{1/2}})\right)\le 2|z^{-2} \RA(y;\gamma-2)|.
\eeqq
for all $y\ge x$ as $x\to\infty$. 
Thus,
\beqq
	\sum_{k=-K(y)}^\infty |z^{-k}\RA(y;\gamma-k)| 
	\le 2 (K(y)+3) \max_{-K(y)\le k\le 2} |z^{-k} \RA(y;\gamma-k)|
\eeqq
for all $y\ge x$ as $x\to\infty$.
Hence, from \eqref{eq:asymptotics_RA} we obtain that there are $C>0$ and $x_0>0$ such that 
\beqq
	\sum_{k=-K(y)}^\infty |z^{-k}\RA(y;\gamma-k)| 
	\le Cy^{1/4} e^{\frac{\gamma x}{2\tau}}  \exp\left\{ \frac{2y^{3/2}}{3\tau^{1/2}} \max_u \fe\left(u;1-\frac{x}{y}\right)\right\}
\eeqq
for all $y\ge x\ge x_0$. 
Inserting the formula \eqref{eq:min_fe} of the maximum, we obtain 
\beqq
	\sum_{k=-K(y)}^\infty |z^{-k}\RA(y;\gamma+k)| 
	\le Cy^{1/4} e^{ \frac{\gamma x}{2\tau} -\frac{\sqrt{x}(3y-x)}{3\sqrt{\tau}}}.
\eeqq
Remembering the error term in \eqref{eq:kernel_A_series_expansion} we obtain the result. 
\end{proof}

\subsection{Proof of Proposition~\ref{prop:LT_HSnorm_est}}
\label{sec:proof_LT_HSnorm_est}

From the formula \eqref{eq:T_to_CA}, 
\begin{equation*}
\begin{split}
	\|\mb \Tgamma\|_2^2 = \int_0^\infty\int_0^\infty |\CT_\gamma(x+s+t)|^2 dsdt
&=\int_x^\infty (y-x) |\CT_\gamma(y)|^2 dy . 
\end{split}
\end{equation*}
We split the integral into two parts: one for $x\le y\le x+ c\sqrt{x}$ and the other for $y\ge x+c\sqrt{x}$, where 
$c=\frac{1-2\gamma}{24\tau^{1/2}}$ for $|\gamma|<\frac12$ and $c=\frac{1}{12\tau^{1/2}}$ for $\gamma=\pm\frac12$. 
Equation~\eqref{eq:RA_asympotitcs} implies that there is $0< \epsilon<1$ such that 
\begin{equation*}
	|\RA(y;\gamma)|^2, \, |z^{\mp1}\RA(y;\gamma\mp1)|^2\le  e^{-\frac{4}{3\sqrt{\tau}}(1-\frac\epsilon2) y^{3/2}}
\end{equation*}
as $x\to \infty$, uniformly for $x\le y\le x+c\sqrt{x}$. 
Thus, for the first part of the integral, Proposition~\ref{prop:asymptotics_CA} and the equation \eqref{eq:RA_asympotitcs} imply that for every $\epsilon\in (0,1)$, 
\begin{equation*}
	\int_{x}^{x+c\sqrt{x}} (y-x) |\CT_\gamma(y)|^2 dy 
	\le  \int_{x}^{x+cx^{1/2}} (y-x) e^{-\frac{4}{3 \sqrt{\tau}}(1-\frac\epsilon2) y^{3/2}} dy
	\le e^{-\frac{4}{3\sqrt{ \tau}}(1-\epsilon) x^{3/2}}
\end{equation*}
for all large enough $x$. 
For the second part of the integral, applying Proposition~\ref{prop:bound_CA} we have
\begin{equation*}
\begin{split}
	&\int_{x+c\sqrt{x} }^\infty (y-x) |\CT_\gamma(y)|^2 dy 
	\le C^2 \int_{x+c\sqrt{x} }^\infty (y-x) \left( y^{1/2}e^{\frac{\gamma x}{\tau}-\frac{2(3y-x) \sqrt{x} }{3\sqrt{\tau} } }+ e^{-\frac{18 K(y) x}{19\tau}}\right) dy
	\le e^{-\frac{4}{3\sqrt{ \tau}}(1-\epsilon) x^{3/2}}
\end{split}
\end{equation*}
for all large enough $x$. 
This completes the proof.

\subsection{Proof of Proposition~\ref{prop:LT_trace_asympt}}\label{sec:proof_LT_trace_asympt}

From the formula \eqref{eq:T_to_CA}, 
\begin{equation*}
	\Tr (\mb \Tngamma\mb \Tgamma ) = \int_0^\infty\int_0^\infty \CT_{-\gamma}(x+s+t)\CT_{\gamma}(x+s+t)dsdt
	=\int_x^\infty (y-x) \CT_{-\gamma}(y)\CT_{\gamma}(y)dy.
\end{equation*}
We split the integral into two parts; the part $I_1$ for $x\le y\le x+ a\sqrt{x}$ and the second part $I_2$ for $y\ge x+ a\sqrt{x}$ where $a= \frac{1-2|\gamma|}{24\sqrt{\tau}}$ when $|\gamma|<1/2$, and $a= \frac1{12\sqrt{\tau}}$ when $\gamma=1/2$. 

Consider the first part $I_1$. 
Proposition~\ref{prop:asymptotics_CA} implies that for $\gamma\in (-1/2, 1/2)$, 
\beqq
	I_1= \int_x^{x+ax^{1/2}}(y-x) \RA(y;\gamma)\RA(y;-\gamma) dy \left( 1+ \Boh(e^{-cx^{1/2}})\right), 
\eeqq
and for $\gamma=\pm 1/2$, 
\beqq
	I_1= \int_x^{x+ax^{1/2}}(y-x) \left( 2\RA(y;1/2)\RA(y;-1/2) +z\RA(y;1/2)^2 +z^{-1}\RA(y;-1/2)^2\right)dy 
	\left( 1+ \Boh(e^{-cx^{1/2}})\right). 
\eeqq
From Lemma \ref{lem:asymptotics_g_large_tail}  and Equation \eqref{eq:asympt_CB2}, 
\begin{equation*}
	\int_x^{x+a x^{1/2}} (y-x)\RA(y;\gamma)\RA(y;-\gamma) dy 
	= \CB \left(\frac{x}{\tau^{1/3}}+\frac{\gamma^2}{4\tau^{4/3}} \right)\left(1+\Boh(e^{-cx^{1/2}})\right),
\end{equation*}
\begin{equation*}
	\int_x^{x+ax^{1/2}}(y-x)\RA(y;1/2)^2 dy = e^{-\frac{1}{96\tau^2}}
	\CB \left(\frac{x}{\tau^{1/3}}+\frac{1}{16\tau^{4/3}};\frac{1}{2\tau^{2/3}} \right)\left(1+\Boh(e^{-cx^{1/2}})\right),
\end{equation*}
and 
\begin{equation*}
	\int_x^{x+ax^{1/2}}(y-x)\RA(y;-1/2)^2 dy = e^{ \frac{1}{96\tau^2}}
	\CB \left(\frac{x}{\tau^{1/3}}+\frac{1}{16\tau^{4/3}};-\frac{1}{2\tau^{2/3}} \right)\left(1+\Boh(e^{-cx^{1/2}})\right),
\end{equation*}

Consider the second part $I_2$.  Using Proposition~\ref{prop:bound_CA} (which holds for $\gamma\in [-1/2, 1/2]$) for each of $\CT_\gamma(y)$ and $\CT_{-\gamma}(y)$, we need to consider four integrals. Noting the cancellation of the terms $-\frac{\gamma x}{2\tau}$ and $\frac{\gamma x}{2\tau}$, the first integral is bounded by a constant times 
\begin{equation*}
\begin{split}
	\int_{x+a\sqrt{x} }^\infty (y-x) y^{1/2}e^{-\frac{2(3y-x) \sqrt{x} }{3\sqrt{\tau} } } dy
	= \Boh\left(  x e^{-\frac{4}{3\sqrt{ \tau}}x^{3/2} - cx  } \right) .
\end{split}
\end{equation*}
for some $c>0$. As we will see in a moment, this extra decay $e^{-cx}$ is relevant.
Two of the remaining three integrals are bounded by a constant times 
\begin{equation*}
\begin{split}
	\int_{x+a\sqrt{x} }^\infty (y-x) y^{1/4}e^{ \pm\frac{\gamma x}{2\tau}-\frac{(3y-x) \sqrt{x} }{3\sqrt{\tau} } -\frac{9 K(y) x}{19\tau}} dy
	= \Boh\left(   e^{-\frac{4(1+\epsilon')}{3\sqrt{ \tau}}x^{3/2} } \right) 
\end{split}
\end{equation*}
for some $\epsilon'>0$ due to the term involving $K(y)$. 
The final integral is bounded by a constant times 
\begin{equation*}
\begin{split}
	\int_{x+a\sqrt{x} }^\infty (y-x) e^{ -\frac{18 K(y) x}{19\tau}} dy
	= \Boh\left(   e^{-\frac{4(1+\epsilon')}{3\sqrt{ \tau}}x^{3/2} } \right) .
\end{split}
\end{equation*}
Thus, $I_2=\Boh\big(  x e^{-\frac{4}{3\sqrt{ \tau}}x^{3/2} - cx } \big)$. 
We compare this with $I_1$. 
Using Lemma~\ref{prop:right_tail_FGUE} the leading terms in $I_1$ are of order $\CB\left(\frac{x}{\tau^{1/3}}+\frac{1}{16\tau^{4/3}}\right)=\Boh\left(e^{-\frac{4}{3\sqrt{\tau}}x^{3/2}}\right)$.
Hence, $I_2$ is exponentially smaller than $I_1$, completing the proof.

\section{Integrable differential equations}\label{sec:integrable_structure}

As mentioned before, from Lemmas~\ref{lem:integrability_operator_H} and \ref{lem:integrability_operator_F} we learn that the operators $\mb H$ and $\mb F$ are IIKS-integrable \cite{IIKS}, and it is now the time to explore this structure.

It is known that the resolvent of an integrable operator can be evaluated in terms of a canonically associated Riemann-Hilbert problem (RHP) \cite{IIKS, Deift99b}.
The main difference of $\mb F$ and $\mb H$ is that $\mb F$ acts on a space $L^2$ with absolutely continuous reference measure, whereas $\mb H$ acts on a space $\ell^2$ with discrete reference measure. Consequently the RHP for the former is continuous (involves only jump conditions) whereas the RHP for the latter is purely discrete (involves only residue conditions). 
The survey paper \cite{Deift99b} reviews many aspects of the theory of continuous IIKS operators. 
Properties of discrete IIKS integrable operators were analyzed by Deift, in a work of Borodin \cite[Section 4]{Borodin_Deift_2000}. In this Section we explore this integrable structure to obtain Theorem~\ref{thm:integrable_systems}.

Riemann-Hilbert problems, whether discrete, continuous or mixed, often arise in the inverse scattering transform in integrable differential equations. The common cubic polynomial factor $V$ in \eqref{def:functionV} is the key to recognize the associated Riemann-Hilbert problems for $\mb H$ and $\mb F$ both as the ones that arise from coupled mKdV equations and coupled nonlinear heat equations, yielding Theorem~\ref{thm:integrable_systems}. This approach of using integrable operators and Riemann-Hilbert problems for finding differential equations was also used in many other problems; see, for example, \cite{Deift-Its-Zhou97, Deift99b, Borodin_Deift_2000, Borodin03, Bertola-Cafasso-2012b}. 

\subsection{Riemann-Hilbert problem}

As a first step, we describe the Riemann-Hilbert Problems (RHPs) associated to $\mb F$ and $\mb H$. For that, recall that $\Lambda$ and $\mathcal S$ were defined in \eqref{def:Lambda_bethe} and $V$ was defined in \eqref{def:functionV}.

The RHP asks for finding $2\times 2$ matrices $X$ (which is associated to $\mb H$) and $Y$ (which is associated to $\mb F$) with the following properties.
\begin{enumerate}[\bf RHP(a)]
\item The entries of $X$ are analytic on $\C\setminus \mathcal S$ and the entries of $Y$ are analytic on $\C\setminus \Lambda$.
\item $X$ has a simple pole at every $\xi\in \mathcal S$, and its residue satisfies
\begin{equation*}
\res_{s=\xi}X(s)=\lim_{s\to\xi}X(s)R_X(s),\quad \text{with}\quad R_X(s):=\vec f(s)\vec g(s)^T.
\end{equation*}
The matrix $Y$ has continuous boundary values $Y_\pm(s)$ as $s$ approaches $\Lambda$ from its $\pm$-side, and they are related by
\begin{equation*}
Y_+(s)=Y_-(s)J_Y(s),\quad \text{with} \quad J_Y(s):=I-\vec a(s)\vec b(s)^T,
\end{equation*}
where $I$ is the $2\times 2$ identity matrix.

\item As $s\to\infty$ the matrices $X$ and $Y$ admit asymptotic expansions of the form
\begin{equation}\label{eq:asymptotics_X}
X(s)\sim I+\sum_{k=1}^\infty \frac{X_k}{s^k}\quad \text{and}\quad Y(s)\sim I+\sum_{k=1}^\infty \frac{Y_k}{s^k}
\end{equation}
for matrices $X_k$ and $Y_k$ that are constant in $s$.
\end{enumerate}

The condition {\bf RHP(c)} for $X$ should be understood as $s\to \infty$ uniformly away from $\mathcal S$. 

For $\gamma=0$ and $\tau=1$, the RHP for $Y$ above coincides with a particular case of the ones studied in \cite{Bertola-Cafasso-2012b,Bertola-Cafasso-2012a}.

Since $\vec g(s)^T \vec f(s)=0$, we find that $R_X(s)^2 = 0$. This implies that the limit $\lim_{\xi\to s}X(\xi)R_X(s)$ in {\bf RHP(b)} converges. 
The residue condition can also be stated in a different way. The following basic result is essentially in \cite[Lemma~4.4]{Borodin_Deift_2000}.

\begin{lem}\label{lem:residue_inverse}
Let $X(\xi)$ be an $r\times r$ matrix function with a simple pole at a point $s\in \C$. 
Let $R$ be a constant $r\times r$ matrix such that $R^2=0$. The following statements are equivalent.
\begin{enumerate}[(i)]
    \item $\displaystyle{\res_s X=\lim_{\xi\to s}X(\xi)R}$
    \item The product $X(\xi) \big( - \frac{R}{\xi-s} + I \big)$ is analytic at $\xi=s$.
\end{enumerate}
\end{lem}

\begin{proof}
Without loss of generality, assume $s=0$ and write $X(\xi)=\frac{X_{-1}}{\xi}+X_0 + \Boh(\xi)$ as $\xi\to 0$. 
Then
$$
	X(\xi) \left(-\frac{R}{\xi} + I \right)=-\frac{X_{-1}R}{\xi^2}+\frac{X_{-1}-X_0 R}{\xi}+\Boh(1),\quad \xi\to 0.
$$
This means that (ii) holds true if and only if $X_{-1} R =0$ and $X_{-1}-X_0 R=0$. 
Since $R^2=0$, these two equations are equivalent to the single equation $X_{-1}=X_0R$, which is equivalent to the statement (i). 
\end{proof}

We now state the solvability of the RHPs. 

\begin{lem} 
For any real values of $x,\gamma$ and $\tau$, the RHP for $Y$ has a unique solution.
The  RHP for $X$ has a solution for all $(\tau, \gamma, x)$ and $0<|z|<1$ satisfying $\det(\I-\mb K_z)\neq 0$ which is an open set of full Lebesgue measure in the space of parameters $(\tau,\gamma,x,z)$. 
Also, if $X$ exists, then it is unique and $\det X\equiv 1$. 
\end{lem}

\begin{proof}
From Lemma~\ref{lem:integrability_operator_F} we know that $\det(\mb I-\mb F)=\det(\mb I-\mb A_{-\gamma}\mb A_{\gamma})$, and the latter is never zero because it coincides with the classical Airy determinant (see \eqref{eq:kpAgamma}). Consequently $\mb I-\mb F$ is invertible and by \cite{Deift99b} the RHP for $Y$ has a unique solution.

By \cite{Borodin_Deift_2000}, the RHP for $X$ has a unique solution if $\mb I-\mb H$ is invertible. 
Hence, the first statement for $X$ follows since $\det(\I-\mb K_z)$ is analytic in all four variables. 
On the other hand, from $R_X(s)^2 = 0$ we obtain that $\det(I-(\xi-s)^{-1}R)\equiv 1$ and Lemma \ref{lem:residue_inverse} implies that $\det X(\xi)$ is an entire function. 
Because $\det X(\xi) \to 1$ as $\xi\to \infty$, Liouville's Theorem implies that $\det X\equiv 1$. 
\end{proof}

Since $\det X\equiv 1$ and $\det Y\equiv 1$, from {\bf RHP(c)} we find that $\Tr X_1=0$ and $\Tr Y_1=0$. We denote the entries of $X_1$ and $Y_1$ as 
\begin{equation}\label{def:notation_X1}
	X_1:=
\begin{pmatrix}
\Qp & \Pp \\
\Rp & -\Qp
\end{pmatrix}, \qquad \text{and} \qquad 
Y_1:=
\begin{pmatrix}
\Qk & \Pk \\
\Rk & -\Qk
\end{pmatrix}
\end{equation}
where $\Qp,\Pp, \Rp$ are functions of $x, \gamma, \tau$, and also $z$, whereas $\Qk,\Pk,\Rk$ are functions of $x,\gamma,\tau$ but not of $z$.

If the integrable operator has some additional structure, then the logarithmic derivative of the Fredholm determinant can be expressed directly in terms of the solution to the associated RHP.
In our case, we show that the logarithmic derivative of the Fredholm determinants in the $x$ variable can be expressed in terms of \eqref{def:notation_X1} as in the next lemma.

\begin{lem}[Deformation formula] \label{prop:RHPforH}
We have 
\begin{equation}\label{eq:identity_fred_det_residue}
	\partial_x \log \det(\mb I-\mb H)= \Qp 
	\quad \text{and} \quad \partial_x \log \det(\mb I-\mb F)=\Qk.
\end{equation}
\end{lem}
\begin{proof}
The proof is standard, and various versions of this formula for continuous RHPs have appeared, for example, in \cite{Deift-Its-Zhou97, Borodin_Deift_2000, Bertola-Cafasso-2012a}. For convenience of the reader in Appendix \ref{sec:proofofdeformation} we provide a proof in the discrete case.
\end{proof}

\subsection{Coupled integrable differential equations}

From the explicit form of the vectors $\vec a,\vec b,\vec f$ and $\vec g$ given in Lemmas~\ref{lem:integrability_operator_H} and \ref{lem:integrability_operator_F}, the matrices $R_X$ and $J_Y$ can be brought to the form
$$
R_X(s)=e^{-\frac{1}{2}\mathcal V(s)\sigma_3}R_0(s)e^{\frac{1}{2}\mathcal V(s)\sigma_3} \quad \text{and}\quad J_Y(s)=e^{-\frac{1}{2}\mathcal V(s)\sigma_3}J_0(s)e^{\frac{1}{2}\mathcal V(s)\sigma_3}
$$
for matrices $R_0(s)$ and $J_0(s)$ which do not depend on $\tau,\gamma,\tau$, and the choice $\mathcal V(s)=V(s;\tau,\gamma,x)$ with $V$ as \eqref{def:functionV}.

For general $\mathcal V$, the dressing method for RHPs consists of, starting from a residue or jump matrix as above, finding an integrable differential equation related to the solution $W=X,Y$ as follows.
Setting $W_0=We^{-\mathcal V \sigma_3/2}$, the new matrix $W_0$ has the residue matrix $R_0$ (or jump matrix $J_0$) which does not depend on the parameters of $\mathcal V$. 
So the derivative of $W_0$ with respect to these parameters satisfies the same RHP as $W$ itself, apart from different asymptotics as $s\to\infty$. 
This implies that $(\partial W_0)W_0^{-1}$ has only removable singularities, and using Liouville's theorem, we find linear differential equations for $W_0$ on the parameters of $\mathcal V$.
When these differential equations for two different parameters are combined they form a Lax pair for $W$. The specific form of this Lax pair depends solely on $\mathcal V$ but not on whether $W$ solves a discrete or continuous RHP.
After some work, the compatibility condition for this Lax pair reduces to nonlinear equations/systems on the entries of $W_1=X_1,Y_1$.

Carrying out the general ideas just outlined, by picking different rational exponents $\mathcal V$ one can produce different integrable systems, including (modified) Korteweg-de Vries equation, nonlinear Schr\"odinger equation, or even more generally the AKNS system \cite{AKNS-74} that contains the former and many other integrable differential equations. 

In our specific context, we can suppress the $\gamma$ term from $V$ into $R_0$ and $J_0$ and write
$$
	R_X(s) = e^{-\frac12( -\frac{\tau}3 s^3+ xs) \sigma_3 } R_1(s) e^{\frac12(  -\frac{\tau}3 s^3+ xs) \sigma_3 }\quad \text{and} \quad J_Y(s) = e^{-\frac12( -\frac{\tau}3 s^3+ xs) \sigma_3 } J_1(s) e^{\frac12(  -\frac{\tau}3 s^3+ xs) \sigma_3 }
$$
for matrices $R_1(s)$ and $J_1(s)$ which do not depend on $\tau$ and $x$. The formalism just mentioned then produces that the entries 
of $X_1$ and $Y_1$ solve the coupled mKdV system claimed in Theorem~\ref{thm:integrable_systems}. If, on the other hand, we suppress the $\tau$ term and write
$$
	R_X(s) = e^{-\frac12(  \frac{\gamma}2 s^2+ xs) \sigma_3 } R_2(s) e^{\frac12( \frac{\gamma}2 s^2+ xs) \sigma_3 }\quad \text{and}\quad J_Y(s) = e^{-\frac12(  \frac{\gamma}2 s^2+ xs) \sigma_3 } J_2(s) e^{\frac12( \frac{\gamma}2 s^2+ xs) \sigma_3 }
$$
for matrices $R_2(s)$ and $J_2(s)$ which do not depend on $\gamma$ and $x$, then we obtain the nonlinear heat equations in Theorem~\ref{thm:integrable_systems}.
We stress that these arguments do not rely on the choice of $R_1,R_2, J_1,J_2$ but solely on the choice of the conjugating exponential factors in the above.

In either of the two situations above, for $X$ we are dealing with a purely discrete RHP with infinitely many poles, which indicates that $\Pp$ and $\Rp$ are infinite soliton solutions. Also, we considered the parameter $z$ that defines the set of poles $\mathcal S$ to be fixed. 
In terms of the integrable systems for $X$, the dependence of $\Pp$ and $\Rp$ on $z$ appears in the initial/boundary conditions, and since this set of poles is not, for general $z$, symmetric under the real axis, these solutions $\Pp$ and $\Rp$ are, also for general $z$, complex-valued.

The equations in Theorem~\ref{thm:integrable_systems} should be supplemented by appropriate initial/boundary conditions, which should be read off from the RHP. Finding these initial/boundary conditions is a challenge of its own, and we do not pursue it in this paper.

The calculations that lead to the exact form of the systems \eqref{eq:coupledMKdV}--\eqref{eq:coupledNLS} is a folklore in integrable systems theory, for convenience of the reader we include them in the Appendix~\ref{sec:iiks_general}.

In concrete terms, the outcome of this discussion is summarized as the next result.

\begin{prop}\label{prop:qpreqs}
The entries $q=\Qk$, $r=\Rk$ and $p=\Pk$ of $Y_1$  in \eqref{def:notation_X1} satisfy
\beq \label{eq:2qxpr2}
	q_x= pr .
\eeq
The pair $(p,r)$ satisfies a coupled system of nonlinear heat equations, 
\begin{equation} \label{eq:prxxxtauPDE} 
\begin{cases}
	&2p_\gamma = - p_{xx} - 2p^2r \\
	&2r_\gamma =  r_{xx} + 2pr^2   
\end{cases} 
\end{equation}
and also a coupled system of mKdV equations 
\begin{equation} \label{eq:prxxgammaPDE} 
\begin{cases}
	&3p_{\tau}+ p_{xxx} +  6 prp_x=0  \\
	&3 r_{\tau} + r_{xxx} + 6 pr r_x =0 . 
\end{cases} 
\end{equation}
The same is true for the entries $q=\Qp$, $r=\Rp$ and $p=\Pp$ of $X_1$  in \eqref{def:notation_X1}, provided the solution $X$ of the RHP exists.
\end{prop}

\subsection{From the coupled systems to the KP equation: proof of Theorem \ref{thm:integrable_systems} and Corollary~\ref{thm:integrable_structure}} \label{sec:kp_equation}

For $u=\Up,\Uk$ as in Corollary~\ref{thm:integrable_structure}, we take the derivative of the equation \eqref{eq:identity_fred_det_residue} and use Lemma \ref{lem:integrability_operator_H} (for $u=\Up$) or Lemma~\ref{lem:integrability_operator_F} (for $u=\Uk$), combined with \eqref{eq:2qxpr2}, and write
\beqq
	u= q_x= pr,
\eeqq
with $q=\Qk,\Qp$ etc. Thus, Proposition \ref{prop:qpreqs} implies Theorem \ref{thm:integrable_systems} except \eqref{eq:prgammaminus} which  we prove in the next subsection. 

We now show that the pair of coupled systems \eqref{eq:prxxxtauPDE} and \eqref{eq:prxxgammaPDE}, when combined together, yield the KP equation, hence proving Theorem~\ref{thm:integrable_structure}.
Using \eqref{eq:prxxxtauPDE}, we have $2u_\gamma=p r_{xx}-p_{xx}r= \partial_x (pr_x- p_xr)$. 
Using \eqref{eq:prxxxtauPDE} again and noting $pr (pr_x+p_x r)=  u u_x$, we find that 
\beqq 
	4u_{\gamma\gamma}=\partial_x\left( pr_{xxx} + p_{xxx}r + 4u u_x -(p_xr_x)_x  \right).
\eeqq
On the other hand, multiplying the first equation in  \eqref{eq:prxxgammaPDE} by $r$ and the second by $p$, and adding the results, we get, using $pr (pr_x+p_x r)=  u u_x$ again, 
\beqq 
	3u_{\tau}+pr_{xxx}+ p_{xxx}r + 6 uu_x=0.
\eeqq 
Finally, the third derivative evaluation of $u=pr$ is given by
\beqq 
	u_{xxx}=pr_{xxx}+ p_{xxx}r+3(p_xr_x)_x .
\eeqq 
We combine the above three equations to remove $pr_{xxx}+ p_{xxx}r$ and $(p_xr_x)_x$ and obtain the KP equation 
$$
12u_{\gamma\gamma}+\left(u_{xxx}+12uu_x+12u_{\tau}\right)_x=0. 
$$
This proves Theorem~\ref{thm:integrable_structure}.

In summary, the calculations above show that any RHP with residue or jump matrix of the form 
$$
	M(s) = e^{-\frac12( -\frac{\tau}3 s^3+ \frac{\gamma}{2} s+ xs) \sigma_3 } M_0(s) e^{\frac12(  -\frac{\tau}3 s^3+\frac{\gamma}{2} s+  xs) \sigma_3 }
$$
with a matrix $M_0(s)$ which do not depend on $\tau, \gamma, x$ is associated to the KP equation. 
Rigorous analysis of the scattering transform for the KP equation has been carried out for classes of real solutions that decay sufficiently fast as $x,\gamma\to \infty$, except possibly along a line (see \cite{Ablowitz-Villarroel-04} and the references therein). 
In these situations, the inverse scattering transform always involves a so-called d-bar problem component, possibly also with poles, and a continuous RHP, all for scalar functions instead of matrix-valued functions. 
Our RHP for $Y$ is nothing but a dressed-up version of the RHP for the Hastings-McLeod solution to PII, which is well understood. In contrast, our RHP for $X$ is a purely discrete matrix RHP. Hence, the KP solution encoded in our discrete matrix RHP falls outside the class of the just mentioned works. 
Determining the exact growth/decay properties of our KP solution is a separate task of interest.

\subsection{Symmetry}

We start with the discrete RHP for $X$. Define
$$
	\widehat X(\xi;\gamma) :=\sigma_1 X(-\xi;-\gamma)\sigma_1, \quad \text{with}\quad \sigma_1 = \big( \begin{smallmatrix} 0 & 1 \\ 1 & 0 \end{smallmatrix} \big).
$$
This matrix tends to the identity matrix $I$ as $\xi \to\infty$. 

A direct calculation shows that the residue matrix for $X$ satisfies the symmetry property
\beqq
	R_{X}(-s; -\gamma)= - \sigma_1 R_X(s;\gamma)\sigma_1 .
\eeqq
Hence, $\widehat X(\xi;\gamma)$ satisfies, for every $s\in \mathcal S$, 
\beqq
	\widehat X(\xi;\gamma) \left( -  \frac{R_{X}(s; \gamma)}{\xi-s} +I \right) 
	= \sigma_1  X(-\xi; -\gamma) \left( - \frac{R_{X}(-s; -\gamma)}{(-\xi)-(-s)} +I \right) \sigma_1 .
\eeqq
By the RHP for $X$ and Lemma~\ref{lem:residue_inverse}, the above formula is analytic at $\xi=s$. 
Hence,  Lemma~\ref{lem:residue_inverse} again shows that $\widehat X(\xi;\gamma)$ satisfies the same residue condition as $X(\xi; \gamma)$. 
By the uniqueness of the solution to the RHP, we find that $\widehat X(\xi)=X(\xi)$, thus obtaining the symmetry property 
$$
	X(\xi;\gamma)=\sigma_1 X(-\xi;-\gamma)\sigma_1.
$$
Inserting this into the asymptotic expansion \eqref{eq:asymptotics_X} we find that 
$X_k= (-1)^k \sigma_1 X_k(-\gamma)\sigma_1$ for all $k\geq 1$.
When $k=1$, this means that the coefficients in \eqref{def:notation_X1} satisfy $p(\gamma)=- r(-\gamma)$, 
proving the symmetry \eqref{eq:prgammaminus}. We also have $q(\gamma)=q(-\gamma)$. 

The proof for $Y$ is similar: the jump symmetry is now
$$
\sigma_1 J_Y(-s;-\gamma)\sigma_1 = J_Y(s;\gamma)^{-1},
$$
and this implies that $\sigma_1Y(-s;-\gamma)\sigma_1=Y(s;\gamma)$, which then from \eqref{eq:asymptotics_X} yields the symmetry \eqref{eq:prgammaminus}.

We discussed in Section \ref{sec:intropde} that using the symmetry, the function $p$ itself satisfies the non-local differential equations \eqref{eq:nonlocal_1} and \eqref{eq:nonlocal_2}. 
The equation \eqref{eq:nonlocal_2} is a non-local nonlinear Schr\"odinger (NLS)  equation with complex time variable. The appearance of the non-local NLS from the AKNS system was obtained relatively recently in \cite{Ablowitz-Musslimani-13}, and the inverse scattering transform of this non-local NLS equation has also recently been analyzed (see \cite{Ablowitz-Luo-Musslimani-18} and the references therein) for a class of initial conditions with finite amplitude as $x\to \infty$. 
Among other distinguishing features, in these works the poles of the discrete RHP were symmetric under reflection onto the imaginary axis, which is not the case for our discrete RHP. 
 
\section{Flat initial condition}\label{sec:flat_IC}

All of the results for the step initial condition extend to the flat initial condition, and we now discuss them.

The limit of the one point distribution with the flat initial condition does not depend on the location parameter $\gamma$. 
It is given by the formula \cite{Baik-Liu16} 
\begin{equation}\label{def:limiting_distribution_flat}
	\FF(x;\tau)=\oint e^{xA_1(z)+\tau A_2(z)+A_3(z)+ B(z)}\det(\mb I-\mb K_z^{(1)})\frac{dz}{2\pi i z},
\end{equation}
where the integral is a small circle around the origin in the counterclockwise direction.
The functions $A_1, A_2, B$ are the same as in the step initial condition case, and $A_3(z)=-\frac{1}{4}\log(1-z)$. 
The operator 
$\mb K_z^{(1)}:\ell^2(\mathcal S_-)\to \ell^2(\mathcal S_-)$ is defined via the kernel
\begin{equation}\label{eq:initial_kernel_K1}
	\bm K_z^{(1)}(\xi_1,\xi_2) =\bm K_z^{(1)}(\xi_1,\xi_2; x,\tau)= \frac{e^{\Psi(\xi_1)+\Psi(\xi_2)}}{\xi_1(\xi_1+\xi_2)}, 
	\qquad \xi_1, \xi_2\in \mathcal S_- ,
\end{equation}
where the function $\Psi$ is given by
\beq \label{def:function_Psi}
	\Psi(\xi)=\Psi(\xi;x,\tau):=-\frac{1}{3}\tau\xi^3 +x\xi -\frac{1}{2}Q(\xi)
	= \frac12 \Phi(\xi; 2x, 2\tau)
\eeq
compared with the function $\Phi$ in \eqref{def:phi_function}. 
We have the following results.

\begin{thm}[Large $\tau$ limit; flat case] \label{thm:limits_FF}
For each fixed $x\in\R$, 
	\begin{equation*} 
	\lim_{\tau\to\infty} \FF\left( -\tau +\frac{\pi^{1/4}}{\sqrt{2}}x\tau^{1/2}; \tau\right)
	=\frac{1}{\sqrt{2\pi}} \int_{-\infty}^x e^{-\frac{y^2}{2}}dy.
	\end{equation*}
\end{thm}

\begin{thm}[Small $\tau$ limit; flat case] \label{thm:smallflat}
For every $x_0\in \R$, 
there exist constants $C,c, \epsilon>0$  and $\tau_0>0$ such that
$$
\left| \FF\left(\tau^{1/3}x;\tau\right)-\FGOE(2^{2/3}x)\right|\leq C e^{-\epsilon \tau^{-2/3}}e^{-c(x-x_0)},\qquad 0<\tau <\tau_0, \quad x\geq x_0,
$$
where $\FGOE$ is the GOE Tracy-Widom distribution. In particular,
	\begin{equation*}
		\lim_{\tau\to0}\FF(\tau^{1/3}x;\tau)=\FGOE(2^{2/3}x).
	\end{equation*}	
\end{thm}

\begin{thm}[Right tail estimate; flat case]\label{thm:right_tail_FF}
For every fixed $\tau>0$, 
there is a constant $c>0$ such that
\beqq
	1- \FF(x;\tau) = (1- \FGOE(2^{2/3}\tau^{-1/3}x))  \left(1+\Boh(e^{-cx^{1/2}})  \right),\qquad x\to\infty.
\eeqq 
\end{thm}

The above result shows that the leading term of the right tail of $\FF(\tau^{1/3}x;\tau)$ does not depend on $\tau$, and it is the same as the right tail of $\FGOE(2^{2/3}x)$, which can be found, for example, in \cite{Baik-Buckingham-DiFranco08}. 

\begin{thm}[Integrable differential equations; flat case] \label{thm:intflat}
The kernel $\mb K^{(1)}$ with $x, \tau$ replaced by $\frac{x}2, \frac{\tau}2$ satisfies 
\beqq \begin{split}
	\partial_x \ln \det(\mb I - \mb K_z^{(1)}|_{\frac{x}2, \frac{\tau}2})
	= \frac12 R(\tau, x)+  \frac12 Q(\tau, x)
\end{split} \eeqq
where $R(\tau, x)= \Rp(\tau, 0, x)$ and $Q(\tau, x)=\Qp(\tau, 0, x)$ in terms of the functions in Theorem~\ref{thm:integrable_systems} and Proposition \ref{prop:qpreqs}. 
The function $R$ satisfies the (defocusing) mKdV equation 
\beqq
	3 R_{\tau} + R_{xxx} - 6 R^2 R_x =0 \qquad \text{and} \quad  Q_x= - R^2. 
\eeqq
\end{thm}

\begin{cor}
The function 
\beqq \begin{split}
	U(\tau, x):= 2\partial_{xx} \ln \det(\mb I - \mb K_z^{(1)}|_{\frac{x}{2},\frac{\tau}{2}})
\end{split} \eeqq
satisfies the Korteweg--de Vries equation (KdV) equation\footnote{The scale $\tau\mapsto 3\tau$ changes to a standard form of the KdV equation $U_{\tau} + U_{xxx} + 6 U U_x =0$.} 
\beqq
	3 U_{\tau} + U_{xxx} + 6 U U_x =0.
\eeqq
\end{cor}

\begin{proof}
We have $U= R_x + Q_x= R_x-R^2$. It is straightforward to check that if $R$ satisfies the mKdV equation, then $U$ satisfies the KdV equation. Indeed, this is a well-known Miura transformation.
\end{proof}

We prove the above theorems in the following subsections.

\subsection{Large $\tau$ limit}

For $x_\tau$ as in \eqref{eq:x_scaling_large_tau}, we use \eqref{def:limiting_distribution_flat}--\eqref{def:function_Psi} and write
\beqq
	\FF ( x_\tau/2; \tau/2 ) 
	=\oint e^{\frac{x_\tau}{2} A_1(z)+ \frac{\tau}{2} A_2(z)+A_3(z)+ B(z)}
	\det(\I -  \mb K_z^{(1)}|_{\frac{x_\tau}{2}, \frac{\tau}{2}}) \frac{dz}{2\pi i z},
\eeqq
emphasizing the explicit expression
\beqq
	\bm K_z^{(1)}|_{\frac{x_\tau}{2}, \frac{\tau}{2}} (\xi_1, \xi_2)=  \frac{e^{\frac12 \Phi_\tau(\xi_1)+ \frac12 \Phi_\tau(\xi_2)}}{\xi_1(\xi_1+\xi_2)}, 
	\qquad \xi_1, \xi_2\in \mathcal S_- ,
\eeqq
with $\Phi_\tau(\xi):=\Phi(\xi; x_\tau, \tau)$ as in \eqref{eq:leading_term_Phi}.

As in Lemma~\ref{lem:kernel_sum_1}, let $\mathcal S_-=\{u_k\}_{k\in \Z}$ and let $\ell_1, \cdots, \ell_N$ be distinct integers. 
Corollary~\ref{prop:estimate_real_phi_uk_u0} and Lemma \ref{lem:uuplusu} imply that 
\beqq
	\sum_{k=1}^N \left| \bm K_z^{(1)}|_{\frac{x}{2}, \frac{\tau}{2}} (\xi, u_{\ell_k}) \right|^2
	\le \frac{1}{c^2} e^{\re \Phi_\tau(\xi)}  \sum_{k=-\infty}^\infty \frac{e^{\re \Phi_\tau(u_k)}}{|k|+2},\quad \xi\in \mathcal S_-,
\eeqq
and we further estimate
\beqq
\sum_{k=-\infty}^\infty \frac{e^{\re \Phi_\tau(u_k)}}{|k|+2}
	\leq \frac{e^{0.4734 \tau}}{2} + 2 \sum_{k=1}^\infty e^{-1.9366|k|^{3/2} \tau}
	\le e^{0.4734 \tau}
\eeqq
for all large enough $\tau$ and every $z$ as in Corollary~\ref{prop:estimate_real_phi_uk_u0}. Thus, there are $\tau_0$ and $C>0$ such that 
\beqq
	\sum_{k=1}^N \left| \bm K_z^{(1)}|_{\frac{x}{2}, \frac{\tau}{2}} (\xi, u_{\ell_k}) \right|^2
	\le C e^{ \re \Phi_\tau(\xi)+0.4734 \tau}
\eeqq
uniformly for $\xi\in \mathcal S_-$ and $\tau\ge \tau_0$; this estimate is similar to Lemma~\ref{lem:kernel_sum_1}.
Similar arguments as used in the proof of Proposition~\ref{thm:localization_fred_det} now show that
\beqq
	\det(\I - \mb K_z^{(1)}|_{\frac{x}{2}, \frac{\tau}{2}}) = 1 - \Tr (\mb K_z^{(1)}|_{\frac{x}{2}, \frac{\tau}{2}}) 
	+ \Boh(e^{-0.255 \tau})
\eeqq
as $\tau\to \infty$, uniformly for $1-\delta \le |z|\le 1-\delta_1$ for some $\delta\in (0,1)$ and any fixed $\delta_1\in (0,\delta)$. 
Here, the exponent in the error is half the one from Proposition~\ref{thm:localization_fred_det} due to the change from $\Phi_\tau$ to $\frac12 \Phi_\tau$ in the kernel. 

We look now at the proof of Proposition \ref{prop:Flgtu}. 
The estimate \eqref{eq:estimate_factor} implies that for every $R\in (0,1-\delta_1)$,
\beqq
	\max_{|z|=R} \re\left(\frac{x_\tau}{2} A_1(z)+\frac{\tau}2 A_2(z)+A_3(z)+ B(z)\right)<0.254\tau
\eeqq
for sufficiently large $\tau$, and Proposition \ref{prop:Flgtu} changes to 
\beqq
	\FF(x_\tau/2, \tau/2)  = 1 - \oint e^{\frac{x_\tau}{2} A_1(z)+ \frac{\tau}{2} A_2(z)+A_3(z)+ B(z)}
	\Tr (  \mb K_z^{(1)}|_{\frac{x_\tau}{2}, \frac{\tau}{2}}) \frac{dz}{2\pi i z} + \Boh(e^{-0.001\tau})
\eeqq
as $\tau \to\infty$. From the formula of $\mb K_z^{(1)}$,  the version of Lemma~\ref{prop:large_tau_estimate_trace} now follows from the simpler estimate
$$
\Tr \mb K_z^{(1)}|_{\frac{x_\tau}{2}, \frac{\tau}{2}}=\frac{e^{\Phi_\tau(u_0(z))}}{2u_0(z)^2} +\sum_{k\in \Z\setminus\{0\}}\frac{e^{\Phi_\tau(u_k)}}{2u_k^2} =\frac{e^{\Phi_\tau(u_0(z))}}{2u_0(z)^2} +\Boh(e^{-1.9 \tau}),
$$
where we error term was obtained using Corollary~\ref{prop:estimate_real_phi_uk_u0} and Lemma~\ref{lem:uuplusu}.

Plugging in the formulas and using the notations \eqref{eq:ff12E}, 
$$
    \int  e^{ \frac{x_\tau}2 A_1(z)+ \frac{\tau}2 A_2(z) +A_3(z)+B(z)} \frac{e^{\Phi_\tau(u_0(z))}}{2u_0(z)^2} \frac{dz}{2\pi iz}
    = \frac{1}{4\pi i} \int \frac{E^{(1)}(z)e^{\frac{1}{2}\tau \ff_1(z)+\frac{1}{2} \tau^{1/2}\hat x \ff_2(z)}}{zu_0(z)^2} dz 
$$
where $ \hat x= \frac{\pi^{1/4} \rx}{2}$ as before, $\rx\in \R$ fixed, and 
\beqq
    E^{(1)}(z):=e^{A_3(z)+B(z)-Q(u_0(z))} =  \frac{E(z)^{1/2}}{(1-z)^{1/4}}, \qquad E(z)=e^{2B(z)-2Q(u_0(z))} \text{ as before}.
\eeqq 
The functions $\ff_1$ and $\ff_2$ have analytic extensions, $\ffRS_1$ and $\ffRS_2$, on the Riemann surface $\mathcal{R}$ (see Section \ref{sec:f12exRS}). 
On the other hand, the function $E(z)$ also extends analytically on $\mathcal{R}$ and has  a zero of order $3$ at $z=1$ (see Section \ref{sec:EextRS}).  
Combined with $(1-z)^{-1/4}$, we see that $E^{(1)}(z)$ extends analytically on $\mathcal R$ and it has a zero of order $1$ at $z=1$. 
We denote the extension by $\EERS^{(1)}(z)$. Recalling Lemma~\ref{prop:local_expansion_prefactor_gaussian}, $\EERS(z)\approx -4u_0(z)^4$ as $z\to 0^{(2)}$. It is also direct to verify that $(1-z)^{-1/4} = -i$ as $z\to 0^{(2)}$. Combining everything, 
$\EERS^{(1)}(z)= -2u_0(z)^2 (1+ \Boh (z^2))$ as $z\to 0^{(2)}$. 
Define the one form (cf. \eqref{eq:oneform_omega})
\begin{equation*}
    \omega_\tau^{(1)} :=\frac{\EERS^{(1)}(z) e^{\frac{\tau}{2} \ffRS_1(z)+ \frac{\tau^{1/2} \hat x}{2} \ffRS_2(z)}}{4\pi i z \URS_0(z)^2} dz
\end{equation*}
on the Riemann surface $\mathcal R$. 
Using this notation, we just verified that (cf. \eqref{eq:Finomega})
\beqq
	\FF(x_\tau/2, \tau/2)  = 1 - \int_{\Gamma} \omega_\tau^{(1)} + \Boh(e^{-0.001\tau})
\eeqq
where $\Gamma$ is the same contour that was used in~\eqref{eq:Finomega}.
We evaluate the integral following Section \ref{sec:asymppolyintg}. The contour is divided into three parts. 
The proofs of Lemma \ref{lem:integral_gamma3} and \ref{lem:GL_int_gamma1} for the parts $\Gamma_1$ and $\Gamma_2$ apply with minimal changes, and contributions from these parts tend to zero as $\tau\to \infty$. 
On the other hand, for the contour $\Gamma_3$, which is a small semicircle about the origin in the second sheet of $\mathcal R$, 
most of the proof of Lemma~\ref{lem:integral_gamma5} does not change, but in the last steps we find 
\beqq
	\lim_{\tau\to \infty} \int_{\Gamma_3} \omega_\tau^{(1)} = \frac1{2\pi i} \int_{-i\infty}^{i\infty}  e^{\frac{1}{2}\left(\frac{1}{8 \sqrt{\pi}} s^2  - \frac{\hat x}{\sqrt{2\pi}} s\right)}  \frac{ds}{s} 
	= \frac{1}{2\pi i}\int_{-\infty}^{\infty} e^{-\frac{u^2}{2}-i\frac{\rx}{\sqrt{2}} u}\frac{du}{u}=\frac{1}{\sqrt{2\pi}}\int_{\frac{\rx}{\sqrt{2}}}e^{-\frac{1}{2}u^2}du,
\eeqq
where for the second inequality we changed variables $s=i2^{3/2}\pi^{1/4}u$ and used that $\hat x= \pi^{1/4}\rx/\sqrt{2}$. Having in mind the explicit expression for $x_\tau$ in \eqref{eq:x_scaling_large_tau}, we just obtained that
$$
\lim_{\tau\to 0}\FF\left( -\frac{\tau}{2}+\frac{\pi^{1/4}}{\sqrt{2}} \frac{\rx}{\sqrt{2}} \left(\frac{\tau}{2}\right)^{1/2},\frac{\tau}{2}  \right)
=\lim_{\tau\to 0}\FF\left( -\tau+\frac{\pi^{1/4}}{\sqrt{2}} \frac{\rx}{\sqrt{2}} \tau^{1/2}, \tau \right)=\frac{1}{\sqrt{2\pi}}\int_{-\infty}^{\frac{\rx }{\sqrt{2}}}e^{-\frac{u^2}{2}}du,
$$
valid for any $\rx \in \R$ fixed. This finishes the proof of Theorem \ref{thm:limits_FF}.

\subsection{Small $\tau$ limit}

After a simple conjugation, we may replace the kernel $\bm K_z^{(1)}(\xi_1,\xi_2)$ in \eqref{eq:initial_kernel_K1} 
by (using the same notation)
\begin{equation*}
	\bm K_z^{(1)}(\xi_1,\xi_2) = - \frac{e^{\Psi(\xi_1)+\Psi(\xi_2)}}{\sqrt{-\xi_1}\sqrt{-\xi_2} (\xi_1+\xi_2)}, 
	\qquad \xi_1, \xi_2\in \mathcal S_-,
\end{equation*}
without changing the Fredholm determinant. 
Using \eqref{def:function_Psi} and the notation from \eqref{eq:kernel_K1}, we see that
\beq \label{eq:KflJ}
	\bm K_z^{(1)}\left(\xi,\eta; \frac{x}{2},\frac{\tau}{2}\right)
	= - \frac{e^{\frac{1}{2} \Phi(\xi; x, \tau) + \frac12 \Phi(\eta; x, \tau)}}{\sqrt{-\xi}\sqrt{-\eta} (\xi+\eta)}
	= - \bJ (\xi, \eta; x, \gamma=0, \tau),\quad \text{that is,}\quad \mb K_z^{(1)}|_{\frac{x}2, \frac{\tau}2}=-\mJ|_{\gamma=0} .
\eeq
Now from \eqref{def:operators_K1_K2} we have $\mJ|_{\gamma=0}=-  \mb G_0 \mb G_{0}^T$, and recalling that 
$\mb \Tgamma=\mb G_{-\gamma}^T\mb G_{-\gamma}$ we find that
\beq \label{eq:KflatT0}
	\det (\I - \mb K_z^{(1)}|_{\frac{x}2, \frac{\tau}2}) = \det(\I - \mb T_0)
\eeq
where $\mb T_0$ has kernel as in Lemma~\ref{lem:kernelK_kernelT} with $\gamma=0$. 
Thus, changing $x$ to $\tau^{1/3} x$, we find that
\beqq
	\FF \left( \frac{\tau^{1/3}x}{2}; \frac{\tau}2 \right)=\oint e^{\frac{\tau^{1/3}x}{2} A_1(z)+ \frac{\tau}{2} A_2(z)+A_3(z)+ B(z)}
	\det(\mb I-\mb A_{x,\tau}) \frac{dz}{2\pi i z}
\eeqq
where the operator $\mb A_{x,\tau}$ is defined by \eqref{def:kernel_operator_A}. 
Recall that $F_{\mathrm{GOE}}(x)= \det(\I - \mb A_{x})$. 
Lemma \ref{prop:convergence_kernels_tw_scaling} established the convergence of $\mb A_{x,\tau}$ to the operator $\mb A_{x}$ with an explicit error bound for the kernel. 
Hence, we find that the Fredholm determinant converges as well with the same error bound as in Corollary \ref{prop:convergence_fred_det_tw}. 
Taking the integral (see Section \ref{sec:comppfsmall}), we find that 
$$
	\left| \FF \left( \frac{\tau^{1/3}x}{2}; \frac{\tau}2 \right)-F_{\mathrm{GOE}}(x)\right|\leq C e^{-\epsilon \tau^{-2/3}}e^{-c(x-x_0)} . 
$$
Changing $\tau\mapsto 2\tau$, and $x\mapsto 2^{2/3}x$, and renaming the constants, we conclude the proof of Theorem~\ref{thm:smallflat}.

\subsection{Right tail}

From \eqref{eq:KflatT0},
\beqq
	\FF ( x/2; \tau/2 ) 
	=\oint e^{\frac{x}{2} A_1(z)+ \frac{\tau}{2} A_2(z)+A_3(z)+ B(z)}
	\det(\I -  \mb T_0) \frac{dz}{2\pi i z} .
\eeqq
As in Section \ref{sec:large_tail}, we take the contour $|z|=e^{-\frac{x}{2\tau}}$. 
Proposition \ref{prop:LT_HSnorm_est} implies that we may change the Fredholm determinant to $1-\Tr(\mb T_0)$ plus an error term. 
The kernel is given by $\bm T_0(s,t)= \CT_{0}(x+ s+t)$ (see \eqref{eq:T_to_CA}). 
Proposition \ref{prop:asymptotics_CA} and~\ref{prop:bound_CA} allow us to replace $\CT_{0}(x+ s+t)$ by $\RA(x+s+t; 0)$ whose asymptotic formula is given in Lemma~\ref{lem:asymptotics_g_large_tail}. 
Thus, after performing the trivial $z$-integral, we find that 
\beqq
	1- \FF( x/2; \tau/2 ) 
	=\int_x^\infty \frac{1}{\tau^{1/3}} \ai \left( 2\tau^{-1/3} y  \right) dy \left( 1+ \Boh(e^{-cx^{1/2}}) \right) .
\eeqq
We have 
\beqq
	\int_x^\infty \frac{1}{\tau^{1/3}} \ai \left( 2\tau^{-1/3} y  \right) dy 
	= \int_x^\infty  \ai \left( 2 y  \right) dy = \Tr(\mb A_{x} ) =(  1- F_{\mathrm{GOE}}(x) )\left( 1+ \Boh(e^{-cx}) \right). 
\eeqq
Scaling $x$ and $\tau$ completes the proof  of Theorem~\ref{thm:right_tail_FF}. 

\subsection{Integrable differential equations} 

From \eqref{eq:KflJ}, 
$$
	\det (\I - \mb K_z^{(1)}|_{\frac{x}2, \frac{\tau}2})  = \det( \I +  \mJ|_{\gamma=0}).
$$
The operator for the step initial condition case satisfies $\mb K_z= \mJ \mJ^T$ (see \eqref{eq:split_kernel_1}). 
Because of the symmetry $\mJ^T|_\gamma= \mJ|_{-\gamma}$ (see \eqref{eq:Jsymr}), we 
find that 
$$
	\mb K_z|_{\gamma=0}= (\mJ|_{\gamma=0})^2.
$$

For all the computations in the rest of this section we set $\gamma=0$, and in particular for simplicity we write $\mb K_z|_{\gamma=0}$ and $\mJ|_{\gamma=0}$ simply $\mb K_z$ and $\mJ$. Also, we use the notation $\partial=\partial_x$.

We have 
$$
	\partial \ln \det (\I - \mb K_z^{(1)}|_{\frac{x}2, \frac{\tau}2})  = \partial \ln \det(\mb I + \mJ) =  \Tr ((\mb I + \mJ)^{-1} \partial \mJ)
	= \Tr ((\mb I- \mJ^2)^{-1}(\mb I - \mJ) \partial \mJ).
$$
Using that $(\mb I-\mJ^2)^{-1}$ and $\mJ$ commute, we write 
\beq \label{eq:imjimjs} \begin{split}
	\partial \ln \det (\I - \mb K_z^{(1)}|_{\frac{x}2, \frac{\tau}2}) 
	&= \Tr ((\mb I- \mJ^2)^{-1}\partial \mJ) - \frac12 \Tr ((\mb I- \mJ^2)^{-1}(\mJ (\partial \mJ)+ (\partial \mJ) \mJ)) \\
	&= \Tr ((\mb I- \mJ^2)^{-1}\partial \mJ) + \frac12 \partial \ln \det( \mb I- \mJ^2) . 
\end{split} \eeq
From Lemmas \ref{lem:integrability_operator_H} and \ref{prop:RHPforH} we know that
$$
\partial \ln \det( \mb I- \mJ^2)= \partial \ln \det(\mb I - \mb K_z)=\partial \ln \det(\mb I - \mb H)= \Qp(\tau,0, x)
$$ 
Hence, it is enough to consider 
\beqq
	\Tr ((\mb I- \mJ^2)^{-1}\partial \mJ).
\eeqq

We use the following lemma. 

\begin{lem}
Let $\mathcal H_1$ and $\mathcal H_2$ be two Hilbert spaces. Suppose that $\mb K_1: \mathcal H_2 \to \mathcal H_1$, $\mb K_2:\mathcal H_2\to \mathcal H_1$, and $\mb K_3: \mathcal H_1\to \mathcal H_1$ are trace class operators.  
Set $\mathcal H=\mathcal H_1\oplus \mathcal H_2$. 
We extend $\mb K_1, \mb K_2, \mb K_3$ to operators on $\mathcal H$ in a natural way, and we use the same notations. 
Let $\mb P_1,\mb  P_2: \mathcal H\to \mathcal H$ be the projection operators to $\mathcal H_1$ and $\mathcal H_2$, respectively. 
Then,
\beqq
	\Tr ((\mb I - \mb K_1\mb K_2)^{-1}\mb K_3)_{\mathcal H_1}= \Tr ((\mb I- \mb P_1 \mb K_1 \mb P_2 - \mb P_2 \mb K_2 \mb P_1)^{-1} \mb P_1 \mb K_3 \mb P_1)_{\mathcal H}
\eeqq
\end{lem}

\begin{proof}
It is enough to show that the non-zero eigenvalues of both operators are the same. 
This is straightforward to check, we skip the details.
\end{proof}

In the proof of Lemma \ref{lem:integrability_operator_H}, we set $\mb W: \ell^2(\mathcal S_+)\to \ell^2(\mathcal S_-)$ to be the reflection operator $(\mb W h)(u)=h(-u)$ and used the decomposition $\mb J^2= \mb K_z= \mb J_1 \mb J_2$ where $\mb J_1= \mb J \mb W$ and $\mb J_2= \mb W^T \mb J^T=\mb W^T\mb J$, with the last equality being valid because $\gamma=0$.  
Recall the definition $\mathcal S= \mathcal S_1 \cup \mathcal S_2$ and observe that $\ell^2(\mathcal S_-)\oplus \ell^2(\mathcal S_+)=\ell^2(\mathcal S)$. 
Denoting by $\mb P_\pm$ the projections onto $\mathcal S_\pm$, the above lemma implies that
\beqq
	\Tr ((\mb I- \mJ^2)^{-1}\partial \mJ)_{\ell^2(\mathcal S_-)}
	= \Tr ((\mb I- \mb P_- \mJ_1 \mb P_+ - \mb P_+\mJ_2 \mb P_- )^{-1} \mb P_- \partial \mJ \mb P_-)_{\ell^2(\mathcal S)}
	= \Tr ((\mb I - \mb H)^{-1} \mb P_- \partial \mJ \mb P_-)_{\ell^2(\mathcal S)}
\eeqq
where $\mb H= \mb P_- \mJ_1 \mb P_+ + \mb P_+\mJ_2 \mb P_-$ is the same operator as in Lemma \ref{lem:integrability_operator_H} (with $\gamma=0$). 

Since $\Phi(s)= -\frac13 \tau s^3+ x s- Q(s)$ and recalling that $\partial=\partial_x$, we find from direct differentiation that 
\beqq
	\chi_-(u) \partial \bJ (u,v) \chi_-(v) 
	= \frac12 \chi_-(u) \frac{e^{\frac12 \Phi(u) + \frac12 \Phi(v)}}{\sqrt{-u}\sqrt{-v}} \chi_-(v)
	= \frac12 f_2(u) g_1(v),
\eeqq
where $f_i$ and $g_i$ are the components of the vectors $\vec f$ and $\vec g$ in Lemma \ref{lem:integrability_operator_H}. Hence, 
\beqq \begin{split}
	\Tr ((\mb I- \mb J^2)^{-1}\partial \mb J)_{\ell^2(\mathcal S_-)}
	= \frac12 \Tr ( (\mb I - \mb H)^{-1} f_2 \otimes g_1) .
\end{split} \eeqq
Using the notation $\vec F := (\mb I- \mb J)^{-1}\vec f$, the above equation becomes 
\beqq \begin{split}
	\Tr ((\mb I- \mb J^2)^{-1}\partial \mb J)_{\ell^2(\mathcal S_-)}
	= \frac12 \Tr (F_2 \otimes g_1) 
	= \frac12  \sum_{s\in \mathcal S} F_2(s) g_1(s) .
\end{split} \eeqq
In formulas \eqref{eq:vecFGandresolv} and \eqref{eq:X1Fg}, we will show that $\sum_{s\in \mathcal S}\vec F(s)\vec g(s)^T=X_1$ is the residue matrix in \eqref{def:notation_X1}. Therefore,
\beqq \begin{split}
	\Tr ((\mb I- \mb J^2)^{-1}\partial \mb J)_{\ell^2(\mathcal S_-)}
	= \frac12  (X_1)_{21} .
\end{split} \eeqq
Recalling the notation $(X_1)_{21}=\Rp$ in \eqref{def:notation_X1} and combining with \eqref{eq:imjimjs}, we obtain 
\begin{equation*}
\begin{split}
	\partial_x \ln \det (\I - \mb K_z^{(1)}|_{\frac{x}2, \frac{\tau}2}) 
	&=\frac12 \Rp(\tau, 0, x)+ \frac12 \Qp(\tau, 0, x) .
\end{split}
\end{equation*}

The coupled system of mKdV equations \eqref{eq:coupledMKdV} and the symmetry $\Pp(\gamma)=-\Rp(-\gamma)$ in \eqref{eq:prgammaminus} implies that $R(\tau, x):= \Rp(\tau, 0, x)$ satisfies the mKdV equation 
\beqq
	3 R_{\tau} + R_{xxx} - 6 R^2 R_x =0.
\eeqq
On the other hand, the equation \eqref{eq:2qxpr2} implies that $Q(\tau, x):= \Qp(\tau, 0, x)$ satisfies 
\beqq
	Q_x= -R^2.
\eeqq
The proof of Theorem \ref{thm:intflat} is completed.

\appendix

\section{Proof of Lemma \ref{prop:RHPforH}}\label{sec:proofofdeformation}

We prove the lemma. 
It is a fundamental property of integrable operators  \cite{Deift99b, Borodin_Deift_2000} that the vector 
\beq \label{eq:vecFGandresolv}
	\vec F := (\mb I- \mb H)^{-1} \vec f
\eeq
can be obtained from the RHP by the formula 
\beq \label{eq:definition_functions_F}
	\vec F(s)=\lim_{\xi\to s} X(\xi)\vec f(s) 
\eeq
where the limit converges. 
Now, the solution of the RHP satisfies the identity 
\begin{equation}\label{eq:solution_X_series}
	X(\xi)=I-\sum_{s\in\mathcal S}\frac{\vec F(s)\vec g(s)^T}{s-\xi} 
	\quad \text{for $\xi\in \C\setminus \mathcal S$},
\end{equation}
with $\vec F(s)$ given by \eqref{eq:definition_functions_F}.
This equation can be seen by checking directly that the right-hand side satisfies the conditions RHP-X (a), (b), and (c)  of the RHP for $X$, and using the uniqueness of the solution. For the condition RHP-X (b), note that 
$$
\res_{\xi=w}\left( I-\sum_{s\in\mathcal S}\frac{\vec F(s)\vec g(s)^T}{s-\xi} \right)=\vec F(w)\vec g(w)^T=\lim_{\xi\to w}X(\xi)\vec f(w)\vec g(w)^T=\lim_{\xi\to w}X(\xi)R_X(w) .
$$
Considering the asymptotic behavior as $\xi\to\infty$, \eqref{eq:solution_X_series} implies that 
\beq \label{eq:X1Fg}
	X_1 = \sum_{s\in \mathcal S} \vec F (s) \vec g(s)^T.
\eeq

From the formulas of  $\vec f$ and $\vec g$,  we see that 
\beqq
	\partial_x \vec f(u) = -\frac{u}2 \sigma_3 \vec f(u) \quad \text{and}\quad \partial_x \vec g(u) = \frac{u}2 \sigma_3 \vec g(u).
\eeqq
Thus, 
\beqq
	\partial_x \bm H (u, v) =  \frac{ \partial_x \vec f(u)^T \vec g(v)+ \vec f(u)^T \partial_x \vec{g}(v)}{u-v} 
	= - \frac12 \vec f(u)^T \sigma_3 \vec g(v),
\eeqq
and using \eqref{eq:vecFGandresolv}, 
\beqq \begin{split}
	\partial_x \log \det(\mb I-\mb H) 
	&=-\Tr((I-\mb H)^{-1}\partial_x \mb H)
	= \frac12 \Tr (( \mb I- \mb H)^{-1} \vec f \otimes (\sigma_3\vec g)) \\
	&= \frac12 \Tr ( \vec F \otimes (\sigma_3\vec g))
	= \frac12 \sum_{s\in \mathcal S} \vec F(s)^T \sigma_3 \vec g(s).
\end{split} \eeqq
On the other hand, from \eqref{eq:X1Fg}, 
\beqq
	\Tr (X_1 \sigma_3) = \sum_{s\in \mathcal S} \Tr (\vec F (s) \vec g(s)^T \sigma_3) 
	= \sum_{s\in \mathcal S} \vec g(s)^T \sigma_3 \vec F (s) 
	= \sum_{s\in \mathcal S} \vec F(s)^T \sigma_3 \vec g (s) .
\eeqq
Equating the last two formulas, we obtain the lemma. 

\section{Proof of Proposition \ref{prop:qpreqs}}\label{sec:iiks_general}

We prove Proposition \ref{prop:qpreqs} focusing only on $X$. The proof for $Y$ is similar: we define $W=Ye^{-V\sigma_3/2}$ similarly as in \eqref{eq:XtoW} below and obtain the Lax equations in Lemma~\ref{lem:lax_equations} in the standard way for continuous Riemann-Hilbert problems. Having the Lax equations at hand, all the remaining arguments are algebraic and can be repeated for $Y$, line-by-line, as below, simply replacing $X$ and $X_k$ by $Y$ and $Y_k$. 

\subsection{The Lax systems}

Define the new matrix function
\beq \label{eq:XtoW}
	W(\xi)=X(\xi)e^{-\frac{V(\xi)}{2}\sigma_3},\quad \xi\in \C\setminus \mathcal S.
\eeq
Then, $W$ satisfies the following RHP: 
\begin{enumerate}[\bf RHP-W 1:]
\item $W:\C\setminus \mathcal S \to \C^{2\times 2}$ is analytic.
\item The points in $\mathcal S$ are simple poles of $W$ with residue matrix
$$
R_W(s)=e^{\frac{V(s)}{2}\sigma_3}R_X(s)e^{-\frac{V(s)}{2}\sigma_3} 
=\begin{pmatrix} 0 &  \frac{1}{s} e^{Q(s)} \chi_+(s) \\  \frac{1}{s} e^{-Q(s)}\chi_-(s) &0 \end{pmatrix} \quad \text{for $s\in \mathcal S$.}
$$
\item As $\xi\to \infty$,
$$
W(\xi)=(I+\Boh(\xi^{-1}))e^{-\frac{V(\xi)}{2}\sigma_3}.
$$
\end{enumerate}

The jump matrix does not depend on the parameters $\tau, \gamma, x$. This has the following consequence. 
 
\begin{lem}[Lax equations]\label{lem:lax_equations}
The matrix $W$ satisfies linear differential equations 
\begin{equation}\label{eq:lax_system}
	\partial_\gamma W(\xi)=A(\xi)W(\xi), \qquad
	\partial_\tau W(\xi)=B(\xi)W(\xi), \qquad 
	\partial_x W(\xi)=D(\xi)W(\xi),
\end{equation}
for some polynomial matrices 
\begin{equation*}
	A(\xi)=\xi^2A_2+\xi A_1+A_0,\qquad
	B(\xi)=\xi^3B_3+\xi^2B_2+\xi B_1+B_0,\qquad 
	D(\xi)=\xi D_1+D_0
\end{equation*}
where $A_2=-\frac14 \sigma_3$, $B_3=\frac16 \sigma_3$, and $D_1= -\frac12 \sigma_3$, and other matrices $A_i, B_i, D_i$ do not depend on $\xi$. 
\end{lem}

\begin{proof}
Since $\det X\equiv 1$, $\det W\equiv 1$. Hence, $W$ is invertible. 
Let $\partial$ be the partial derivative with respect either $\tau, \gamma$ or $x$. 
For $s\in \mathcal S$, Lemma \ref{lem:residue_inverse} implies that 
\beqq
	 Z(\xi):= W(\xi) \left( - \frac{R_W(s)}{\xi-s} + I \right) 
\eeqq
is analytic at $\xi=s$. 
Since $R_W(s)^2=0$, we have $\det Z(\xi)=1$, and hence $Z(\xi)^{-1}$ is also analytic at $\xi=s$. 
Since $R_W$ does not depend on the parameters, the derivative becomes
\beqq
	  \partial Z(\xi)= \partial W(\xi) \left( - \frac{R_W}{\xi-s} + I \right) .
\eeqq
Hence, 
\beqq
	(\partial W(\xi)) W(\xi)^{-1}=  \partial W(\xi) \left( - \frac{R_W}{\xi-s} + I \right)  \left( \frac{R_W}{\xi-s}+I \right) W(\xi)^{-1} = ( \partial Z(\xi)) Z(\xi)^{-1}
\eeqq
is analytic at $\xi=s$. Thus, $(\partial W(\xi)) W(\xi)^{-1}$ is an entire function.

The asymptotic condition, RHP-W 3, implies that
$$(\partial W(\xi)) W(\xi)^{-1}=\Boh(\xi^{-1}) -\frac{\partial V(\xi)}{2}  (I+\Boh(\xi^{-1})) \sigma_3  (I+\Boh(\xi^{-1}))$$
as $\xi\to \infty$.  
Since
$$
\partial_\gamma V(\xi)=\frac{\xi^2}{2},\quad \partial_\tau V(\xi)=-\frac{\xi^3}{3},\quad \partial_x V(\xi)=\xi, 
$$
the Liouville's theorem implies that $(\partial W(\xi)) W(\xi)^{-1}$ is a polynomial with the leading term $-\frac{\partial V(\xi)}{2}\sigma_3$. We thus obtain the result.
\end{proof}

We now insert \eqref{eq:XtoW} and the asymptotic series condition \eqref{eq:asymptotics_X} into the equations \eqref{eq:lax_system} to obtain sequences of relationships between $A_i, B_i, D_i$ and $X_k$. We find that
\begin{equation*} 
\begin{split}
	4A_1 &=-6B_2= 2D_0 = [\sigma_3, X_1] =: \cm_1 , \\
	4A_0 &=-6B_1= - [\sigma_3, X_1] X_1 + [\sigma_3, X_2] =: \cm_2 , \\
	 -6B_0 &= [\sigma_3,X_1]X_1^2 - [\sigma_3,X_2]X_1- [\sigma_3,X_1]X_2 + [\sigma_3,X_3] =: \cm_3
\end{split}
\end{equation*}
and for $k\ge 1$, 
\begin{align}
	4\partial_{\gamma}X_k &= -[\sigma_3,X_{k+2}] +\cm_1 X_{k+1} + \cm_2 X_k, \label{eq:A_system}\\
	-6 \partial_\tau X_k &= - [\sigma_3,X_{k+3}]+ \cm _1 X_{k+2}+ \cm_2 X_{k+1}+ \cm_3 X_k ,  \label{eq:B_system}\\
	2\partial_x X_k &= - [\sigma_3,X_{k+1}] + \cm_1 X_k  \label{eq:D_system}
\end{align}
where the bracket denotes the commutator $[P, Q]=PQ-QP$. 

Using \eqref{eq:D_system} when $k=1, 2$, we may remove the commutator terms in $\cm_2$ and $\cm_3$ and express them as 
\begin{equation}\label{eq:system_B_X}
	-\frac12 \cm_2=\partial_x X_1 \quad\text{and}\quad 
	-\frac12 \cm_3= (\partial_xX_1) X_1+ \partial_x X_2. 
\end{equation}

Recall the notation \eqref{def:notation_X1}. The diagonal entries of the equation \eqref{eq:D_system} with $k=1$ gives us the identity   
\begin{equation}\label{eq:qxpr}
	q_x=pr . 
\end{equation}
This proves \eqref{eq:2qxpr2} of Proposition \ref{prop:qpreqs}.

\subsection{Coupled nonlinear heat equations}

We find a partial differential equation for $X_1$ in the variables $x$ and $\gamma$ and prove \eqref{eq:prxxxtauPDE} of Proposition \ref{prop:qpreqs}. 

The $k=1$ case of equation \eqref{eq:A_system} becomes, using \eqref{eq:system_B_X} for $\cm_2$,  
\beqq
	4\partial_\gamma X_1 = -2 (\partial_x X_1) X_1 + [\sigma_3,X_1]X_{2} - [\sigma_3,X_{3}] . 
\eeqq
We remove $[\sigma_3, X_3]$ using the $k=2$ case of the equation \eqref{eq:D_system}. 
The equation becomes 
\begin{equation}\label{eq:xgammaX2X1}
	2 \partial_\gamma X_1 = - ( \partial_x X_1)X_1 + \partial_x X_2. 
\end{equation}
This equation contains $X_2$ which we remove as follows. 
The $x$-derivative of the $k=1$ case of \eqref{eq:D_system} is 
$$
	[\sigma_3, \partial_x X_2]= - 2\partial_{xx} X_1 + \partial_x ([\sigma_3, X_1] X_1). 
$$
We solve for $\partial_x X_2$ in \eqref{eq:xgammaX2X1} and insert it to the above equation and arrive at an equation involving only $X_1$. Using the general commutator identity $[A, BC]=[A, B]C+ B[A, C]$, the equation can be written as
\begin{equation} \label{eq:xgammaX12}
	2 [\sigma_3, \partial_\gamma X_1]= -2 \partial_{xx} X_1 + [\sigma_3, X_1]  (\partial_x X_1) -  ( \partial_x X_1) [\sigma_3, X_1] . 
\end{equation}

We insert the entries $q, p, r$ of $X_1$ in \eqref{def:notation_X1} into \eqref{eq:xgammaX12}. 
The diagonal entries simply give the $x$-derivative of \eqref{eq:qxpr}. 
The off-diagonal entries give two non-trivial equations, which after inserting  $q_x=pr$ become \eqref{eq:prxxxtauPDE} of Proposition \ref{prop:qpreqs}.

\subsection{Coupled mKdV equations}

We find a partial differential equation for $X_1$ in the variables $x$ and $\tau$ and prove \eqref{eq:prxxgammaPDE} of Proposition \ref{prop:qpreqs}. 
In this section, we set
\begin{equation*}
	Q:= [\sigma_3, X_1] \quad \text{and} \quad R:= [\sigma_3, X_2]=  - 2 \partial_x X_1 + Q X_1 .
\end{equation*}
where the second equality in $R$ follows from the $k=1$ case of \eqref{eq:D_system}. 
Consider the $k=1$ case of \eqref{eq:B_system}. 
The equation contains $[\sigma_3, X_4]$. 
We remove it using the $k=3$ case of \eqref{eq:D_system}. 
Further inserting the formula \eqref{eq:system_B_X}, the $k=1$ case of \eqref{eq:B_system} becomes 
\begin{equation*}
	\partial_x X_3= ( \partial_x X_1)X_2 + (\partial_x X_2) X_1- (\partial_x X_1)X_1^2 -3 \partial_\tau X_1 .
\end{equation*}
The $k=2$ case of \eqref{eq:D_system} is 
\begin{equation}\label{eq:X3X2Qcomtmp}
	[\sigma_3, X_3]= - 2 \partial_x X_2 + Q X_2. 
\end{equation}
Taking the $x$-derivative of this equation, and then inserting the formula of $\partial_x X_3$ above, we find 
$$
	[\sigma_3, ( \partial_x X_1)X_2 + (\partial_x X_2) X_1] - [\sigma_3, (\partial_x X_1)X_1^2] -3 \partial_\tau Q
	= -2 \partial_{xx} X_2 + \partial_x (QX_2).
$$
We use the general commutator identity $[A, BC]=[A, B]C+ B[A, C]$ for the first term on the left-hand side and move the term $\partial_x (QX_2)$ to the left-hand side of the equation. 
Since $\Tr X_1=0$, the matrix $X_1^2$ is diagonal, and hence $[\sigma_3, X_1^2]=0$, implying that $[\sigma_3, (\partial_x X_1)X_1^2]= [\sigma_3, \partial_x X_1]X_1^2=(\partial_x Q )X_1^2$. 
Hence, recalling the notations $Q$ and $R$, the above equation can be written as 
\begin{equation*}
\begin{split}
	(\partial_x X_1) R + ( \partial_x R)  X_1+ (\partial_x X_2) Q - Q (\partial_x X_2) 
	 -  (\partial_x Q )X_1^2 -3  \partial_\tau Q &
	= -2 \partial_{xx} X_2 . 
\end{split} \end{equation*}
Taking the commutator with $\sigma_3$, 
\begin{equation}\label{eq:cmtp1}
\begin{split}
	[\sigma_3, (\partial_x X_1) R + ( \partial_x R)  X_1]+ [\sigma_3, (\partial_x X_2) Q - Q (\partial_x X_2)] 
	 - [\sigma_3, (\partial_x Q )X_1^2]  - 3 [\sigma_3,  \partial_\tau Q ] 
	= -2 \partial_{xx} R . 
\end{split} \end{equation}
We express $[\sigma_3, (\partial_x X_2) Q - Q (\partial_x X_2)]$ in terms of $ Q$ and $R$ using the next lemma. 

\begin{lem}
We have 
$$
	2[\sigma_3, (\partial_x X_2) Q] =  - QRQ 
	\qquad\text{and}\qquad 
	2[\sigma_3,Q (\partial_x X_2) ]   = Q^2R. 
$$
\end{lem}

\begin{proof}
The equation \eqref{eq:X3X2Qcomtmp} gives $2 \partial_x X_2 =  Q X_2- [\sigma_3, X_3]$. Observe that since $[\sigma_3, X_3]$ and $Q=[\sigma_3, X_1]$ have zero diagonal entries, $[\sigma_3, X_3] Q$ is a diagonal matrix, and hence $[\sigma_3, [\sigma_3, X_3] Q]=0$. Similarly we have $[\sigma_3, Q[\sigma_3, X_3]]=0$. 
Therefore, 
$$
	2[\sigma_3, (\partial_x X_2) Q] = [\sigma_3, Q X_2 Q- [\sigma_3, X_3] Q] = [\sigma_3, QX_2 Q]
$$
and 
$$
	2[\sigma_3,Q (\partial_x X_2) ] = [\sigma_3, Q^2 X_2 - Q[\sigma_3, X_3]] = [\sigma_3, Q^2X_2 ].
$$
Now, since $Q=[\sigma_3, X_1]$, we have $\sigma_3 Q+Q \sigma_3=0$.
Using this identity, we can write
$$
	[\sigma_3, QX_2 Q] = -Q[\sigma_3, X_2] Q = - QRQ
\qquad \text{and}\qquad 
	[\sigma_3, Q^2X_2 ] = Q^2[\sigma_3, X_2]  = Q^2R. 
$$
\end{proof}

The equation \eqref{eq:cmtp1} thus becomes an equation only of $X_1$: 
\begin{equation}\label{eq:X1xxxtauPDE}
\begin{split}
	2[\sigma_3, (\partial_x X_1) R + ( \partial_x R)  X_1]- Q(RQ+QR)
	 - 2[\sigma_3, (\partial_x Q )X_1^2] - 6[\sigma_3, \partial_\tau Q ] 
	&= -4 \partial_{xx} R .
\end{split} \end{equation}
We now insert the entries \eqref{def:notation_X1} of $X_1$. The entries are  
$$
	X_1=\begin{bmatrix} q & p \\ r & -q \end{bmatrix} \quad \text{and}\quad  
	Q=\begin{bmatrix} 0 & 2p \\ -2r & 0 \end{bmatrix} \quad \text{and}\quad 
	R= \begin{bmatrix} 0 & -2p_x - 2pq \\ -2 r_x - 2rq & 0 \end{bmatrix}
$$
where we used \eqref{eq:qxpr} for the diagonal entries of $R$.  It is useful to observe that 
$$
	RQ+QR= -4(r_xp-p_xr)I \quad\text{and}\quad X_1^2= (q^2+pr) I.
$$
Inserting these into  \eqref{eq:X1xxxtauPDE}, the diagonal entries of the equation give trivial identities. 
The off-diagonal entries, on the other hand, imply, using $q_x=pr$, \eqref{eq:prxxgammaPDE} of Proposition \ref{prop:qpreqs}. 


\def\cydot{\leavevmode\raise.4ex\hbox{.}}

\end{document}